\documentclass[12pt,a4paper]{article}

\usepackage{hyperref}

\usepackage[english]{babel}
\usepackage{fancyhdr}
\usepackage{fullpage}
\usepackage[dvips]{graphicx}
\usepackage[T1]{fontenc}
\usepackage{fouriernc}
\usepackage[latin1]{inputenc}
\usepackage{amsmath}
\usepackage{amsfonts}
\usepackage{amssymb}
\usepackage{mathrsfs}
\usepackage{amsthm}
\usepackage{latexsym}
\usepackage{array}
\usepackage{mathrsfs}
\usepackage{amsxtra}
\usepackage{amscd}
\usepackage{mathtools}
\usepackage{verbatim}
\usepackage{enumerate}
\usepackage[title]{appendix}
\usepackage[usenames]{color}
\definecolor{beige}{rgb}{0.96, 0.96, 0.86}
\definecolor{airforceblue}{rgb}{0.36, 0.54, 0.66}
\definecolor{antiquefuchsia}{rgb}{0.57, 0.36, 0.51}
\definecolor{awesome}{rgb}{1.0, 0.13, 0.32}

\usepackage{tikz} 
\usetikzlibrary{arrows,shapes,matrix}

\usepackage{xcolor}
\hypersetup{
    colorlinks,
    linkcolor={green!60!blue},
    citecolor={yellow!30!red},
    urlcolor={green!60!blue}
}

\newcommand{\wt}{\widetilde}

\newcommand{\id}{\operatorname{id}}
\newcommand{\lpaths}[1]{\mathcal{L}^{#1}_{\mathcal{P}_T}(\paths)}
\newcommand*{\Lpaths}{\lpaths{p}}

\newcommand*{\Bb}[1]{B_b([0,T],#1)}

\newcommand*{\Cb}[1]{C([0,T],#1)}

\newcommand*{\Gat}[4]{\mathcal{G}^{#4}(#1,#2;#3)}
\newcommand*{\Gatot}[3]{\mathcal{G}^{#3}(#1,#2)}

\newcommand*{\paths}{\mathbb{S}}

\newcommand*{\medcap}{\mathbin{\scalebox{1.5}{\ensuremath{\cap}}}}

\newcommand{\sconv}{\stackrel{\scriptscriptstyle dW}{*_t}}
\newcommand{\conv}{\stackrel{}{*_t}}
\newcommand{\sconvt}[1]{\stackrel{\scriptscriptstyle dW}{*_{#1}}}
\newcommand{\convt}[1]{\stackrel{}{*_{#1}}}

\numberwithin{equation}{section}

\theoremstyle{plain}
\newtheorem{theorem}{Theorem}[section]
\newtheorem{corollary}[theorem]{Corollary}
\newtheorem{proposition}[theorem]{Proposition}
\newtheorem{lemma}[theorem]{Lemma}
\newtheorem{definition}[theorem]{Definition}

\newtheorem{assumption}[theorem]{Assumption}

\theoremstyle{definition}
\newtheorem{remark}[theorem]{Remark}

\newcommand{\lin}{L}
\newcommand{\mc}{\mathcal}
\renewcommand{\epsilon}{\varepsilon}
\renewcommand{\phi}{\varphi}

\def\qed{{\hfill\hbox{\enspace${ \blacksquare}$}} \smallskip}

\newcommand{\e}{e}

  \def\Swiech
{{\accent1 S}wie{\hbox{\kern -0.31em\lower
0.77ex \hbox{$\textfont1=\scriptfont1 \lhook$}}}ch}

\def\SWIECH
{{\accent1 S}WIE{\hbox{\kern -0.21em\lower
0.77ex\hbox{$\textfont1=\scriptfont1
\lhook$}}}CH}

\usepackage{dsfont} 

\def\b*{\begin{eqnarray*}}
\def\e*{\end{eqnarray*}}

\def \0{\mathbf{0}}

\theoremstyle{plain}

\theoremstyle{definition}

\theoremstyle{plain}
 \newtheorem{Theorem}{Theorem}[section]

\theoremstyle{definition}
 \newtheorem{Remark}[Theorem]{Remark}

 \theoremstyle{definition}

\def\<{\left\langle }
\def\>{\right\rangle }

\title{Path-dependent SDEs in Hilbert spaces}
\date{}
\author{Mauro Rosestolato\thanks{CMAP, \'Ecole Polytechnique, Paris, France,
e-mail: \texttt{mauro.rosestolato@polytechnique.edu}. 
This research has been
partially supported by
the ERC
321111 Rofirm.
}
}

\begin{document}

\maketitle

\begin{abstract} 
We study path-dependent SDEs in Hilbert spaces.
By using methods based on contractions in Banach spaces, we prove existence and uniqueness of mild solutions, continuity of mild solutions with respect to perturbations of all the data of the system, G\^ateaux differentiability of generic order $n$ of mild solutions with respect to the starting point, continuity  of the G\^ateaux derivatives with respect to all the data.
The analysis is performed for generic spaces of paths that do not necessarily coincide with the space of continuous functions.
\end{abstract}

\vspace{10pt}

\vspace{10pt} 

\section{Introduction}




In this paper  we deal with mild solutions to path-dependent SDEs
evolving in a separable Hilbert space $H$,
 of the form
\begin{equation}
  \label{2017-05-29:01}
  \begin{dcases}
    dX_t=(AX_t+b((\cdot,t),X))dt+ \sigma((\cdot,s),X)dW_s&\forall s\in(t,T]\\
    X_s=Y_s& s\in[0,t],
  \end{dcases}
\end{equation}
where $t\in[0,T)$, $Y$ is a $H$-valued
adapted process defined on a filtered probability space
$(\Omega,\mathcal{F},\{\mathcal{F}_t\}_{t\in[0,T]},\mathbb{P})$, $W$ is a cylindrical Wiener process
on $(\Omega,\mathcal{F},\{\mathcal{F}_t\}_{t\in[0,T]},\mathbb{P})$
 taking values in a separable Hilbert space $U$, 
 $b((\omega,s),X)$ is a $H$-valued random variable depending on $\omega\in \Omega$, on the time $s$, and on the path $X$, $\sigma((\omega,s),X)$ is a $L_2(U,H)$-valued random variable depending on $\omega\in\Omega$, on  the time $s$, and on the path $X$, and $A$ is the generator of a $C_0$-semigroup $S$ on $H$.
By using methods based on implicit functions associated to contractions in Banach spaces,
we study continuity of the mild solution $X^{t,Y}$
of \eqref{2017-05-29:01} with respect to $t,Y,A,b,\sigma$
under standard Lipschitz conditions on $b,\sigma$,
 G\^ateaux differentiability of generic order $n\geq 1$ of $X^{t,Y}$ with respect to $Y$ under G\^ateaux differentiability assumptions on $b,\sigma$, and continuity
with respect to $t,Y,A,b,\sigma$
 of the G\^ateaux differentials $ \partial ^n_{Y}X^{t,Y}$.




Path-dependent SDEs 
in finite dimensional spaces
are studied in \cite{Mohammed1984}.
The standard reference for SDEs in Hilbert spaces is \cite{DaPrato2014}.
More generally, in addition to SDEs in Hilbert spaces,
also the case of 
path-dependent SDEs in Hilbert spaces is  considered in \cite[Ch.\ 3]{Gawarecki2011}, but for the path-dependent case  the study is there limited mainly to existence and uniqueness of mild solutions.
Our framework generalize the latter one by
weakening the Lipschitz conditions on the coefficients,
 by letting the starting process $Y$ belong to a generic space of paths contained in $\Bb{H}$ (\footnote{$\Bb{H}$ denotes the space of bounded Borel functions $[0,T]\rightarrow H$.}) obeying few conditions, but not necessarily assumed to be $\Cb{H}$, and by 
providing results on 
 differentiability 
with respect to the initial datum 
and on continuity with respect to all the data.

In the literature
on mild solutions to SDEs in Hilbert spaces,
differentiability with respect to the initial
datum is always proved 
only up to order $n=2$,
 in the sense of G\^ateaux 
(\cite{DaPrato2004,DaPrato2014})
 or Fre\'chet (\cite{Gawarecki2011,Knoche2001}.
In \cite[Theorem 7.3.6]{DaPrato2004}
the
 case $n> 2$  is stated but not proved.
%
There are no available  results regarding differentiability with respect to the initial condition of mild solutions to SDEs of the type \eqref{2017-05-29:01}.
One of the contributions of the present work
is to fill this gap in the literature,
by extending 
 to a generic order $n$, in the G\^ateaux sense, and to the path-dependent case the results so far available.

In case \eqref{2017-05-29:01} is not path-dependent,
the continuity of $X^{t,Y}$, $ \partial _YX^{t,Y}$, and $ \partial ^2_{Y}X^{t,Y}$,
separately with respect to $t,Y$ and $A,b,\sigma$, is
considered and used in \cite[Ch.\ 7]{DaPrato2004}.
We extend
these previous results
 to the path-dependent case and to G\^ateaux derivatives $ \partial ^n_YX^{t,Y}$ of generic order $n$, proving joint continuity with respect to all the data $t,Y,A,b,\sigma$.

Similarly as in the cited literature, we obtain our results for mild solutions (differentiability and continuity with respect to the data) starting from analogous results for implicit functions associated to Banach space-valued contracting maps. 
Because of that, the first part of the 
paper is entirely devoted to study parametric contractions in Banach spaces and regularity of the associated implicit functions.
In this respect, 
 regarding G\^ateaux differentiability
of implicit functions associated to parametric contractions and continuity of the derivatives under perturbation of the data, 
we  prove a general result, for a  generic order $n$ of differentiability, extending
 the results in \cite{Cerrai2001,DaPrato2004,DaPrato2014}, that were limited to the case $n=2$.

In a unified framework, our work provides a collection of results for mild solutions to path-dependent SDEs which are very general, within the standard   case of Lipschitz-type assumptions on the coefficients, a useful toolbox for
starting dealing with path-dependent stochastic analysis in Hilbert spaces.
For example, 
the so called ``vertical derivative'' 
 in the finite dimensional functional It\=o calculus (\cite{Cont2013,Dupire2009}) of functionals like $F(t,\mathbf{x})=\mathbb{E}[\varphi(X^{t,\mathbf{x}})]$, where  $\varphi$ is a functional on the space $\mathbb{D}$  of c\`adl\`ag functions and
$\mathbf{x}\in \mathbb{D}$,
is easily obtained starting from the partial derivative of $X^{t,\mathbf{x}}$ with respect to a step function, which can be treated in our setting by choosing $\mathbb{D}$ as
space of paths
 (we refer to Remark~\ref{2017-05-01:01} for further details).
 Another field in which the tools here provided can be employed is the study of stochastic representations of classical solutions to path dependent Kolmogorov equations,
where second order derivatives are required.
Furthermore, the continuity of the mild solution and of its derivatives with respect to all the data, including the coefficients, can be used e.g.\ when merely continuous Lipschitz coefficients need to be approximated by smoothed out coefficients, which is in general helpful when dealing with Kolmogorov equations in Hilbert spaces (path- or non-path-dependent) for which notions other than classical solutions are considered, as strong-viscosity solutions
(\cite{Cosso2014a,Cosso2014b}) or strong solutions (\cite{Cerrai2001}).

\bigskip
The contents of the paper are organized as follows. 
First, in Section~\ref{sec:recalls-dif-banach}, we
recall some notions regarding strongly continuous G\^ateaux differentiability 
and some basic results for contractions in Banach spaces.
Then 
we provide the first main result (Theorem~\ref{teo:derivabilita.punto.fisso}): the strongly continuous  G\^ateaux differentiability up to a generic order $n$
of fixed-point maps associated to parametric contractions which are differentiable only with respect to some subspaces.
We conclude the section with a
result regarding
the continuity
 of 
the G\^ateaux differentials of the implicit function with respect to the data
(Proposition~\ref{propp:2012-05-23-aa}).

In Section~\ref{2017-04-25:24} we consider path-dependent SDEs.
After a standard existence and uniqueness result (Theorem~\ref{2016-04-13:00}), we move to study G\^ateaux differentiability with respect to the initial datum up to order $n$ of mild solutions, in Theorem~\ref{2016-04-05:05}, which is the other main result and justifies the study made in Section~\ref{sec:recalls-dif-banach}.
We conclude  with 
Theorem~\ref{2016-04-22:10}, which concerns the
continuity of the G\^ateaux differentials 
with respect 
to all the data of the
system (coefficients, initial time, initial condition).

\section{Preliminaries}
\label{sec:recalls-dif-banach}

In this section we recall the notions and develop the tools that we will apply to study path-dependent SDEs in Section~\ref{2017-04-25:24}.
We focus on strongly continuous G\^ateaux differentiability of fixed-point maps associated to parametric contractions in Banach spaces.

\subsection{Strongly continuous G\^ateux differentials}
\label{2017-04-24:00}

We begin by recalling the basic definitions regarding G\^ateaux differentials, mainly following \cite{Flett1980}.
Then
we will define the space of strongly continuously G\^ateaux differentiable functions, that will be the reference spaces in the following sections.

\medskip
If $X$, $Y$ are topological vector spaces, $U\subset X$ is a set,
$f\colon U\to Y$ is a function, $u\in U$, $x\in X$ is such that $[u-\epsilon x,u+\epsilon x]\subset U$~({\footnote{If $x,x'\in X$, the segment $[x,x']$ is the set $\{\zeta x+(1-\zeta)x'|\zeta\in[0,1]\}$.}) for some $\epsilon>0$, 
the
directional derivative of $f$ at $u$ for the increment $x$ is the limit  
$$
  \partial _x f(u)\coloneqq \lim_{t\rightarrow 0}\frac{f(u+tx)-f(u)}{t}\label{2016-02-22:00}
$$
whenever it exists.  
Also in the case in which the directional
derivative $ \partial _xf(u)$ is defined for all $x\in X$, it need not be linear.  

Higher order directional derivatives are defined recursively.
  For $n\geq 1$, $u\in U$, 
the
$n$th-order directional derivative
$ \partial ^n_{x_1\ldots x_{n}}f(u)$ at $u$ for the increments
$x_1,\ldots,x_n\in X$
is the directional derivative
of
$ \partial ^{n-1}_{x_1\ldots x_{n-1}}f$ at $u$ for the increment
$x_n$ (notice that this implies, by definition, the existence of $\partial ^n_{x_1\ldots x_{n-1}}f(u')$ for $u'$ in some neighborhood of $u'$ in  $U\medcap (u+\mathbb{R}x_n)$)



If $Y$ is locally convex, we denote by
$\lin_s(X,Y)$\label{2016-02-22:01}
the
space $\lin(X,Y)$ endowed with the coarsest topology wich makes
continuous the linear functions of the form
$$
\lin(X,Y)\to Y,\ \Lambda\mapsto \Lambda(x),
$$
for all $x\in X$.
Then $\lin_s(X,Y)$ is a locally convex space.

Let $X_0$ be a topological vector space continuously embedded into $X$.
If $u\in
U$, if $ \partial  _x f(u)$ exists for all $x\in X_0$ and $ X_0\rightarrow Y,\ x \mapsto   \partial  _x
f(u)$, belongs to $\lin(X_0,Y)$, then  $f$
is said to be \emph{G\^ateaux
  differentiable at $u$ with respect to  $X_0$}
 and the map $X_0\rightarrow Y,\ x \mapsto  \partial_x f(u)$, is the \emph{G\^ateaux differential of $f$ at $u$
  with respect to  $X_0$}. In this case, we denote the G\^ateaux differential of
$f$ at $u$ by $ \partial _{X_0}f(u)$ and its evaluation $ \partial  _x f(u)$ by
$ \partial _{X_0}f(u).x$.{\label{2016-02-22:02}} 
 If $ \partial _{X_0}f(u)$ exists for all $u\in U$, then we
say that $f$ is \emph{G\^ateaux differentiable with respect to $X_0$}, or, in case $X_0=X$, we just say that $f$ is \emph{G\^ateaux differentiable} and we use the notation $ \partial f(u)$ in place of $ \partial _Xf(u)$.

A function  $f\colon U\to Y$ is said to be \emph{strongly continuously
  G\^ateaux differentiable with respect to $X_0$} 
if it is 
G\^ateaux
differentiable with respect to $X_0$ and  
$$
U\to \lin_s(X_0,Y),\  u\mapsto  \partial _{X_0}f(u)
$$
is continuous.  If $n>1$, we say that $f$ is \emph{strongly
  continuously G\^ateaux differentiable up to order $n$ with respect to $X_0$}
if it is strongly continuously G\^ateaux differentiable up to order
$n-1$ with respect to $X_0$ and  
$$
 \partial _{X_0}^{n-1}f\colon U\to \overbrace{\lin_s(X_0,\lin_s(X_0,\cdots
  \lin_s(X_0,Y)}^{n-1\textrm{ times $L_s$}}\cdots ))
$$
exists and is strongly continuously G\^ateaux differentiable with respect to $X_0$. In
this case,  we denote $ \partial _{X_0}^nf\coloneqq \partial _{X_0} \partial _{X_0}^{n-1}f${\label{2016-02-22:03}} and $ \partial  ^nf\coloneqq \partial  \partial ^{n-1}f$.



\bigskip
Let $X,X_0$ be topological vector spaces, with $X_0$ continuously embedded into $X$,  let $U$ be an open subset of $X$, and let $Y$ be a locally convex space.

\smallskip
  We denote by $\Gat{U}{Y}{X_0}{n}$ the space of 
functions $f\colon U\rightarrow Y$ 
which are continuous and
strongly continuously G\^ateaux differentiable
up to order $n$
 with respect to $X_0$.
In case $X_0=X$, we use the notation $\Gatot{U}{Y}{n}$ instead of
$\Gat{U}{Y}{X}{n}$.
\smallskip



Let
$\lin^{(n)}_s(X_0^n,Y)$
\label{2016-02-22:05}
 be the vector space of $n$-linear functions
from $X_0^n$ into $Y$ which are continuous with respect to each variable
separately, endowed with the
coarsest vector topology making continuous all the linear
functions  of the form
$$
\lin^{(n)}_s(X_0^n,Y)\to Y,\  \Lambda\to \Lambda(x_1,\ldots,x_n)
$$
for $x_1,\ldots,x_n\in X_0$. 
Then $\lin^n_s(X_0^n,Y)$ is a locally convex space.
Trough the canonical identification 
(as topological vector spaces)
$$
\overbrace{\lin_s(X_0,\lin_s(X_0,\cdots \lin_s(X_0,Y)}^{n\textrm{    times $L_s$}}\cdots ))\cong \lin^{(n)}_s(X_0^n,Y),
$$
we can consider
$ \partial _{X_0 }^nf$ as taking values in 
$\lin_s^{(n)}(X_0^n,Y)$, whenever $f\in \Gat{U}{Y}{X_0}{n}$.

\medskip

If $X_0$, $X$, $Y$ are normed spaces, $U$ is an open subset of $X$,
  $ \partial  _x f(u)$ exists for all $u\in U$, $x\in X_0$,
$\partial_x f(u)$
is continuous with respect to $u$, for all $x\in X_0$,
 then
 $\partial _xf(u)$ is linear in  $x$
(see \cite[Lemma 4.1.5]{Flett1980}). 

The following proposition 
is a characterization
for
 the continuity conditions on the directional derivatives of a function $f\in \Gat{U}{Y}{X_0}{n}$, when $X_0,X,Y$ are normed spaces.

\begin{proposition}\label{prop:2014-07-30:01}
  Let $n\geq 1$, let $X_0,X,Y$ be normed spaces,
 with $X_0$ continuously embedded into $X$, and let $U$
be an open subset of $X$.
  Then $f\in \Gat{U}{Y}{X_0}{n}$
if and only if $f$ is continuous,
the directional derivatives
  $\partial _{x_1\ldots x_j}^jf(u)$ exist for all $u\in U$,
  $x_1,\ldots,x_j\in X_0$, $j=1,\ldots,n$, and the functions  
  \begin{equation}
    \label{eq:2017-05-14:00}
     U\times X_0^j\to Y,\  (u,x_1,\ldots,x_j)\mapsto
\partial ^j_{x_1\ldots x_j}f(u)
\end{equation}
are separately continuous in each variable.
 In this case,
 \begin{equation}
   \label{eq:2016-02-22:06}
    \partial _{X_0}^jf(u).(x_1,\ldots,x_j)=\partial _{x_1\ldots x_j}^jf(u) \qquad  \forall u\in U,
\ \forall x_1,\ldots,x_j\in X_0,
 \ j=1,\ldots,n.
\end{equation}
\end{proposition}

\begin{proof}
  Suppose that 
the derivatives 
$\partial^j_{x_1\ldots x_j}f(u)$ exists for all $u\in U$, $x_1,\ldots,x_j\in X_0$, $j=1,\ldots,n$, separately continuous in $u,x_1,\ldots, x_j$.
We want to show that $f\in \Gat{U}{Y}{X_0}{n}$.

We proceed by induction on $n$.
Let $n=1$.
 Since $\partial _xf(u)$ is continuous in $u$, for all $x\in X_0$, we have
  that $X_0\rightarrow Y,\  x\mapsto \partial _xf(u)$ is linear 
(\cite[Lemma 4.1.5]{Flett1980}). 
By assumption, it is also continuous.
Hence $x \mapsto \partial_x f(u)\in L(X_0,Y)$ for all $u\in U$.
This shows the existence of $ \partial _{X_0}f$.
  The continuity of $U\to \lin_s(X_0,Y),\  u\mapsto
   \partial _{X_0}  f(u)$, comes from the separate continuity 
of \eqref{eq:2017-05-14:00} and
 from the definition of 
the locally convex topology on $L_s(X_0,Y)$.
This shows that $f\in \Gat{U}{Y}{X_0}{1}$.


Let now $n>1$.
By inductive hypothesis, we  may assume that $f\in \Gat{U}{Y}{X_0}{n-1}$ and 
  $$
 \partial _{X_0 }^{j}f(u).(x_1,\ldots,x_j)=\partial _{x_1\ldots x_j}^jf(u)\qquad \forall u\in U,
\ \forall j=1,\ldots,n-1,
\ \forall (x_1,\ldots, x_j)\in X_0^j.
$$ 
Let $x_n\in X_0$. The limit
\begin{equation}
  \label{eq:2017-05-14:01}
  \lim_{t\rightarrow 0}\frac{ \partial ^{n-1}_{X_0 }f(u+tx_n)- \partial ^{n-1}_{X_0 }f(u)}{t}=\Lambda
\end{equation}
exists in $\lin_s^{(n-1)}(X_0^{n-1},Y)$
 if and only if $\Lambda\in \lin_s^{(n-1)}(X_0^{n-1},Y)$ and, for all
$x_1,\ldots,x_{n-1}\in X_0$, the limit
\begin{equation}
  \label{eq:2017-05-14:02}
  \lim_{t\rightarrow 0}\frac{\partial ^{n-1}_{x_1\ldots
    x_{n-1}}f(u+tx_n)-\partial ^{n-1}_{x_1\ldots x_{n-1}}f(u)}{t}=\Lambda(x_1,\ldots,x_{n-1})
\end{equation}
holds in $Y$.
By assumption, the limit \eqref{eq:2017-05-14:02} is equal to 
$ \partial ^n_{x_1\ldots x_{n-1}x_n}f(u)$, for all $x_1,\ldots,x_{n-1}$.
Since, by assumption, $ \partial ^n_{x_1\ldots x_{n-1}x_n}f(u)$ is separately continuous in $u,x_1,\ldots,x_{n-1},x_n$, we have that the limit \eqref{eq:2017-05-14:01} exists in $\lin_s^{(n-1)}(X_0^{n-1},Y)$ and is given by
$$
 \partial _{x_n} \partial^{n-1} _{X_0}f(u).(x_1,\ldots,x_{n-1})=\Lambda(x_1,\ldots,x_{n-1}) =\partial ^n_{x_1\ldots x_{n-1}x_n}f(u)\qquad \forall x_1,\ldots,x_{n-1}\in X_0.
$$
Since $u$ and $x_n$ were arbitrary, we have proved that
$ \partial _{x_n} \partial ^{n-1}_{X_0}f(u)$  exists for all $u$, $x_n$.
Moreover,  for all $x_1,\ldots,x_n\in X_0$, the
function  
$$
U\to Y,\  u\mapsto 
\partial_{x_n} \partial ^{n-1}_{X_0}f(u).(x_1,\ldots,x_{n-1})
=\partial_{x_n}\partial ^{n}_{x_1\ldots x_{n-1}}f(u)
$$
is
continuous, by 
separate continuity of 
\eqref{eq:2017-05-14:00}.
Then 
$\partial ^{n}_{x_1\ldots
  x_{n-1}x_n}f(u)$ is linear in $x_n$.
The continuity of
\begin{equation}
  \label{eq:2016-04-20:01}
  X_0\rightarrow \lin_s^{(n-1)}(X_0^{n-1},Y),\ 
 x\mapsto\partial _x \partial _{X_0 }^{n-1}f(u)
\end{equation}
comes from the 
continuity of $\partial ^{n}_{x_1\ldots x_{n-1}x}f(u)$ in each variable, separately.
Hence \eqref{eq:2016-04-20:01}  
belongs to $L_s(X_0,L_s^{n-1}(X_0^{n-1},Y))$
for all $u\in U$.
 This shows that 
$ \partial _{X_0 }^{n-1}f$
is G\^ateaux differentiable with respect to $X_0$
and that
\begin{equation*}
     \partial _{X_0}^nf(u).(x_1,\ldots,x_n)
= 
\partial^n_{x_1\ldots x_n}f(u)
\qquad  \forall u\in U,
\ \forall x_1,\ldots,x_n\in X_0,
\end{equation*}
and  shows also the continuity of
$$
U\rightarrow L^{(n)}_s(X_0^n,Y),\ u \mapsto   \partial _{X_0}^nf(u),
$$
due to 
 the continuity of
the derivatives of $f$, separately in each direction.
Then we have proved that
$f\in \Gat{U}{Y}{X_0}{n}$ and that
\eqref{eq:2016-02-22:06} holds.

Now suppose that
$f\in \Gat{U}{Y}{X_0}{n}$.
By the very definition
of $ \partial _{X_0 }f$,
$\partial _xf(u)$ exists
for all $x\in
X_0$ and $u\in U$, 
it is
separately continuous in $u,x$, and coincides with $ \partial _{X_0 }f(u).x$.  By induction, assume that 
$\partial^{n-1}_{x_1\ldots x_{n-1}}f(u)$ exists
and that
\begin{equation}
  \label{eq:2016-02-22:07}
   \partial ^{n-1}_{X_0}f(u).(x_1,\ldots,x_{n-1})=\partial^{n-1}_{x_1\ldots x_{n-1}}f(u)\qquad
\forall u\in U,\ \forall x_1,\ldots,x_{n-1}\in X_0.
\end{equation}
Since
$ \partial _{X_0 }^{n-1}f(u)$ is G\^ateaux differentiable, 
the directional derivative $\partial _{x_n} \partial ^{n-1}_{X_0 }f(u)$ exists.
Hence, by 
\eqref{eq:2016-02-22:07},
the derivative
$\partial ^n_{x_1\ldots x_{n-1}x_n}f(u)$
exists for all
$x_1,\ldots,x_{n-1},x_n\in X_0$.  The continuity of
$\partial ^n_{x_1\ldots x_{n-1}x_n}f(u)$ with respect to $u$ comes from
the continuity of $ \partial _{X_0 }^nf$. The continuity of
$\partial ^n_{x_1\ldots x_{j}\ldots x_n}f(u)$ with respect to $x_j$
comes from the fact that, for all $x_{j+1},\ldots,x_n\in X_0$, $u\in U$,
$$
X_0^j\rightarrow Y,\ 
(x_1',\ldots,x_j')  \mapsto  \partial ^n_{X_0 }f(u).(
x_1',\ldots,x_j',x_{j+1},\ldots,x_{n})
$$
belongs to $\lin^{(j)}_s(X_0^j,Y)$.
\end{proof}

\begin{remark}\label{rem:2013.09.06}
If $X_0$ is Banach, $X$ is normed, $Y$ is locally
convex, and $f\in \Gat{X}{Y}{X_0}{n}$,
then, by 
Proposition~\ref{prop:2014-07-30:01}
and the Banach-Steinhaus theorem, 
if follows that 
the map
$$U\times X_0^n\to Y,\  (u,x_1,\ldots,x_n)\mapsto  \partial _{X_0 }^nf(u).(x_1,\ldots,x_n)$$ is continuous, jointly 
in $u,x_1,\ldots,x_n$.
\end{remark}

\begin{remark}\label{rem:schwarz}
Under the assumption of Proposition
\ref{prop:2014-07-30:01},
by Schwarz' theorem,
  $$
  y^*( \partial ^2_{zw} f(u))=\partial ^2_{zw}(y^* f)(u)=
  \partial ^2_{wz}(y^* f)(u)= y^* (\partial ^2_{wz}
    f(u)),\  \forall u\in U,\  \forall w,x\in X_0,\ \forall y^*\in Y^*.
  $$
Hence $\partial ^2_{wz} f=\partial ^2_{wz}f$ for all $w,z\in X_0$.
\end{remark}


\subsubsection{Chain rule}

In this subsection, we show the
 classical Fa\`a di Bruno formula, together with a corresponding stability result, for 
derivatives
of 
order $n\geq 1$
 of compositions of strongly continuously G\^ateaux differentiable functions. 
We will use
this formula 
in order to prove the main results of Section \ref{2016-02-22:08} (Theorem \ref{teo:derivabilita.punto.fisso} and Proposition \ref{propp:2012-05-23-aa}).

In
\cite{Clark2013}, 
a version of
Proposition \ref{teo:faa-di-bruno} is provided for the case of ``chain differentials''.
We could prove that the strongly continuously G\^ateaux differentiable functions that we consider satisfy the assumptions of \cite[Theorem 2]{Clark2013}.
This would provide Proposition \ref{teo:faa-di-bruno} as a corollary of \cite[Theorem 2]{Clark2013}.
Since the direct proof of 
Proposition \ref{teo:faa-di-bruno} is quite concise, we prefer to report it, and avoid  introducing other notions of differential.
Besides, we give the related stability results.

\begin{lemma}\label{lemm:formula_derivate_composte}
  Let $k\geq 0$, let $X_1$, $X_2$, $X_3$ be Banach spaces, let $U$ be an open subset of $X_1$, and let $X_0$ be a subspace of $X_1$.
Let
  $f,f_1,\ldots,f_k\colon U\to X_2$ be functions
having
directional
  derivatives
$ \partial _xf,
\partial _xf_1,\ldots,
\partial _xf_k$
 with respect to all  $x\in X_0$
   and let $g\in \Gatot{X_2}{X_3}{k+1}$.
  Then
  \begin{equation}\label{eq:gatogamma}
    \gamma\colon U\to X_3,\  u\mapsto \partial^k_{f_1(u)\ldots f_k(u)}g\left(f(u)\right)
  \end{equation}
  has directional derivatives $ \partial _x\gamma$ with respect to all   $x\in X_0$ and
  \begin{equation}\label{eq:formula_derivate_composte}
\hskip-1pt    \partial _x\gamma(u)=
    \partial^{k+1}_{\partial _xf(u)f_1(u)\ldots f_k(u)}g\left(f(u)\right)+
    \sum_{i=1}^k \partial^k_{f_1(u)\ldots \partial _xf_i(u)\ldots f_k(u)}g\left(f(u)\right) \ \ \forall u\in U,\ \forall x\in X_0.
  \end{equation}
  If $X_0$ is a Banach space continuously embedded in $X_1$ and if $f,f_1,\ldots,f_k\in \Gat{U}{X_2}{X_0}{1}$, then 
$\gamma\in \Gat{U}{X_3}{X_0}{1}$.
\end{lemma}
\begin{proof}
  Let $u\in U$, $x\in X_0$, and
let $[u-\epsilon x,u+\epsilon
  x]\subset U$, for some $\epsilon>0$.  
Let
  $h\in[-\epsilon,\epsilon]\setminus\{0\}$.
By strong continuity of $  \partial ^{k+1}g$ and by $k$-linearity of $ \partial ^k_{x_1\ldots x_k}g$ with respect to $x_1,\ldots,x_k$, we can write
  \begin{equation*}
    \begin{split}
          \frac{\gamma(u+hx)-\gamma(u)}{h}&=\\
&\hskip-25pt =   \frac{1}{h}\left(\partial^k_{f_1(u+hx)f_2(u+hx)\ldots
        f_k(u+hx)}g\left(f(u+hx)\right)-\partial^k_{f_1(u+hx)f_2(u+hx)\ldots
        f_k(u+hx)}g\left(f(u)\right)\right)\\
 &\hskip-25pt \phantom{=} +  \frac{1}{h}\sum_{i=1}^k\left(\partial^k_{f_1(u)f_2(u)\ldots f_{i-1}(u)f_i(u+hx)f_{i+1}(u+hx)\ldots f_k(u+hx)}g\left(f(u)\right)\right.\\
  &\hskip+110pt   \left.-    \partial^k_{f_1(u)f_2(u)\ldots f_{i-1}(u)f_i(u)f_{i+1}(u+hx)\ldots f_k(u+hx)}g\left(f(u)\right)\right)\\
    &\hskip-25pt =\int_0^1
    \partial_{\frac{f(u+hx)-f(u)}{h}} \partial^k_{f_1(u+hx)f_2(u+hx)\ldots  f_k(u+hx)}g\left(f(u)+\theta\left(f(u+hx)-fux)\right)\right) {d}  \theta\\
   &\hskip-25pt \phantom{=} +\sum_{i=1}^k \partial^k_{f_1(u)f_2(u)\ldots
      f_{i-1}(u)\frac{f_i(u+hx)-f_i(u)}{h}f_{i+1}(u+hx)\ldots
      f_k(u+hx)}g\left(f(u)\right).
  \end{split}
\end{equation*}
By continuity of $f,f_1,\ldots,f_k$
on the set $(u+\mathbb{R}x)\medcap U$ and by
joint continuity of $ \partial ^{k+1}g$,
the
  integrand function 
is uniformly continuous in
 $(h,\theta)\in
  ([-\epsilon ,\epsilon ]\setminus \{0\})\times [0,1]$.
Then we can pass  to the limit $h\rightarrow 0$ and obtain
  \eqref{eq:formula_derivate_composte}.  

If
  $f,f_1,\ldots,f_k\in \Gat{U}{X_2}{X_0}{1}$,
  then the strong continuity of
  $ \partial _{X_0}\gamma$ 
  comes from Proposition \ref{prop:2014-07-30:01}
and formula  \eqref{eq:formula_derivate_composte},
by recalling also Remark~\ref{rem:2013.09.06}.
\end{proof}

\begin{lemma}\label{propp:2012-05-22-ac}
  Let $n\in \mathbb{N}$. Let $X_0$, $X_1$, $X_2$, $X_3$ be Banach
  spaces, with $X_0$ continuously embedded in $X_1$, and let $U\subset  X_1$ be an open set.
Let
\begin{equation*}
  \begin{dcases}
    f_0,\ldots,f_n\in \Gat{U}{X_2}{X_0}{1}&\\
    f_0^{(k)},\ldots,f_n^{(k)}\in \Gat{U}{X_2}{X_0}{1}& \forall k\in \mathbb{N}\\
  g \in \Gatot{X_2}{X_3}{n+1}\\
  g^{(k)}\in \Gatot{X_2}{X_3}{n+1}&\forall n\in \mathbb{N}.
\end{dcases}
\end{equation*}
Suppose that, for $i=0,\ldots,n$,
\begin{equation*}
  \begin{dcases}
    \lim_{k\rightarrow \infty}f_i^{(k)}(u)= f_i(u)\\
\lim_{k\rightarrow \infty}\partial _xf_i^{(k)}(u)= \partial
    _xf_i(u),
  \end{dcases}
\end{equation*}
uniformly for $u$ on compact subsets of $U$ and $x$ on
compact subsets of $X_0$, and that
$$
\lim_{k\rightarrow \infty}\partial^{j} _{x_1\ldots x_{j}}g^{(k)}(x_0)=\partial^{j} _{x_1\ldots x_j}g(x_0)\qquad j=n,n+1,
$$
uniformly for $x_0,x_1,\ldots,x_j$ on compact subsets of $X_2$. 
Define
\begin{equation}
  \begin{dcases}
    \gamma&\colon U\to X_3,\ u\mapsto 
    \partial^n_{
      f_1(u)\ldots f_n(u)}g(f_0(u))\\
    \gamma^{(k)}&\colon U\to X_3,\ u\mapsto \partial^n_{f_1^{(k)}(u)\ldots
      f^{(k)}_n(u)}g^{(k)}(f^{(k)}_0(u)),\qquad \forall k\in \mathbb{N}.
  \end{dcases}
\end{equation}
Then
\begin{equation}\label{eq:2012-05-22-ae}
\lim_{k\rightarrow \infty}  \partial _x\gamma^{(k)}(u)=
  \partial _x\gamma(u)
\end{equation}
uniformly for $u$ on compact subsets of $U$ and $x$ on
compact subsets of $X_0$.
\end{lemma}
\begin{proof}
Since the composition of sequences of continuous functions uniformly convergent on compact sets is convergent 
to the composition of the limits, uniformly
on compact sets, 
it is sufficient
to recall Remark \ref{rem:2013.09.06},
 apply Lemma
  \ref{lemm:formula_derivate_composte}, 
and consider \eqref{eq:formula_derivate_composte}.
\end{proof}
Let $X_0,X_1$ be Banach spaces, with $X_0$ continuously embedded in $X_1$, and let $U$ be an open subset of $X_1$.
Let $n\in \mathbb{N}$, $n\geq 1$,
 $\mathbf{x}_n\coloneqq \{x_1,\ldots,x_n\}\subset X_0^n$,
 $j\in \{1,\ldots,n\}$.
Then
 \begin{itemize}
\itemsep=-1mm
 \item  $P^j(\mathbf{x}_n)$ denotes the set of
partitions of $\mathbf{x}_n$
 in $j$ non-empty
subsets.
\item \label{2016-02-25:04}
If $f\in \Gat{U}{X_1}{X_2}{n}$
  and
$\mathbf{q}\coloneqq \{y_1,\ldots,y_j\}\subset \mathbf{x}_n$,  then
 $\partial^j_{\mathbf{q}}f(u)$ denotes the derivative
$\partial^j_{y_1\ldots y_j}f(u)$
(\footnote{By
 Remark \ref{rem:schwarz}, there is no
ambiguity due to the fact that $\mathbf{q}$ is not ordered.}).
\item 
 $|\mathbf{q}|$  denotes the cardinality of $\mathbf{q}$.
\end{itemize}

\begin{proposition}[Fa\`a di Bruno's formula]\label{teo:faa-di-bruno}
  Let $n\geq 1$. 
Let $X_0,X_1,X_2,X_3$ be Banach
  spaces, with $X_0$ continuously embedded in $X_1$, and let $U$ be an open subset of $X_1$. 
If $f\in \Gat{U}{X_2}{X_0}{n}$ and
$g\in \Gatot{X_2}{X_3}{n}$, then $g\circ f\in \Gat{U}{X_3}{X_0}{n}$.
  Moreover
     \begin{equation}\label{eq:formula_derivata_gat_composta}
    \partial^j_{\mathbf{x}_j}g\circ f(u)=
    \sum_{i=1}^j \sum_{\{\mathbf{p}^i_1,\ldots,\mathbf{p}^i_i\}\in P^i(\mathbf{x}_j)}\partial^j_{\partial^{|\mathbf{p}^i_1|}_{\mathbf{p}^i_1}f(u)\ldots\partial^{|\mathbf{p}^i_i|}_{\mathbf{p}^i_i}f(u)}g\left(f(u)\right).
  \end{equation}
  for all $u\in U$, $j=1,\ldots, n$, $\mathbf{x}_j=\{x_1,\ldots,x_j\}\subset X_0^j$.
\end{proposition}
\begin{proof}
The proof is standard and is  obtained by induction on $n$ 
 and by making use of
Lemma \ref{lemm:formula_derivate_composte} at each 
step of the inductive argument.
The case $n=1$ is obtained by applying Lemma~\ref{lemm:formula_derivate_composte} with $k=0$.
Now consider the case
 $n\geq 2$.  
By inductive hypothesis, formula
  \eqref{eq:formula_derivata_gat_composta} 
holds true 
for  $j=1,\ldots,n-1$, and we need to prove that it holds for $j=n$.
Let $u\in U$, $x_1,\ldots,x_n\in X_0$, $\mathbf{x}_{n-1}\coloneqq \{x_1,\ldots,x_{n-1}\}$.
Then, by \eqref{eq:formula_derivata_gat_composta},
$$
\partial^{n-1}_{x_1\ldots x_{n-1}}g\circ f(u)= \sum_{i=1}^{n-1}
\sum_{\{\mathbf{p}^i_1,\ldots,\mathbf{p}^i_i\}\in
  P^i(\mathbf{x}_{n-1})}\partial^{n-1}_{\partial^{|\mathbf{p}^i_1|}_{\mathbf{p}^i_1}f(u)\ldots \partial^{|\mathbf{p}^i_i|}_{\mathbf{p}^i_i}f(u)}g\left(f(u)\right).
$$
By applying Lemma \ref{lemm:formula_derivate_composte},
with $k=i$
 and
$f_j=\partial^{|\mathbf{p}^i_j|}_{\mathbf{p}_j^i}f$, for $j=1,\ldots,i$,
 to
 each
member of the sum over $P^i(\mathbf{x}_{n-1})$, 
we
obtain, for all $x_n\in X_0$,
\begin{equation*}
  \begin{split}
  \partial^{n}_{x_1\ldots x_{n-1}x_n}g\circ f(u)&= \sum_{i=1}^{n-1}
  \sum_{\{\mathbf{p}^i_1,\ldots,\mathbf{p}^i_i\}\in
    P^i(\mathbf{x}_{n-1})} \bigg(
    \partial^n_{\partial_{x_n}f(u)\partial^{|\mathbf{p}^i_1|}_{\mathbf{p}^i_1}f(u)\ldots \partial^{|\mathbf{p}^i_i|}_{\mathbf{p}^i_i}f(u)}g(f(u))\\
    &\hskip160pt+ \sum_{l=1}^i
      \partial^n_{\partial^{|\mathbf{p}^i_1|}_{\mathbf{p}^i_1}f(u)\ldots\partial_{x_n}\partial^{|\mathbf{p}^i_l|}_{\mathbf{p}^i_l}f(u)\ldots\partial^{|\mathbf{p}^i_i|}_{\mathbf{p}^i_i}f(u)}g(f(u))
    \bigg)\\
    &=\sum_{i=1}^{n-1} \left(
      \sum_{\substack{\{\mathbf{p}^{i+1}_1,\ldots,\mathbf{p}^{i+1}_{i+1}\}\in
          P^{i+1}(\mathbf{x}_{n})\colon \\\{x_n\}\in
          \{\mathbf{p}^{i+1}_1,\ldots,\mathbf{p}^{i+1}_{i+1}\}}}
      \partial^n_{\partial^{|\mathbf{p}^{i+1}_1|}_{\mathbf{p}^{i+1}_1}f(u)\ldots
\partial^{|\mathbf{p}^{i+1}_{i+1}|}_{\mathbf{p}^{i+1}_{i+1}}f(u)}
      g(f(u))\right.\\
    &\hskip160pt\left.  +
      \sum_{\substack{\{\mathbf{p}^{i}_1,\ldots,\mathbf{p}^{i}_{i}\}\in
          P^{i}(\mathbf{x}_{n})\colon \\\{x_n\}\not\in
          \{\mathbf{p}^{i}_1,\ldots,\mathbf{p}^{i}_{i}\}}}
      \partial^n_{\partial^{|\mathbf{p}^{i}_1|}_{\mathbf{p}^{i}_1}f(u)\ldots\partial^{|\mathbf{p}^{i}_{i}|}_{\mathbf{p}^{i}_i}f(u)}
      g(f(u))
    \right)\\
    &= \sum_{i=1}^n \sum_{\{\mathbf{p}^i_1,\ldots,\mathbf{p}^i_i\}\in
      P^i(\mathbf{x}_n)}\partial^n_{\partial^{|\mathbf{p}^i_1|}_{\mathbf{p}^i_1}f(u)\ldots\partial^{|\mathbf{p}^i_i|}_{\mathbf{p}^i_i}f(u)}g(f(u)).
  \end{split}
\end{equation*}
This concludes the proof of \eqref{eq:formula_derivata_gat_composta}.
\end{proof}

\begin{proposition}\label{propp:2012-05-24-ab}
  Let $n\geq 1$.
Let $X_0,X_1,X_2,X_3$ be Banach
  spaces, with $X_0$ continuously embedded in $X_1$, and let $U$ be an open subset of $X_1$.
Let
\begin{equation*}
  \begin{dcases}
    f\in \Gat{U}{X_2}{X_0}{n}&\\
    f^{(k)}\in \Gat{U}{X_2}{X_0}{n}&\forall k\in \mathbb{N}\\
  g \in \Gatot{X_2}{X_3}{n}\\
  g^{(k)} \in \Gatot{X_2}{X_3}{n}
&\forall k\in \mathbb{N}.
\end{dcases}
\end{equation*}
Suppose that
  \begin{equation*}
    \begin{dcases}
 \lim_{k\rightarrow \infty}
   f^{(k)}(u)=f(u)\\
\lim_{k\rightarrow \infty}\partial ^j_{x_1\ldots 
    x_j}f^{(k)}(u)= \partial ^j_{x_1\ldots x_j}f(u)\qquad
  \mbox{for }j=1,\ldots,n,
       \end{dcases}
  \end{equation*}
  uniformly for $u$ on compact subsets of
   $U$ and $x_1,\ldots,x_j$ on compact subsets of $X_0$, and that
   \begin{equation*}
     \begin{dcases}
\lim_{k\rightarrow \infty}   g^{(k)}(x)= g(x)\\
\lim_{k\rightarrow \infty} \partial ^j_{x_1\ldots
     x_j}g^{(k)}(x)= \partial ^j_{x_1\ldots x_j}f(x)
\qquad
  \mbox{for }j=1,\ldots,n,
     \end{dcases}
   \end{equation*}
 uniformly for $x,x_1,\ldots,x_j$ on compact subsets of $X_2$.  Then 
  \begin{equation*}
    \begin{dcases}
  \lim_{k\rightarrow\infty}          g^{(k)}\circ f^{(k)}(u)=g\circ f(u)\\
\lim_{k\rightarrow \infty}      \partial ^j_{x_1\ldots
        x_j} g^{(k)}\circ f^{(k)}(u)=\partial ^j_{x_1\ldots
        x_j}g\circ f(u)
\qquad
  \mbox{for }j=1,\ldots,n,
    \end{dcases}
  \end{equation*}
 uniformly for $u$ on compact
  subsets of $U$ and $x_1,\ldots,x_j$ on compact subsets of $X_0$.
\end{proposition}
\begin{proof}
  Use recursively formula \eqref{eq:formula_derivata_gat_composta} and
  Lemma \ref{propp:2012-05-22-ac}.
\end{proof}

\subsection{Contractions in Banach spaces: survey of basic results}
\label{2017-04-28:01}


In this section,
we assume that $X$ and $Y$ are Banach spaces, and that  $U$ is an open subset of $X$.
We recall that, if $\alpha\in [0,1)$
and 
$h\colon U\times Y\rightarrow Y$,
then $h$ is said to be a  \emph{parametric $\alpha$-contraction}
if 
$$
|h(u,y)-h(u,y')|_Y\leq \alpha |y-y'|\qquad
\forall u\in U,\ \forall y,y'\in Y.
$$
By the Banach contraction principle,
 to any such $h$ we can associate
a uniquely defined map $\varphi\colon U\rightarrow Y$ such that $h(u,\varphi(u))=\varphi(u)$ for all $u\in U$.
We
refer to
$\varphi$ as to \emph{the fixed-point map associated to $h$}.
For future reference, we summurize some basic continuity properties that $\varphi$ inherites from~$h$.

\smallskip
The following lemma 
  can be found in \cite[p.\ 13]{Granas2003}.

\begin{lemma}\label{2016-02-25:02}
Let $ \alpha \in[0,1)$ and let  $h(u,\cdot)\colon  U\times Y\to Y$,  $h_n(u,\cdot): U\times Y\to Y$, for $n\in\mathbb{N}$, be parametric $\alpha$-contractions.
Denote by $\varphi$ (resp.\ $\varphi_n$) the fixed-point map associated to $h$ (resp.\ $h_n$).
\begin{enumerate}[(i)]
\item \label{2016-02-23:02}
 If $h_n\rightarrow h$ pointwise on $U\times Y$, then $\varphi_n\rightarrow \varphi$ pointwise on $U$.
\item\label{2016-02-23:03}
If $A\subset U$ is a set and if there exists an increasing concave function $w$ on $\mathbb{R}^+$ such that $w(0)=0$ and
\begin{equation}\label{2016-02-23:01}
|h(u,y)-h(u',y)|_Y \leq w(|u-u'|_X)\qquad \forall  u,u'\in A,\ \forall y\in Y,
\end{equation}
then
\[
|\varphi(u)-\varphi(u')|_Y \leq \frac{1}{1-\alpha}w(|u-u'|_X)
\qquad \forall \, u,u'\in A.
\]
\item \label{2016-02-25:01} If $h$ is continuous, then $\varphi$ is continuous.
\end{enumerate}
\end{lemma}
\begin{proof}
Since
$h$ and $h_n$ are
$\alpha$-contractions,
we have
\begin{gather}
    |\varphi_n(u)-\varphi(u')|\leq \frac{|h_n(u,\varphi(u)))-h(u',\varphi(u'))|}{1-\alpha},\label{2016-02-23:04}\\[4pt]
    |\varphi(u)-\varphi(u')|\leq
    \frac{|h(u,\varphi(u'))-h(u',\varphi(u'))|}{1-\alpha},
\label{2016-02-23:05}
\end{gather}
for all $u,u'\in U$.
Then \eqref{2016-02-23:04} yields \emph{(\ref{2016-02-23:02})}
 by taking $u=u'$ and letting $n\to \infty$, and
\eqref{2016-02-23:05}
yields
\emph{(\ref{2016-02-23:03})} by using 
\eqref{2016-02-23:01}.

Regarding \emph{(\ref{2016-02-25:01})},
let $u'\in U$,  $u_n\rightarrow u'$ in $U$,
 let $V\subset U$ be an open set 
containing $u'$,
and let $\bar n\in \mathbb{N}$ such that
$u_n-u'+V\subset U$ for all $n\geq \bar n$.
Define $h_n\colon V\times Y\rightarrow Y$ by $h_n(u,y)\coloneqq h(u+u_n-u',y)$ for all $(u,y)\in V\times Y$.
Then $h_n$ is a parametric $\alpha$-contraction.
Denote by $\varphi_n$ its associated fixed-point map.
Then, by  continuity of $h$ and by
\emph{(\ref{2016-02-23:02})},
$\varphi_n(u)\rightarrow \varphi(u)$ for all $u\in V$.
In particular, $\varphi(u_n)=\varphi_n(u')\rightarrow \varphi(u')$, hence $\varphi$ is continuous.
\end{proof}

\begin{remark}\label{2016-02-25:00}
If $h\colon U\times Y\rightarrow Y$ is a parametric $\alpha$-contraction ($\alpha\in [0,1)$)
belonging to $\Gat{U\times Y}{Y}{\{0\}\times Y}{1}$,
then
\begin{equation}
  \label{eq:2016-04-20:03}
  |  \partial _Yh(u,y)|_{L(Y)}\leq \alpha \qquad \forall u\in U,\ y\in Y,
\end{equation}
where $|\cdot|_{L(Y)}$ denotes the operator norm on $L(Y)$..
Hence $ \partial _Y h(u,y)$ is invertible and the family 
$\{(I- \partial _Yh(u,y))^{-1}\}_{(u,y)\in U\times Y}$ is uniformly bounded in $L(Y)$.
For what follows, it is important to notice that, for all
$y\in Y$,
\begin{equation}
  \label{eq:2016-02-25:03}
  U\times Y\rightarrow Y,\ (u,y') \mapsto (I- \partial _Yh(u,y'))^{-1}y
\end{equation}
is continuous, hence,
because of the formula
$$
(I- \partial _Yh(u,y'))^{-1}y=
\sum_{n\in \mathbb{N}} \left(  \partial _Yh(u,y') \right) ^ny
$$
and of Lebesgue's dominated convergence theorem (for series),
 $(I- \partial _Yh(u,y'))^{-1}y$ is jointly continuous in $u,y',y$.
\end{remark}

The following proposition shows that the fixed-point map
 $\varphi$
associated to a parametric $\alpha$-contraction $h$
inherits from $h$ the strongly continuous G\^ateaux
differentiability.

\begin{proposition}\label{propp:gat.diff.phi}
If 
$h\in \Gatot{U\times Y}{Y}{1}$
is a parametric $\alpha$-contraction
 and if
 $\varphi$ is the fixed-point map associated to $h$,
then 
$\varphi\in \Gatot{U}{Y}{1}$ and
  \begin{equation}\label{eq:derphi}
    \partial _x\varphi(u)=\left(I- \partial _Yh\left(u,\varphi(u)\right)\right)^{-1}\left(\partial _xh\left(u,\varphi(u)\right)\right)\qquad \forall u\in U,\forall x\in X.
  \end{equation}
\end{proposition}
\begin{proof}
For 
the proof, see
\cite[Lemma 2.9]{DaPrato2014}, or
 \cite[Proposition C.0.3]{Cerrai2001}, taking into account also \cite[Remark C.0.4]{Cerrai2001},
Lemma \ref{2016-02-25:02}\emph{(\ref{2016-02-25:01})}, 
Remark \ref{2016-02-25:00}.
\end{proof}

\subsection{G\^ateaux differentiability of order $n$ of fixed-point maps}\label{2016-02-22:08}

In this section we provide a result for the G\^ateux differentiability
 up to a generic order $n$
 of a fixed-point map $\varphi$ associated to a parametric $\alpha$-contraction $h$, under the assumption that $h$ is G\^ateaux differentiable only with respect to some invariant subspaces of the domain.

The main result of this section is Theorem \ref{teo:derivabilita.punto.fisso}, which is suitable to be applied to
mild solutions of SDEs (Section \ref{2016-04-20:02}).
When $n=1$,
Theorem
\ref{teo:derivabilita.punto.fisso}
reduces to
Proposition
\ref{propp:gat.diff.phi}.
In the case $n=2$,
Theorem
\ref{teo:derivabilita.punto.fisso}
is also well-known, and a proof can be found in
\cite[Proposition C.0.5]{Cerrai2001}.
On the other hand, when the order of differentiability $n$ is generic,
the fact that the parametric $\alpha$-contraction is assumed to be differentiable only with respect to certain subspaces makes non-trivial
the proof
 of the theorem.
To our knowledge, a reference for the case $n\geq 3$ is not available in the literature.
The main issue consists in providing a precise formulation of the statement, with its assumptions,
that 
can be proved
 by induction.

\medskip
For the sake of readability, we collect the assumptions of Theorem \ref{teo:derivabilita.punto.fisso}
in the following

\begin{assumption}\label{2016-02-24:00}
${}$
  \begin{enumerate}[(1)]
  \item \label{2016-02-24:01}
   $n\geq 1$ and $\alpha\in [0,1)$;
 \item \label{2016-02-24:02} $X$ is a  Banach space and $U$
is an open subset of $X$.
\item \label{2016-02-24:03}
  $Y_1\supset Y_2\supset \ldots \supset Y_n$
is a decreasing sequence of Banach spaces, with
norms \mbox{$|\cdot|  _1$}, \ldots , $|  \cdot  |  _n$, respectively.
\item \label{2016-02-24:04}
For $k=1,\ldots,n$ and $j=1,2,\ldots,k$, 
the canonical embedding of $Y_k$ into $Y_j$, denoted by
  $i_{k,j}\colon Y_k\to Y_{j}$,
is continuous.
\item \label{2016-02-24:05}
 $h_1\colon U\times Y_1\to Y_1$
is a function such that
 $h_1\left(U\times Y_k\right)\subset Y_k$ for  $k=2,\ldots,n$. 
For $k=2,\ldots,n$,
we denote by $h_k$ the induced function
\begin{equation}
  \label{eq:2016-04-20:04}
  h_k\colon U\times Y_k\rightarrow Y_k,\  (u,y)\mapsto h_1(u,y).
\end{equation}
\item \label{2016-02-24:06}
For $k=1,\ldots,n$, $h_k$ is continuous and satisfies
  \begin{equation}\label{eq:cond_contrazione}
    \left |  h_k(u,y)-h_k(u,y')\right | _k\leq \alpha |  y-y' | _k\qquad \forall u\in U,\ \forall y,y'\in Y_k.
  \end{equation}
  \item \label{2016-02-24:07}
For $k=1,\ldots,n$, $h_k\in \Gat{U\times Y_k}{Y_k}{X\times \{0\}}{n}$.
  \item \label{2016-02-24:09}
For $k=1,\ldots,n-1$, 
$h_k\in \Gat{U\times Y_k}{Y_k}{X\times Y_{k+1}}{n}$
\item \label{2016-02-24:08}
  For
$k=1,\ldots,n$,
 $j=1,\ldots,    n-1$, for all $u\in U$, $z_1,\ldots,z_j\in X$,  $y,z_{j+1}\in Y_k$, and for all permutations $\sigma$ of $\{1,\ldots,j+1\}$,
the 
directional derivative $\partial ^{j+1}_{z_{\sigma(1)}\ldots z_{\sigma(j+1)}}
h_k(u,y)$
 exists,
and 
\begin{equation}\label{2016-02-24:11}
U\times Y_k\times X^j\times Y_k\to Y_{k},\ (u,y,z_1,\ldots,z_j,z_{j+1})\mapsto \partial ^{j+1}_{z_{\sigma(1)}\ldots z_{\sigma(j)}z_{\sigma(j+1)}}h_{k}(u,y)
\end{equation}
is continuous.
\newcounter{c1}
\setcounter{c1}{\value{enumi}}
\end{enumerate}
\end{assumption}

\begin{theorem}\label{teo:derivabilita.punto.fisso}
Let Assumption~\ref{2016-02-24:00} be satisfied and let $\varphi\colon U\rightarrow Y_1$ denote the fixed-point function associated to the parametric $\alpha$-contraction $h_1$.
Then,
for $j=1,\ldots,n$, 
we have
$\varphi\in \Gatot{U}{Y_{n-j+1}}{j}$ and,
for all $u\in U$, $x_1,\ldots,x_j\in X$,
$\partial ^j_{x_1\ldots x_j}\varphi(u)$
is given by the formula
  \begin{equation}\label{eq:formula_fixed_point_derivatives}
    \begin{split}
 \hskip-13pt         \partial ^j_{x_1\ldots x_j}\varphi(u)=&
    \left(I- \partial _{Y_1}h_1(u,\varphi(u))\right)^{-1}
      \partial ^j_{x_1\ldots x_j}h_1(u,\varphi(u))\\
&\hskip-40pt+
      \sum_{\substack{\mathbf{x}\in 2^{\{x_1,\ldots,x_j\}}
\\ \mathbf{x}\neq
          \emptyset}}\sum_{i=\max\{1,2-j+|\mathbf{x}|\}}^{|\mathbf{x}|}\sum_{\substack{\mathbf{p}\in
          P^i(\mathbf{x})\\\mathbf{p}=(\mathbf{p}_{1},\ldots,\mathbf{p}_i)}}
       \left(I- \partial _{Y_1}h_1(u,\varphi(u))\right)^{-1}\partial^j [\mathbf{x}^c,\mathbf{p}]h_1(u,\varphi(u))
    \end{split}
  \end{equation}
  where
  $2^{\{x_1,\ldots,x_i\}}$
 is  the power  set 
  of $ \{x_1,\ldots,x_i\}$, 
$P^i(\mathbf{x})$ is the
  set of partitions of $\mathbf{x}$ in $i$ non-empty parts,
$\mathbf{x}^c\coloneqq\{x_1,\ldots, x_j\}\setminus \mathbf{x}$,
 and 
$
\partial^j [\mathbf{x}^c,\mathbf{p}]\coloneqq
 \partial ^{j-|\mathbf{x}|}_{\mathbf{x}^c}
 \partial ^{|\mathbf{x}|}
_{   \partial ^{|\mathbf{p}_1|}_{\mathbf{p}_{1}}\varphi(u),
    \ldots,
    \partial ^{|\mathbf{p}_{i}|}_{\mathbf{p}_{i}}
\varphi(u)
}
$~(\,\footnote{Recall notation at p.\ \pageref{2016-02-25:04}.}).
\end{theorem}
\begin{proof}\label{2016-02-26:00}
The proof is 
 by induction on $n$. The case $n=1$ is
provided
 by Proposition~\ref{propp:gat.diff.phi}. 

Let $n\geq 2$.  
Clearly, it is sufficient to prove that 
$\varphi\in \Gatot{U}{Y_n}{n}$
and that \eqref{eq:formula_fixed_point_derivatives} holds true for $j=n$.
Since we are assuming that the theorem holds true for $n-1$, we can apply it with the data
$$
\wt h_1\colon U\times \wt Y_2\rightarrow \wt Y_2,\ 
\ldots,
\wt h_{n-1}\colon U\times \wt Y_n\rightarrow \wt Y_n,
$$
 where
 $\widetilde  h_k\coloneqq 
h_{k+1}$, $\widetilde  Y_k\coloneqq Y_{k+1}$, for $k=1,\ldots,n-1$. 
According to the claim,
the fixed-point
 function $\widetilde  \varphi$ of $\widetilde  h_1$
belongs to $\Gatot{U}{\wt Y_{(n-1)-j+1}}{j}$, for $j=1,\ldots,n-1$,
and formula
\eqref{eq:formula_fixed_point_derivatives} holds true for $\widetilde  \varphi$ and $j=1,\ldots,n-1$. Since $\varphi(u)=(i_{2,1}\circ \widetilde  \varphi) (u)$, for  $u\in U$,
 we have $\varphi\in \Gatot{U}{\wt Y_{n-j}}{j}=\Gatot{U}{Y_{n-j+1}}{j}$, for $j=1,\ldots,n-1$, and
  $$
  \partial ^j_{x_1\ldots x_j}\varphi(u)=\partial ^j_{x_1\ldots x_j}\widetilde 
  \varphi(u)\in \widetilde  Y_{n-j}=Y_{n-j+1}, \qquad \forall u\in U, \ \forall x_1,\ldots,x_j\in X.
  $$
Then
  \eqref{eq:formula_fixed_point_derivatives}
holds true for $\varphi$ up to order $j=n-1$.
In particular
$\varphi\in \Gatot{U}{Y_2}{n-1}$, hence,
 for $x_1,\ldots,x_n\in
X$, $\epsilon>0$, we  can write
  \begin{equation}\label{2017-05-23:04}
    \begin{split}
&          \partial ^{n-1}_{x_1\ldots x_{n-1}}\varphi(u+\varepsilon x_n) -
    \partial ^{n-1}_{x_1\ldots x_{n-1}}\varphi(u) \\
&=    \left( \partial _{Y_1}h_1(u+\varepsilon x_n,\varphi(u+\varepsilon x_n)).\partial ^{n-1}_{x_1\ldots
        x_{n-1}}\varphi(u+\varepsilon x_n)-
\partial _{Y_1}h_1
(u,\varphi(u)).\partial ^{n-1}_{x_1\ldots x_{n-1}}\varphi(u)\right)\\
&  \phantom{=:}  + \left(\mc S(u+\varepsilon x_n)-
      \mc S(u)\right)\\
 &\eqqcolon\mathbf{I}+\mathbf{II},
  \end{split}
\end{equation}
  where $\mc S(\cdot)$ denotes the sum
$$
\mathcal{S}(v)\coloneqq
          \partial ^{n-1}_{x_1\ldots x_{n-1}}h_1(v,\varphi(v))
+
      \sum_{\substack{\mathbf{x}\in 2^{\{x_1,\ldots,x_{n-1}\}}
\\ \mathbf{x}\neq
          \emptyset}}\sum_{i=\max\{1,2-(n-1)+|\mathbf{x}|\}}^{|\mathbf{x}|}\sum_{\substack{\mathbf{p}\in
          P^i(\mathbf{x})\\\mathbf{p}=(\mathbf{p}_{1},\ldots,\mathbf{p}_i)}}
       \partial^{n-1} [\mathbf{x}^c,\mathbf{p}]h_1(v,\varphi(v)),
$$
for $v\in U$.
By recalling that $\varphi\in \Gatot{U}{Y_{n-j+1}}{j}$, $j=1,\ldots,n-1$,
hence
  by taking into account
  with respect to which space the derivatives of $\varphi$ are continuous, we write
  \begin{equation}\label{2017-05-23:05}
    \begin{split}
          \mathbf{I}=&\partial _{\partial ^{n-1}_{x_1\ldots x_{n-1}}\varphi(u+\varepsilon
      x_n)}h_1(u+\varepsilon x_n,\varphi(u+\varepsilon x_n))-
    \partial _{\partial ^{n-1}_{x_1\ldots x_{n-1}}\varphi(u)}h_1(u,\varphi(u))\\
    =&\int_0^1 \partial _{x_n}\partial _{\partial ^{n-1}_{x_1\ldots
        x_{n-1}}\varphi(u+\varepsilon x_n)}h_1(u+\theta \varepsilon
    x_n,\varphi(u+\varepsilon x_n))\varepsilon  {d}  \theta
    \\
    &+\int_0^1
    \partial _{\frac{\varphi(u+\varepsilon x_n)-\varphi(u)}{\varepsilon}}\partial _{\partial ^{n-1}_{x_1\ldots x_{n-1}}\varphi(u+\varepsilon x_n)}h_1(u,\varphi(u)+\theta (\varphi(u+\varepsilon x_n)-\varphi(u)))\varepsilon  {d}  \theta \\
&    + \partial _{\partial ^{n-1}_{x_1\ldots x_{n-1}}\varphi(u+\varepsilon x_n)-
      {\partial ^{n-1}_{x_1\ldots x_{n-1}}\varphi(u)} }h_1(u,\varphi(u))\\
    =&\mathbf{I_1}+\mathbf{I_2}+ \partial _{Y_1}h_1(u,\varphi(u)).\left(\partial ^{n-1}_{x_1\ldots
        x_{n-1}}\varphi(u+\varepsilon x_n)- {\partial ^{n-1}_{x_1\ldots
          x_{n-1}}\varphi(u)}\right),
  \end{split}
\end{equation}
  with~(\footnote{The limits should be understood in the
    suitable spaces $Y_k$. 
 For instance, when
    computing $\lim_{\varepsilon\to 0}\frac{\mathbf{I_1}}{\varepsilon}$, the
    object $\partial ^{n-1}_{x_1\ldots x_{n-1}}\varphi(u+\varepsilon x_n)$
    should be considered in the space $Y_2$, which can be done thanks to the
    inductive hypothesis.})
  $$
  \lim_{\varepsilon \to
    0}\frac{\mathbf{I_1}}{\varepsilon}=\partial _{x_n}\partial _{\partial ^{n-1}_{x_1\ldots
      x_{n-1}}\varphi(u)}h_1(u,\varphi(u)) \qquad \mbox{and} \qquad
  \lim_{\varepsilon \to
    0}\frac{\mathbf{I_2}}{\varepsilon}=\partial _{\partial _{x_n}\varphi(u)}\partial _{\partial ^{n-1}_{x_1\ldots
      x_{n-1}}\varphi(u)}h_1(u,\varphi(u)) .
  $$
In a similar way, 
  \begin{equation}\label{2017-05-23:00}
    \begin{split}
         \lim_{\varepsilon\to 0}&\frac{\mathbf{II}}{\varepsilon}=
    \partial _{x_n}\partial ^{n-1}_{x_1\ldots x_{n-1}}h_1(u,\varphi(u))
    +\partial _{\partial _{x_n}\varphi(u)}\partial ^{n-1}_{x_1 \ldots x_{n-1}}h_1(u,\varphi(u))\\
&    +\sum_{\substack{\mathbf{x}\in 2^{\{x_1,\ldots,x_{n-1}\}}\\        \mathbf{x}\neq \emptyset}} \sum_{i=\max\{1,2-(n-1)+|\mathbf{x}|\}}^{|\mathbf{x}|}
    \sum_{\substack{\mathbf{p}\in
        P^i(\mathbf{x})\\\mathbf{p}=(\mathbf{p}_1,\ldots,
\mathbf{p}_i)}}
    \vphantom{\sum_M^M} \partial _{x_n}
    \partial^{n-1} [\mathbf{x}^c,\mathbf{p}]h_1(u,\varphi(u))\\
&+
    \sum_{\substack{\mathbf{x}\in 2^{\{x_1,\ldots,x_{n-1}\}}\\
        \mathbf{x}\neq \emptyset}} \sum_{i=\max\{1,2-(n-1)+|\mathbf{x}|\}}^{|\mathbf{x}|}
    \sum_{\substack{\mathbf{p}\in
        P^i(\mathbf{p})\\\mathbf{p}=(\mathbf{p}_1,\ldots,
\mathbf{p}_i)}}
    \vphantom{\sum_M^M}\left(
\vphantom{\sum_{\substack{\mathbf{p}\in
          P^i(\mathbf{x})\\\mathbf{p}=(\mathbf{p}_1,\ldots,\mathbf{p}_i)}}\partial _{\mathbf{x}
        ^c}^{|\mathbf{x}^c|}}
    \partial _{\partial _{x_n}\varphi(u)}
    \partial^{n-1} [\mathbf{x}^c,\mathbf{p}]h_1(u,\varphi(u))\right.
    \\
&+
\left.    \sum_{j=1}^i\partial _{\mathbf{x}
      ^c}^{|\mathbf{x}^c|}\partial _{\partial ^{|
\mathbf{p}_1|}_{\mathbf{p}_1}\varphi(u)}
    \ldots \partial _{\partial ^{|\mathbf{p}_{j-1}
|}_{\mathbf{p}_{j-1}}\varphi(u)}
    \partial _{\partial _{x_n}\partial ^{|
\mathbf{p}_j
|}_{
\mathbf{p}_j
}\varphi(u)}
    \partial _{\partial ^{|\mathbf{p}_{j+1}|}_{\mathbf{p}
_{j+1}}\varphi(u)} \ldots
    \partial _{\partial ^{|\mathbf{p}_i|}_{\mathbf{p}_i
}\varphi(u)}
    h_1(u,\varphi(u))
\vphantom{\sum_{\substack{p_\pi\in
          P^l(\pi)\\p_\pi=(p_{\pi,1},\ldots,p_{\pi,l})}}\partial _{\pi
        ^c}^{|\pi^c|}}
\right).
  \end{split}
\end{equation}
Notice that
  \begin{equation}\label{2017-05-23:01}
    \begin{split}
       & 
\hskip-8pt          \sum_{\substack{\mathbf{x}\in 
2^{\{x_1,\ldots,x_{n-1}\}}
\\
        \mathbf{x}\neq \emptyset}} \sum_{i=\max\{1,2-(n-1)+|\mathbf{x}|\}}^{|\mathbf{x}|}
    \sum_{\substack{p_\pi\in
        P^i(\mathbf{x})\\\mathbf{p}=
(
\mathbf{p}_1,\ldots,\mathbf{p}_i
)}} \partial _{x_n}
    \partial^{n-1} [\mathbf{x}^c,\mathbf{p}]h_1(u,\varphi(u))\\
&\hskip-10pt =   \sum_{\substack{\mathbf{x}\in 2^{\{x_1,\ldots,x_n\}}
\\
       \mathbf{x}\neq \emptyset\\ x_n\not\in \mathbf{x}}}
    \sum_{i=\max\{1,2-n+|\mathbf{x}|\}}^{|\mathbf{x}|} \sum_{\substack{\mathbf{p}\in
        P^i(\mathbf{x})\\\mathbf{p}=(\mathbf{p}_1,\ldots,
\mathbf{p}_i)}}
    \partial^n [\mathbf{x}^c,\mathbf{p}]h_1(u,\varphi(u)) -\partial _{x_n} \partial _{\partial ^{n-1}_{x_1\ldots
        x_{n-1}}\varphi(u)}h_1(u,\varphi(u))
  \end{split}
\end{equation}
and
  \begin{equation}\label{2017-05-23:02}
    \begin{split}
         \sum_{\substack{\mathbf{x}\in 
2^{\{x_1,\ldots,x_{n-1}\}}\\
        \mathbf{x}\neq \emptyset}} &\sum_{i=\max\{1,2-(n-1)+|\mathbf{x}|\}}^{|\mathbf{x}|}
    \sum_{\substack{\mathbf{p}\in
        P^i(\mathbf{x})\\\mathbf{p}=(\mathbf{p}_1,\ldots,
\mathbf{p}_i)}}
    \partial _{\partial _{x_n}\varphi(u)}\partial^{n-1} [\mathbf{x}^c,\mathbf{p}]h_1(u,\varphi(u))\\
&=      \sum_{\substack{\mathbf{x}\in
2^{\{x_1,\ldots,x_n\}}
\\
        x_n\in\mathbf{x}\\ \mathbf{x}\neq \{x_n\}}} \sum_{i=\max\{1,2-n+|\mathbf{x}|\}}^{|\mathbf{x}|}
    \sum_{\substack{\mathbf{p}\in
        P^i(\mathbf{x})\\\mathbf{p}=(\mathbf{p}_1,\ldots,
\mathbf{p}_i)\\\{x_n\}\in
        \mathbf{p}}} \partial^n [\mathbf{x}^c,\mathbf{p}]h_1(u,\varphi(u))\\
  &\phantom{=}  -\partial _{\partial _{x_n}\varphi(u)}\partial _{\partial ^{n-1}_{x_1\ldots
        x_{n-1}}\varphi(u)}h_1(u,\varphi(u))
  \end{split}
\end{equation}
and
  \begin{equation}\label{2017-05-23:03}
    \begin{split}
          \sum_{\substack{\mathbf{x}\in 
2^{\{x_1,\ldots,x_{n-1}\}}\\
       \mathbf{x} \neq \emptyset}} &\sum_{i=\max\{1,2-(n-1)+|\mathbf{x}|\}}^{|\mathbf{x}|}
    \sum_{\substack{\mathbf{p}\in
        P^i(\mathbf{x})\\\mathbf{p}=(\mathbf{p}_1,\ldots,
\mathbf{p}_i
)}}
    \sum_{j=1}^i
L(\mathbf{p},j;u)
\\
 &=   \sum_{\substack{\mathbf{x}\in 2^{\{x_1,\ldots,x_n\}}\\
        x_n\in \mathbf{x}\\
\mathbf{x}\neq \{x_n\}}} \sum_{i=\max\{1,2-n+|\mathbf{x}|\}}^{|\mathbf{x}|}
    \sum_{\substack{\mathbf{p}\in
        P^i(\mathbf{x})\\\mathbf{p}=(\mathbf{p}_1,\ldots,
\mathbf{p}_i)\\\{x_n\}\not\in
        \mathbf{p}}} \partial^n [\mathbf{x}^c,\mathbf{p}
]h_1(u,\varphi(u))
  \end{split}
\end{equation}
where
$$
L(\mathbf{p},j;u)
\coloneqq 
\\
    \partial _{\mathbf{x} ^c}^{|\mathbf{x}^c|}
    \partial^{|\mathbf{x}|} _{
      \partial ^{|\mathbf{p}_1|}_{\mathbf{p}
        _1}\varphi(u) \ldots
      \partial ^{|\mathbf{p}_{j-1}|}_{
        \mathbf{p}_{j-1}
      \varphi(u)}
      \partial _{x_n}\partial ^{|\mathbf{p}_j
        |}_{\mathbf{p}_j}\varphi(u)
      \partial ^{|\mathbf{p}_{j+1}
        |}_{\mathbf{p}_{j+1}}\varphi(u)
      \ldots
      \partial ^{|\mathbf{p}_i|}_{\mathbf{p}_i
      }\varphi(u)}
    h_1(u,\varphi(u)).
$$
By collecting
\eqref{2017-05-23:00},
\eqref{2017-05-23:01},
\eqref{2017-05-23:02},
\eqref{2017-05-23:03},
we obtain
\begin{equation*}
  \begin{split}
    \lim_{\varepsilon\to 0}\frac{\mathbf{II}}{\varepsilon}=
    &\partial
    _{\partial _{x_n}\varphi(u)}\partial
    ^{n-1}_{x_1\ldots x_{n-1}}h_1(u,\varphi(u)) +\partial
    ^n_{x_1\ldots
      x_{n}}h_1(u,\varphi(u))
-\partial _{\partial _{x_n}\varphi(u)}\partial _{\partial
      ^{n-1}_{x_1\ldots x_{n-1}}\varphi(u)}h_1(u,\varphi(u)) \\
&
+    \sum_{\substack{\mathbf{x}\in 
2^{\{x_1\ldots x_n\}}
\\\mathbf{x}\neq \emptyset\\
        \mathbf{x}\neq\{x_n\}}}\sum_{i=\max\{1,2-n+|\mathbf{x}|\}}^{|\mathbf{x}|}\sum_{\substack{\mathbf{p}\in
        P^i(\mathbf{x})\\\mathbf{p}=(\mathbf{p}_1,\ldots,
\mathbf{p}_i)}}
    \partial^n [\mathbf{x}^c,\mathbf{p}]h_1(u,\varphi(u))
    -\partial _{x_n}\partial _{\partial ^{n-1}_{x_1\ldots x_{n-1}}\varphi(u)}h_1(u,\varphi(u))\\
    =
    &\partial
    ^n_{x_1\ldots x_{n}}h_1(u,\varphi(u))
 -\partial _{\partial _{x_n}\varphi(u)}\partial _{\partial
      ^{n-1}_{x_1\ldots x_{n-1}}\varphi(u)}h_1(u,\varphi(u)) 
\\
& +\sum_{\substack{\mathbf{x}\in 
2^{\{x_1,\ldots,x_n\}}
\\\mathbf{x}\neq
        \emptyset}}\sum_{i=\max\{1,2-n+|\mathbf{x}|\}}^{|\mathbf{x}|}\sum_{\substack{\mathbf{p}\in
        P^i(\mathbf{x})\\\mathbf{p}=(\mathbf{p}_1,\ldots,
\mathbf{p}_i)}}
    \partial^n [\mathbf{x}^c,\mathbf{p}]h_1(u,\varphi(u))
    -\partial _{x_n}\partial _{\partial ^{n-1}_{x_1\ldots x_{n-1}}\varphi(u)}h_1(u,\varphi(u)).
  \end{split}
\end{equation*}
Hence
$$
\lim_{\varepsilon \to 0}\left(\frac{\mathbf{I_1}}{\varepsilon}+
  \frac{\mathbf{I_2}}{\varepsilon}+ \frac{\mathbf{II}}{\varepsilon}\right)=
\sum_{\substack{\mathbf{x}\in 
2^{\{x_1,\ldots,x_n\}}
\\\mathbf{x}\neq
    \emptyset}}\sum_{i=\max\{1,2-n+|\mathbf{x}|\}}^{|\mathbf{x}|}\sum_{\substack{\mathbf{p}\in
    P^i(\mathbf{x})\\\mathbf{p}=(\mathbf{p}_1,\ldots,\mathbf{p}_i)}} \partial^n [\mathbf{x}^c,\mathbf{p}]
h_1(u,\varphi(u))
+\partial ^n_{x_1\ldots x_{n}}h_1(u,\varphi(u)),
$$
and,
by recalling \eqref{2017-05-23:04}, \eqref{2017-05-23:05}, we obtain
\begin{multline*}
  \lim_{\varepsilon \to 0} \left( I-
     \partial _{Y_1}h_1(u,\varphi(u))\right).\frac{\partial ^{n-1}_{x_1\ldots
      x_{n-1}}\varphi(u+\varepsilon x_n)- {\partial ^{n-1}_{x_1\ldots
        x_{n-1}}\varphi(u)} }{\varepsilon}
  \\
  =\sum_{\substack{\mathbf{x}\in 
2^{\{x_1,\ldots,x_n\}}
\\
\mathbf{x}\neq
      \emptyset}}\sum_{i=\max\{1,2-n+|\mathbf{x}|\}}^{|
\mathbf{x}|}\sum_{\substack{\mathbf{p}\in
      P^i(\mathbf{x})\\\mathbf{p}=(\mathbf{p}_1,\ldots,
      \mathbf{p}_i)}}
  \partial^n [\mathbf{x}^c,\mathbf{p}]h_1(u,\varphi(u)) +\partial ^n_{x_1\ldots
    x_{n}}h_1(u,\varphi(u)).
\end{multline*}
Finally, 
we can conclude the proof by recalling that $I- \partial _{Y_1}h_1(u,\varphi(u))$ is invertible with strongly continuous inverse.
\end{proof}

Theorem \ref{teo:derivabilita.punto.fisso} says that
 $\varphi$ is $Y_n$-valued,  continuous as a map from $U$ into $Y_n$,
and,
for $j=1,\ldots,n$,
for all  $u\in U$, 
  $x_1,\ldots,x_j\in X$, 
the directional derivative $\partial ^j_{x_1\ldots x_j}\varphi(u)$
exists, it belongs to $Y_{n-j+1}$, 
the map
 $$
 U\times X^j\to Y_{n-j+1},\ (u,x_1,\ldots,x_j) \mapsto 
\partial ^j_{x_1\ldots x_j}\varphi(u)
$$ 
is continuous,
  and
\eqref{eq:formula_fixed_point_derivatives} holds true.

Formula \eqref{eq:formula_fixed_point_derivatives} can be useful 
e.g.\ when considering
the boundedness of the derivatives of $\varphi$,
or when studying convergences of derivatives under perturbations of  $h$, as
 Corollary~\ref{corr:2012-04-20-aa}
and Proposition~\ref{propp:2012-05-23-aa} show.

\begin{corollary}\label{corr:2012-04-20-aa}
Let Assumption~\ref{2016-02-24:00} be satisfied.
Suppose that there exists
 \mbox{$M>0$} such that
\begin{equation}    \label{eq:2012-04-20-ab}
\hskip-8pt  \begin{dcases}
|  \partial _{y}h_k(u,y') | _k \leq M |y|_k&
\begin{dcases}
  \forall u\in U,\\
\forall y,y'\in Y_k,\ k=1,\ldots,n
\end{dcases}\\
|  \partial ^j_{x_1\ldots x_j}h_k(u,y) | _k \leq M
\prod_{l=1}^j
|x_l|_X
&
\begin{dcases}
  \forall u\in U,\ \forall x_1,\ldots,x_j\in X,\\
 \forall y\in Y_k,\ j,k=1,\ldots,n
\end{dcases}\\
|  \partial ^{j+i}_{ x_1\ldots x_j y_1\ldots y_i}h_k(u,y) | _k \leq M
\prod_{l=1}^j
|x_l|_X
\cdot \prod_{l=1}^i |y_l|_{k+1}
&
\begin{dcases}
  \forall u\in U,\ \forall x_1,\ldots,x_j\in X,\\
 \forall y\in Y_k, \ \forall y_1,\ldots,y_i\in Y_{k+1},\\
 k=1,\ldots,n-1,\\
 j,i=1,\ldots,n-1,\ 1\leq j+i\leq n-1.
\end{dcases}
\end{dcases}
\end{equation}
  Then, for $k=1,\ldots,n$, 
  $$
  \sup_{\substack{u\in U\\
      x_1,\ldots,x_k\in X\\
       |  x_1 |_X  =\ldots = |  x_k |_X  =1}}  |  \partial ^k_{x_1\ldots
    x_k}\varphi(u) | _{n-k+1} \leq C(\alpha,M),
  $$
where $C(\alpha,M)\in \mathbb{R}$ depends only on $\alpha$, $M$.
\end{corollary}
\begin{proof}
  Reason by induction taking into account
\eqref{eq:formula_fixed_point_derivatives}
and 
\eqref{eq:2016-04-20:03}.
\end{proof}

\begin{proposition}\label{propp:2012-05-23-aa}
Suppose that Assumption~\ref{2016-02-24:00}  holds true 
for a given $h_1$
and
 that $h_1^{(1)},h_1^{(2)},h_1^{(3)}\ldots$ is a sequence of functions, each of which satisfies
Assumption~\ref{2016-02-24:00}, uniformly with respect to the same $n$, $\alpha$.
Let
  $h^{(m)}_k$
denote  the map
associated to  $h^{(m)}_1$ 
according to
\eqref{eq:2016-04-20:04} and let
$\varphi^{(m)}$ denote the fixed-point map associated to the parametric $\alpha$-contraction $h^{(m)}_1$.
Suppose that the following convergences occur.
  \begin{enumerate}[(i)]
  \item\label{2016-04-21:04} For $k=1,\ldots,n$, $y\in Y_k$,
    \begin{equation}\label{eq:2012-05-23-ab}
      \lim_{m\rightarrow \infty}h_k^{(m)}(u,y)= h_k(u,y)\mbox{ in}\  Y_k
    \end{equation}
    uniformly for $u$ on compact subsets
    of $U$;
  \item\label{2016-04-21:05} for $k=1,\ldots,n$,
    \begin{equation}\label{eq:2012-05-23-ac}
      \begin{dcases}
        \lim_{m\rightarrow \infty}      \partial _xh_k^{(m)}(u,y)= \partial _xh_k(u,y)
&\hskip-8pt\mbox{ in } Y_k\\
\lim_{m\rightarrow \infty}      \partial _yh_k^{(m)}(u,y')
= \partial _yh_k(u,y')
&\hskip-8pt\mbox{ in } Y_k
\end{dcases}
\end{equation}
    uniformly for $u$ on compact subsets
    of $U$,
$x$ on compact subsets of $X$, and $y,y'$  on compact subsets of $Y_k$;
  \item\label{2016-04-21:06} for all $k=1,\ldots,n-1$, $u\in U$, $j,i=0,\ldots,n$, $1\leq j+i\leq n$,
    \begin{equation}\label{eq:2012-05-23-ad}
\lim_{m\rightarrow \infty}      \partial ^{j+i}_{x_1\ldots x_jy_1\ldots y_i}
h_k^{(m)}(u,y)=
\partial ^{j+i}_{x_1\ldots x_jy_1\ldots y_i}
h_k(u,y)
\mbox{ in } Y_k
    \end{equation}
    uniformly for $u$ on compact subsets
    of $U$,
$x_1,\ldots,x_j$ on compact subsets
    of $X$,
 $y$ on compact subsets of $Y_k$, $y_1,\ldots,y_i$ on
    compact subsets of $Y_{k+1}$.
  \end{enumerate}
  Then $\varphi^{(m)}\rightarrow \varphi$ uniformly on compact subsets of $Y_n$ and, for all $j=1,\ldots,n$
  \begin{equation}\label{eq:2012-05-23-ae}
\lim_{m\rightarrow \infty}    \partial ^j_{x_1\ldots x_j}\varphi^{(m)}(u)= \partial ^j_{x_1\ldots x_j}\varphi(u)
\mbox{ in } Y_{n-j+1}
  \end{equation}
  uniformly for $u$ on compact
  subsets of $U$ and $x_1,\ldots,x_j$ on compact subsets of $X$.
\end{proposition}
\begin{proof}
Notice that 
 \eqref{eq:2012-05-23-ab}
and the fact that each $h_k^{(m)}$
is a parametric $\alpha$-contraction (with the same $\alpha$) imply
 the uniform convergence $h_k^{(m)}\rightarrow h_k$ on compact subsets of $Y_k$.
In particular, the sequence $h_k^{(1)},h^{(2)}_k,h^{(3)}_k,\ldots$ is uniformly equicontinuous on compact sets.
Then, by 
Lemma~\ref{2016-02-25:02}\emph{(\ref{2016-02-23:02})},\emph{(\ref{2016-02-23:03})}, $\varphi^{(m)}\rightarrow \varphi$ in $Y_k$ uniformly on compact subsets of $Y_k$, for $k=1,\ldots,n$.
Moreover,
by
\eqref{eq:2016-04-20:03}, that holds
for all $h_1^{(m)}$
 uniformly in $m$,
we have the boundedness of
$(I- \partial _{Y_1}h_1^{(m)})^{-1}$, uniformly in $m$.
Convergence \eqref{eq:2012-05-23-ae}
is then obtained by  reasoning by induction on
\eqref{eq:formula_fixed_point_derivatives},
taking into account the strong continuity of 
$(I- \partial _{Y_1}h_1)^{-1}$.
\end{proof}

\section{Path-dependent  SDEs in Hilbert spaces}\label{2017-04-25:24}

In this section we study mild solutions of path-dependent SDEs in Hilbert spaces.
In particular,
by applying the results of the previous section, we address
differentiability with respect to the initial datum and stability of the derivatives.
By emulating the arguments of \cite[Ch.\ 7]{DaPrato2004} for the Markovian case and for differentiability up to order 2, we extend the results there provided
to 
the following path-dependent setting and 
to differentiability of generic order $n$.

\smallskip
Let  $H$ and $U$ be real separable Hilbert spaces, with scalar product denoted by $\langle \cdot,\cdot\rangle_H$ and $\langle \cdot,\cdot\rangle_U$,
respectively.
Let
 $ \mathfrak e\coloneqq \{e_n\}_{n\in \mathcal{N}}$
be an orthonormal basis of $H$,
 where $\mathcal{N}=\{1,\ldots,N\}$
if $H$ has dimension $N\in \mathbb{N}\setminus\{0\}$, or $\mathcal{N}=\mathbb{N}$ if $H$ has infinite dimension,
and let
$\mathfrak e'\coloneqq \{e'_m\}_{m\in \mathcal{M}}$
be an orthonormal basis of $U$,
 where $\mathcal{M}=\{1,\ldots,M\}$ if $U$ has dimension $M\in \mathbb{N}\setminus \{0\}$, or $\mathcal{M}=\mathbb{N}$ if $U$ has infinite dimension.
If $\mathbf{x}\colon [0,T]\rightarrow \mathcal{S}$ is a function taking values in any set $\mathcal{S}$ and if $t\in[0,T]$, we denote by $\mathbf{x}_{t\wedge \cdot}$ the function defined by
\begin{equation*}
  \begin{dcases}
    \mathbf{x}_{t\wedge \cdot}(s)\coloneqq\mathbf{x}(s)&s\in[0,t]\\
    \mathbf{x}_{t\wedge \cdot}(s)\coloneqq \mathbf{x}(t)&s\in(t,T].
  \end{dcases}
\end{equation*}
For elements of stochastic analysis in infinite dimension   used hereafter, we refer to \cite{DaPrato2014,Gawarecki2011}.

\smallskip
We begin by considering the SDE
\begin{equation}\label{2016-03-01:00}
\begin{dcases}
  d X_s = \left(AX_s + b\left((\cdot,s),X\right)\right)  d t+ \sigma\left((\cdot,s),X\right)  d W_s & s\in(t, T]\\
X_s   = Y_s & s\in [0,t],
\end{dcases}
\end{equation}
where $t\in[0,T]$,
$Y$ is a $H$-valued 
adapted process defined on a complete filtered probability space
$(\Omega,\mathcal{F},\mathbb{F}\coloneqq\{\mathcal{F}_t\}_{t\in[0,T]},\mathbb{P})$,
 $W$ is a $U$-valued cylindrical Wiener process defined on  $(\Omega,\mathcal{F},\mathbb{F},\mathbb{P})$,
 $b((\omega,s),X)$ is a $H$-valued random variable depending on $\omega\in \Omega$, on the time $s$, and on the path $X$, $\sigma((\omega,s),X)$ is a $L_2(U,H)$-valued random variable depending on $\omega\in\Omega$, on  the time $s$, and on the path $X$, and $A$ is the generator of a $C_0$-semigroup $S$ on $H$.



\medskip

We introduce the following notation:
  \begin{itemize}
\itemsep=-1mm
\item\label{2016-04-13:09} $\paths$ denotes a  closed
 subspace of $\Bb{H}$ (\footnote{We recall that $\Bb{H}$ is endowed with the norm $|\cdot|_\infty$.}) such that
\begin{equation}
    \label{eq:2016-03-07:01}
\hskip-0.4cm      \begin{dcases}
(a)\ \Cb{H}\subset \paths\\
(b)\ \mathbf{x}_{t\wedge \cdot}\in \paths, \ \forall \mathbf{x}\in \paths, \ \forall t\in[0,T]\\
(c)\  \mbox{for all }T\in L(H)\mbox{ and }\mathbf{x}\in \paths, \ \mbox{the map}\ [0,T]\rightarrow H,\ t \mapsto T\mathbf{x}_t,\ \mbox{belongs to }\paths.
\end{dcases}
\end{equation}
Hereafter, unless otherwise specified, $\paths$ will be always considered as a Banach space endowed with the norm $|\cdot|_\infty$.
For example, $\paths$ could be $\Cb{H}$,
the space of c\`adl\`ag functions $[0,T]\rightarrow H$,
or  $\Bb{H}$ itself.

\item $\Omega_T$ denotes the product space $ \Omega\times [0,T]$ and $\mathcal{P}_T$ denotes the product measure $\mathbb{P} \otimes m$ on $(\Omega_T,\mathcal{F}_T \otimes \mathcal{B}_{[0,T]})$, where $m$ is the Lebesgue measure and $\mathcal{B}_{[0,1]}$ is the Borel $\sigma$-algebra on $[0,1]$.

\item
 $\mathcal{L}^0_{\mathcal{P}_T}(\paths)$ denotes the space of 
functions
$X\colon \Omega_T\rightarrow H$
 such that
 \begin{equation}
   \label{eq:2016-03-07:02}
   \begin{dcases}
(a)\ 
\forall \omega\in \Omega,\ \mbox{the map}\
[0,T]\rightarrow H,\
t \mapsto  X_t(\omega),\ \mbox{belongs to }\ \paths\\
(b)\   (\Omega_T,\mathcal{P}_T)\rightarrow \paths,\ (\omega,t)  \mapsto  X_{t\wedge \cdot}(\omega)\mbox{ is measurable.}
\end{dcases}
\end{equation}
Two processes $X,X'\in \mathcal{L}^0_{\mathcal{P}_T}(\paths)$ are equal if and only if \mbox{$\mathbb{P}(|X-X'|_\infty=0)=1$}.
\item
For $p\in [1,\infty)$,
 $\Lpaths$ denotes the space of  
equivalence classes
of
functions
$X\in \mathcal{L}^0_{\mathcal{P}_T}(\paths)$
such that
$\Omega_T\rightarrow \paths,\ (\omega,t)  \mapsto  X_{t\wedge \cdot}(\omega)$ has separable range and
\begin{equation}
  \label{eq:2016-03-07:03}
  |X|_{\Lpaths}\coloneqq  \left( \mathbb{E} \left[ |X|_\infty^p \right]  \right) ^{1/p}<\infty.
\end{equation}
\item For $p,q\in[1,\infty)$ and $\beta\in [0,1)$,
 $\Lambda^{p,q,p}_{\mathcal{P}_T,S,\beta}(L(U,H))$ denotes the space of functions 
$\Phi\colon \Omega_T\rightarrow L(U,H)$ such that
\begin{equation*}
  \begin{dcases}
    \Phi u\colon (\Omega_T,\mathcal{P}_T)\rightarrow H,\ (\omega,t) \mapsto \Phi_t(\omega)u,\mbox{ is measurable,\ }\forall u\in U\\
|\Phi|_{p,q,S,\beta}\coloneqq \left( 
\int_0^T \left( \int_0^t
(t-s)^{-\beta q}
 \left(  \mathbb{E} \left[ |S_{t-s}\Phi_s|_{L_2(U,H)}^p \right]  \right) ^{q/p}ds \right)^{p/q} dt
 \right) ^{1/p}<\infty.
  \end{dcases}
\end{equation*}
The space $\Lambda^{p,q,p}_{\mathcal{P}_T,S,\beta}(L(U,H))$ is normed by
 $|\cdot|_{p,q,S,\beta}$ (see 
 Remark \ref{2016-03-31:02}  below).
\item 
 $\overline \Lambda ^{p,q,p}_{\mathcal{P}_T,S,\beta}(L(U,H))$
denotes
the completion of $\Lambda^{p,q,p}_{\mathcal{P}_T,S,\beta}(L(U,H))$. 
We keep the notation $|\cdot|_{p,q,S,\beta}$ for the extended norm.
\end{itemize}

It can be seen that
 $(\Lpaths,|\cdot|_{\Lpaths})$ is a Banach space ($\mathbb{F}$ is supposed to be complete).

\begin{remark}
  \label{2016-03-31:02}
  To see that $|\cdot|_{p,q,S,\beta}$ is a norm and not just a seminorm, suppose that $|\Phi|_{p,q,S,\beta}=0$. In particular, for $u\in U$,
$$
\int_{[0,T]^2}\mathbf{1}_{(0,T]}(t-s)(t-s)^{-\beta}\mathbb{E} \left[ |S_{t-s}\Phi_s u|_{H} \right] ds  \otimes  dt=0,
$$
which entails,  for $ \mathbb{P} \otimes m$-a.e.\ $(\omega,s)\in\Omega_T$, 
\begin{equation}
  \label{eq:2016-03-31:03}
  S_{t-s}\Phi_s(\omega)u=0\qquad m\mbox{-a.e.\ }t\in(s,T].
\end{equation}
Since $S$ is strongly continuous, \eqref{eq:2016-03-31:03} gives
$$
\Phi_s(\omega)u=0\qquad \mathbb{P} \otimes m\mbox{-a.e.\ }(\omega,s)\in\Omega_T,
$$
which provides $\Phi=0$ $\mathbb{P} \otimes m$-a.e., since $U$ is supposed to be separable and $\Phi_s(\omega)\in L(U,H)$ for all $\omega,s$.
\end{remark}

\begin{remark}\label{2016-04-07:04}
  The space $\overline \Lambda ^{p,q,p}_{\mathcal{P}_T,S,\beta}(L(U,H))$ can be naturally identified with a closed subspace of the space of all those  
measurable functions
$$
\zeta\colon (\Omega_T\times [0,T],\mathcal{P}_T \otimes \mathcal{B}_T)\rightarrow L_2(U,H)
$$
such that
$$
\begin{dcases}
\zeta((\omega,s),t)=0,\   \forall ((\omega,s),t)\in \Omega_T\times [0,T],\ s>t,  \\
  |\zeta|_{p,q,p}\coloneqq 
\left( 
\int_0^T \left( \int_0^t
 \left(  \mathbb{E} \left[ |\zeta((\cdot,s),t)|_{L_2(U,H)}^p \right]  \right) ^{q/p}ds \right)^{p/q} dt
 \right) ^{1/p}<\infty.&
\end{dcases}
$$
Indeed, if we denote by $L^{p,q,p}_{\mathcal{P}_T \otimes \mathcal{B}_T}(L_2(U,H))$ such a space, then $L^{p,q,p}_{\mathcal{P}_T \otimes \mathcal{B}_T}(L_2(U,H))$ endowed with $|\cdot|_{p,q,p}$ is a Banach space and the map
$$
\iota\colon \Lambda^{p,q,p}_{\mathcal{P}_T,S,\beta}(L(U,H))\rightarrow L^{p,q,p}_{\mathcal{P}_T \otimes \mathcal{B}_T}(L_2(U,H))
$$
defined by
$$
\iota(\Phi)(\omega,s,t)\coloneqq
\begin{dcases}
  (t-s)^{-\beta}S_{t-s}\Phi_s(\omega)&\forall ((\omega,s),t)\in \Omega_T\times [0,T],\ s\leq t,\\
0 & \mbox{otherwise}.
\end{dcases}
$$
is an isometry.
\end{remark}

The reason to introduce the space $\overline
\Lambda^{p,q,p}_{\mathcal{P}_T,S,\beta}(L(U,H))$ is related to the
existence of a continuous version of the stochastic convolution and to
the factorization method used to construct such a version.  Let 
$p>\max\{2,1/\beta\} $,
 $t\in[0,T]$, and $\Phi\in
\Lambda^{p,2,p}_{\mathcal{P}_T,S,\beta}(L(U,H))$. 
If we consider
the two stochastic convolutions
\begin{equation}
  \label{eq:2016-03-23:03}
  Y_{t'}\coloneqq \mathbf{1}_{[t,T]}(t')\int_t^{t'}S_{t'-s}\Phi_sdW_s,\quad
Z_{t'}\coloneqq \mathbf{1}_{[t,T]}(t')\int_t^{t'}(t'-s)^{-\beta}S_{t'-s}\Phi_sdW_s,
\end{equation}
then $Y_{t'}$ is well-defined for all $t'\in[0,T]$,
$Z_{t'}$ is well-defined for $m$-a.e.\ $t'\in[0,T]$, and 
$Y_{t'},Z_{t'}$
 belong to $L^p((\Omega,\mathcal{F}_{t'},\mathbb{P}),H)$.
By 
 using the stochastic Fubini theorem and the factorization method  (see \cite{DaPrato2014}),
we can find a predictable process $\widetilde Z$ such that:
\begin{enumerate}[(a)]
\item\label{2016-03-23:04} for $m$-a.e.\ $t\in[0,T]$, $\widetilde Z_t=Z_t$ $\mathbb{P}$-a.e.;
\item for all $t'\in[0,T]$,
 the following  formula holds
  \begin{equation}
    \label{eq:2016-03-23:01}
    Y_{t'}=c_\beta
    \mathbf{1}_{[t,T]}(t')
    \int_t^{t'} (t'-s)^{\beta-1}\widetilde Z_sds\qquad \mathbb{P}\mbox{-a.e.,}
  \end{equation}
where $c_\beta$ is a constant depending only on $\beta$.
\end{enumerate}
\noindent By \eqref{eq:2016-03-23:03}, \eqref{2016-03-23:04}, \cite[Lemma 7.7]{DaPrato1992},
 it follows that $\widetilde Z(\omega)\in L^p((0,T),H)$ for $\mathbb{P}$-a.e.\ $\omega\in\Omega$, hence, by \cite[Lemma 3.2]{Gawarecki2011}, the right-hand side of 
\eqref{eq:2016-03-23:01} is continuous in $t'$.

This classical argument shows that there exists a pathwise continuous process $S\sconv \Phi$ such that, for all $t'\in[0,T]$, $(S\sconv\Phi)_{t'}=Y_{t'}$ $\mathbb{P}$-a.e..
In particular,
$S\sconv \Phi\in \mathcal{L}^0_{\mathcal{P}_T}(\Cb{H})$.
By
\eqref{eq:2016-03-23:03}, \eqref{eq:2016-03-23:01}, 
H\"older's inequality,
 and \cite[Lemma 7.7]{DaPrato1992}, we also have 
\begin{equation}
  \label{eq:2016-04-06:00}
  \mathbb{E} \left[ |S\sconv \Phi|_\infty^p \right] \leq c_\beta^p \left( \int_0^Tv^{\frac{(\beta-1)p}{p-1}}dv\right)^{p-1}
\mathbb{E} \left[ \int_0^T|\widetilde Z_s|_H^pds \right] 
\leq
c'_{\beta,T,p}|\Phi|_{p,2,S,\beta}^p,
\end{equation}
where $c'_{\beta,T,p}$ is a constant depending only on $\beta,T,p$.
This shows that the linear map 
$S\sconv\#$, defined as
\begin{equation}
  \label{eq:2016-03-23:05}
  \Lambda^{p,2,p}_{\mathcal{P}_T,S,\beta}(L(U,H))\rightarrow \mathcal{L}^p_{\mathcal{P}_T}(\Cb{H}),\ \Phi \mapsto S\sconv \Phi,
\end{equation}
is well-defined and continuous.
Then, we can uniquely extend \eqref{eq:2016-03-23:05} to a continuous linear map on $\overline \Lambda^{p,2,p}_{\mathcal{P}_T,S,\beta}(L(U,H))$, that
we can see as $\Lpaths$-valued, since, by assumption, $\Cb{H}\subset \paths$. We end up with a continuous linear map,
again denoted by $S\sconv \#$,
\begin{equation}
  \label{eq:2016-03-23:06}
 S\sconv \#\colon \overline \Lambda^{p,2,p}_{\mathcal{P}_T,S,\beta}(L(U,H))\rightarrow 
 \Lpaths.
\end{equation}
Summarizing,
\begin{enumerate}[(1)]
\item the map $S\sconv\#$ is linear, continuous, $\mathcal{L}^p_{\mathcal{P}_T}(C([0,T],H))$-valued;
\item the operator norm of $S\sconv\#$ depends only on $\beta,T,p$;
\item if $\Phi\in \Lambda^{p,2,p}_{\mathcal{P}_T,S,\beta}(L_2(U,H))$, $S\sconv \Phi$ is a continuous version of the process $Y$ in \eqref{eq:2016-03-23:03}.
\end{enumerate}

Within the approach using the factorization method,
the space $\overline \Lambda^{p,2,p}_{\mathcal{P}_T,S,\beta}(L(U,H))$ 
is then naturally introduced
if we want to see the stochastic convolution as a continuous linear operator acting on a Banach space and providing pathwise continuous processes,
and this perspective is useful when applying
to SDEs
the results based on parametric $\alpha$-contractions obtained in the first part of the paper.

\medskip
We make some observations that will be useful later.
Let $\hat S$ be another $C_0$-semigroup on $H$, and let 
 $\Phi\in  \Lambda^{p,2,p}_{\mathcal{P}_T,S,\beta}(L(U,H))$,
$\hat \Phi  \in\Lambda^{p,2,p}_{\mathcal{P}_T,\hat S,\beta}(L(U,H))$.
Then, by using the
 factorization formula \eqref{eq:2016-03-23:01}
 both with respect to the couples $(S,\Phi)$ and  $(\hat S,\hat \Phi)$, and by an estimate analogous to  \eqref{eq:2016-04-06:00},
we obtain
\begin{equation}
    \label{eq:2016-04-06:01}
    \begin{multlined}[c][.85\displaywidth]
       \mathbb{E} \left[ |S\sconv \Phi- \hat S\sconv \hat \Phi|^p_\infty \right] \leq\\
\leq c'_{\beta,T,p}
\int_0^T
 \left( 
\int_0^t
(t-s)^{-2\beta}
 \left( 
\mathbb{E} \left[ 
|S_{t-s}\Phi_s -\hat S_{t-s}\hat \Phi_s|_{L_2(U,H)}^p
 \right] 
 \right) ^{2/p}
ds
 \right) ^{p/2}
dt.
\end{multlined}
\end{equation}
For $0\leq t_1\leq t_2\leq T$ and $\Phi\in \Lambda^{p,2,p}_{\mathcal{P}_T,S,\beta}(L(U,H))$, we also have
\begin{equation}
  \label{eq:2016-04-20:05}
  (S\sconvt{t_1} \Phi-
S\sconvt{t_2} \Phi)_s
=
\mathbf{1}_{[t_1,t_2]}(s)(S\sconvt{t_1} \Phi)_s
+\mathbf{1}_{(t_2,T]}(s)S_{s-t_2} (S\sconvt{t_1} \Phi)_{t_2}\qquad \forall s\in [0,T].
\end{equation}
\bigskip
Since
$$
\sup_{s\in[t_1,t_2]}
 |(S\sconvt{t_1} \Phi)_s|_H
\leq
 |S\sconvt{t_1} (\mathbf{1}_{[t_1,t_2]}(\cdot)\Phi)|_\infty\qquad \mathbb{P}\mbox{-a.e.},
$$
we obtain,
by \eqref{eq:2016-04-06:00},
\begin{equation}
  \label{2016-04-16:00}
\lim_{t_2-t_1\rightarrow 0^+} 
\mathbb{E} \left[ 
\sup_{s\in[t_1,t_2]}
|(S\sconvt{t_1} \Phi)_s|^p _H \right] 
\leq
\lim_{t_2-t_1\rightarrow 0^+} 
 c'_{\beta,T,p}|\mathbf{1}_{[t_1,t_2]}(\cdot)\Phi|_{p,2,S,\beta}^p=0,
\end{equation}
where the latter limit can be seen by applying Lebesgue's dominated convergence theorem three times, to the three integrals defining $|\cdot|_{p,2,S,\beta}$.
Actually, since the linear map
$$
\overline \Lambda^{p,2,p}_{\mathcal{P}_T,S,\beta}(L(U,H))\rightarrow \overline\Lambda^{p,2,p}_{\mathcal{P}_T,S,\beta}(L(U,H)),\ \Phi\rightarrow \mathbf{1}_{[t_1,t_2]}(\cdot)\Phi
$$
is bounded, uniformly in $t_1,t_2$, the limit
\eqref{2016-04-16:00} is uniform for $\Phi$ in compact subsets of $\overline \Lambda^{p,2,p}_{\mathcal{P}_T,S,\beta}(L(U,H))$ and $t_1,t_2\in[0,T]$, 
 $t_2-t_1\rightarrow 0^+$.
Then, by \eqref{eq:2016-04-20:05} and \eqref{2016-04-16:00}, we finally obtain
\begin{equation}
  \label{eq:2016-04-16:01}
\lim_{|t_2-t_1|\rightarrow 0}|S\sconvt{t_1} \Phi-
S\sconvt{t_2} \Phi|_{\Lpaths}=0
\end{equation}
uniformly for
$\Phi$ in compact subsets of $\overline \Lambda^{p,2,p}_{\mathcal{P}_T,S,\beta}(L(U,H))$.
In particular,
thanks to the
uniform boundedness of
 $\{S\sconv \#\}_{t\in[0,T]}$ (see \eqref{eq:2016-04-06:00}),
the map
 \begin{equation}
   \label{eq:2016-04-16:02}
   [0,T]\times \overline \Lambda^{p,2,p}_{\mathcal{P}_T,S,\beta}(L(U,H))\rightarrow \Lpaths,\ (t,\Phi) \mapsto S\sconv \Phi
 \end{equation}
is continuous.

\subsection{Existence and uniqueness of mild solution}\label{2017-04-28:02}

The following assumption will be standing for the remaining part of this manuscript.
We recall that, if $E$ is a Banach space, then $\mathcal{B}_E$ denotes its Borel $\sigma$-algebra.

\begin{assumption}\label{2016-03-24:08}
${}$
  \begin{enumerate}[(i)]
  \item\label{2016-03-24:09} $b\colon (\Omega_T\times \paths,\mathcal{P}_T
 \otimes \mathcal{B}_\paths)\rightarrow (H,\mathcal{B}_H)$ is measurable; 
\item\label{2016-03-24:12} $\sigma\colon (\Omega_T\times \paths,\mathcal{P}_T
 \otimes \mathcal{B}_\paths)\rightarrow
L(U,H)$ is strongly measurable, that is 
$(\Omega_T\times \paths,\mathcal{P}_T
 \otimes \mathcal{B}_\paths)\rightarrow
H, \ ((\omega,t),\mathbf{x}) \mapsto \sigma((\omega,t),\mathbf{x})u$ is measurable, for all $u\in U$;
\item\label{2016-03-24:10} (non-anticipativity condition)
for all $ ((\omega,t),\mathbf{x})\in \Omega_T \times \paths$,
$ b((\omega,t),\mathbf{x})=b((\omega,t),\mathbf{x}_{t\wedge \cdot})$
and
$  \sigma((\omega,t),\mathbf{x})=\sigma((\omega,t),\mathbf{x}_{t\wedge \cdot})$;
\item\label{2016-03-24:11} there exists $g\in L^1((0,T),\mathbb{R})$ such that
  \begin{equation*}
    \begin{dcases}
      |b((\omega,t),\mathbf{x})|_H\leq g(t)(1+|\mathbf{x}|_\infty)& \forall ((\omega,t),\mathbf{x})\in\Omega_T\times \paths,\\
      |b((\omega,t),\mathbf{x})-
b((\omega,t),\mathbf{x}')|_H\leq g(t)|\mathbf{x}-\mathbf{x}'|_\infty& \forall (\omega,t)\in\Omega_T,\ \forall \mathbf{x},\mathbf{x}'\in \paths;
    \end{dcases}
  \end{equation*}
\item\label{2016-03-24:14} there exist $M>0$, $\gamma\in [0,1/2)$ 
such that
  \begin{equation*}
    \begin{dcases}
      |S_t\sigma((\omega,s),\mathbf{x})|_{L_2(U,H)}\leq M t^{-\gamma}(1+|\mathbf{x}|_\infty)& \forall ((\omega,s),\mathbf{x})\in\Omega_T\times \paths,\ \forall t\in (0,T],\\
      |S_t\sigma((\omega,s),\mathbf{x})-
S_t\sigma((\omega,s),\mathbf{x}')|_{L_2(U,H)}\leq M t^{-\gamma}|\mathbf{x}-\mathbf{x}'|_\infty& \forall (\omega,s)\in\Omega_T,
\ \forall t\in(0,T],
\ \forall\mathbf{x},\mathbf{x}'\in \paths.
    \end{dcases}
  \end{equation*}
\end{enumerate}
\end{assumption}

\medskip
\begin{remark}
Assumption \ref{2016-03-24:08}\emph{(\ref{2016-03-24:11})} could be generalized to the form
\begin{equation*}
  \label{eq:2017-05-11:03}
    \begin{dcases}
      |S_tb((\omega,s),\mathbf{x})|_H\leq t^{-\gamma}g(s)(1+|\mathbf{x}|_\infty)& \forall ((\omega,s),\mathbf{x})\in\Omega_T\times \paths,\ \forall t\in(0,T]\\
      |S_t(b((\omega,s),\mathbf{x})-
b((\omega,s),\mathbf{x}'))|_H\leq t^{-\gamma}g(s)|\mathbf{x}-\mathbf{x}'|_\infty& \forall (\omega,s)\in\Omega_T,
\ \forall t\in(0,T],
\ \forall \mathbf{x},\mathbf{x}'\in \paths, 
    \end{dcases}
\end{equation*}
with $g$ suitably integrable,
and similarly for Assumption \ref{2016-03-24:08}\emph{(\ref{2016-03-24:14})}.
The results obtained and the methods used
hereafter 
 can be adapted to cover these more general assumptions.
\end{remark}

\medskip
\begin{definition}[Mild solution]
Let $Y\in \mathcal{L}^0_{\mathcal{P}_T}(\paths)$ and $t\in [0,T)$.
A function $X\in \mathcal{L}^0_{\mathcal{P}_T}(\paths)$ is 
 a \emph{mild solution} to
\eqref{2016-03-01:00} if,
for all $t'\in [t,T]$,
  \begin{equation*}
    \mathbb{P} \left( \int_t^{t'}|S_{t-s}b(\cdot,s,X)|_Hds
      +
      \int_t^{t'}|S_{t-s}\sigma(\cdot,s,X)|^2_{L_2(U,H)}ds
     <\infty\right)=1,
  \end{equation*}
and
\begin{equation*}
  \begin{dcases}
    \forall t'\in [0,t],&
  X_{t'}=Y_{t'}\ \mathbb{P}\mbox{-a.e.,}\\
\forall t'\in(t,T],&
  X_{t'}=S_{t'-t}Y_{t}
  +\int_t^{t'}S_{t'-s}b((\cdot,s),X)ds
  +\int_t^{t'}S_{t'-s}\sigma((\cdot,s),X)dW_s
\ \mathbb{P}\mbox{-a.e..}
\end{dcases}
\end{equation*}
\end{definition}

\medskip

Using a classical contraction argument, we are going to 
prove existence and uniqueness of mild solution in the space $\Lpaths$,
when the initial datum $Y$ belongs to $ \Lpaths$, 
for $p$ large enough.
This will let us
 apply the theory developed in Section~\ref{sec:recalls-dif-banach}.

\medskip
  For $t\in[0,T]$ and
 $$
p>p^*\coloneqq \frac{2}{1-2\gamma},
\ \beta\in (1/p,1/2-\gamma),
$$
 we define the following maps:
\begin{subequations}
  \begin{equation*}\label{2016-03-24:03}
    \id^S_t\colon \Lpaths\rightarrow \Lpaths,\ Y \mapsto
\mathbf{1}_{[0,t]}(\cdot)
 Y+\mathbf{1}_{(t,T]}(\cdot)S_{\cdot-t}Y_t
  \end{equation*}
  \begin{equation*}
    \label{eq:2016-03-24:04}
    F_b\colon \Lpaths\rightarrow L^{p,1}_{\mathcal{P}_T}(H),\ X \mapsto  b((\cdot,\cdot),X)
  \end{equation*}
  \begin{equation*}
    \label{eq:2016-03-24:05}
    F_\sigma\colon \Lpaths\rightarrow \overline \Lambda^{p,2,p}_{\mathcal{P}_T,S,\beta}(L(U,H)),\ X \mapsto \sigma((\cdot,\cdot),X)
  \end{equation*}
  \begin{equation*}
    \label{eq:2016-03-24:06}
    S\conv \#\colon L^{p,1}_{\mathcal{P}_T}(H)\rightarrow \Lpaths,\ X \mapsto 
\mathbf{1}_{[t,T]}(\cdot)\int_t^\cdot  S_{\cdot-s}X_sds,
\end{equation*}
and we recall the map
\begin{equation*}
  \label{eq:2016-03-24:07}
  S\sconv \#\colon \overline \Lambda^{p,2,p}_{\mathcal{P}_T,S,\beta}(L(U,H))\rightarrow \Lpaths,\ \Phi \mapsto S\sconv \Phi.
\end{equation*}
\end{subequations}
Then $\id^S_t$ is well-defined, due to ($a$) and ($b$) in \eqref{eq:2016-03-07:01}, because we can write
\begin{equation}
  \label{eq:2016-04-21:01}
  \id^S_t  (Y)=Y_{t\wedge \cdot}+\mathbf{1}_{(t,T]}(\cdot)(S_{\cdot-t}-I)Y_t.
\end{equation}
As regarding $F_b$,
by Assumption \ref{2016-03-24:08}\emph{(\ref{2016-03-24:09})},\emph{(\ref{2016-03-24:10})}, and by 
($b$) in
\eqref{eq:2016-03-07:01},
the map
$$
\Omega_T\rightarrow  H,\ (\omega,t) \mapsto  b((\omega,t),X(\omega))
=
b((\omega,t),X_{t\wedge \cdot}(\omega))
$$
is predictable.
Moreover, by 
Assumption \ref{2016-03-24:08}\emph{(\ref{2016-03-24:11})}, we have
\begin{equation*}
 \int_0^T  \left( \mathbb{E} \left[ |b(\cdot,t,X_{t\wedge \cdot})|^p \right]  \right) ^{1/p}dt\leq  
 \int_0^T  g(t)\left( \mathbb{E} \left[ (1+|X|_\infty)^p \right]  \right) ^{1/p}dt\leq  |g|_{L^1((0,T),\mathbb{R})}(1+|X|_{\Lpaths}),
\end{equation*}
which shows that $F_b(X)\in L^{p,1}_{\mathcal{P}_T}(H)$.
By Assumption \ref{2016-03-24:08}\emph{(\ref{2016-03-24:11})}, we also have that $F_b$ is Lipschitz, with Lipschitz constant dominated by $|g|_{L^1((0,1),\mathbb{R})}$.
Similarly as done for $F_b$, 
by using Assumption \ref{2016-03-24:08}\emph{(\ref{2016-03-24:12})},
one can see that, for $X\in \Lpaths$, the map
$$
(\Omega_T,\mathcal{P}_T)\rightarrow L(U,H),\ (\omega,t) \mapsto  \sigma((\omega,t),X_{t\wedge \cdot}(\omega))
$$
is strongly measurable.
Moreover, by
Assumption \ref{2016-03-24:08}\emph{(\ref{2016-03-24:14})}, we have
\begin{equation*}
  \begin{split}
|F_\sigma(X)|_{p,2,S,\beta}&=   \left(   \int_0^T  \left(\int_0^t (t-s)^{-\beta 2}
      \left( \mathbb{E} \left[ |S_{t-s} \sigma((\cdot,s),X_{s\wedge \cdot})|_{L_2(U,H)}^p \right] \right) ^{2/p}ds
       \right) ^{p/2}dt \right) ^{1/p}\\
&\leq M
    \left(  \int_0^T  \left(\int_0^t v^{-(\beta+\gamma) 2}
      dv
       \right) ^{p/2}dt \right) ^{1/p}
(1+|X|_{\Lpaths})
  \end{split}
\end{equation*}
and the latter term is finite because $\beta<1/2-\gamma$ and $X\in \Lpaths$.
Then $F_\sigma$ is well-defined.
With similar computations, we have that $F_\sigma$ is Lipschitz, with Lipschitz constant depending only on $M$, $\beta$, $\gamma$, $p$.
Regarding $S\conv \#$, 
 if $X\in L^{p,1}_{\mathcal{P}_T} (H)$, then $X(\omega)\in L^1((0,T),H)$ for $\mathbb{P}$-a.e.\ $\omega\in \Omega$, hence  it is easily checked that
$$
[0,T]\rightarrow H,\  t' \mapsto \mathbf{1}_{[0,t]}(t')\int_t^{t'}S_{t'-s}X_s(\omega)ds
$$
is continuous, and then it belongs to $\paths$.
Since $\mathbb{F}$ is complete, we can assume that \mbox{$S\conv X(\omega)$}  is continuous for all $\omega$, hence it is predictable, because it is $\mathbb{F}$-adapted.
Since the trajectories are continuous, we also have the measurability of 
$$
(\Omega_T,\mathcal{P}_T)\rightarrow \Cb{H}\subset \paths,\ (\omega,t') \mapsto  (S\conv X)_{t'\wedge \cdot}(\omega).
$$
Then, to show that $S\conv X\in \Lpaths$, it remains to verify the integrability condition.
We have
\begin{equation*}
  |S\conv X|_{\Lpaths}\leq M' \left( \mathbb{E} \left[ \left(  \int_0^T|X_s|_Hds \right) ^p \right] \right) ^{1/p}
\leq M'\int_0^T  \left( \mathbb{E} \left[ |X_s|_H^p \right]  \right) ^{1/p}ds=M' |X|_{p,1},
\end{equation*}
where 
$$
\begin{minipage}{0.8\linewidth}
  \begin{center}
      $M'$ is any upper bound for $ {\displaystyle\sup_{t\in[0,T]}|S_t|_{L(H)}}$.
    \end{center}
  \end{minipage}
$$
The good definition of $S\sconv \#$ was discussed above (observe that $p>\max\{2,1/\beta\}$).

\smallskip
We can then build the map 
\begin{equation}
  \label{eq:2016-03-24:00}
  \psi\colon \Lpaths\times \Lpaths
  \rightarrow \Lpaths,\ (Y,X) \mapsto \id^S_t(Y)+S\conv F_{b}(X)+S\sconv F_{\sigma}(X).
\end{equation}
In what follows, whenever we need to make explicit the dependence of $\psi(Y,X)$ on the data $t,S,b,\sigma$, we  write $\psi(Y,X;t,S,b,\sigma)$.

\smallskip

We first show that, for ech $Y\in \Lpaths$, $\psi(Y,\cdot)$ has a unique fixed point $X$.
Such a fixed point is a mild solution to 
\eqref{2016-03-01:00}.

\smallskip
The advantage of introducing the setting above is that it permits to see $\psi$ as a composition of maps
that have
different 
regularity and that can be considered individually
when studying
the regularity of the mild solution $X^{t,Y}$ with respect to $Y$ or
 the dependence of $X^{t,Y}$  with respect to a perturbation of the data $Y,t,S,b,\sigma$.

\smallskip
For $\lambda>0$, we consider the following norm on $\Lpaths$
$$
|X|_{\Lpaths,\lambda}\coloneqq  \left( \mathbb{E} \left[ \sup_{t\in[0,T]}e^{-\lambda p t}|X_t|^p \right]  \right) ^{1/p}\qquad \forall X\in \Lpaths.
$$
Then $|\cdot|_{\Lpaths, \lambda}$ is equivalent to $|\cdot|_{\Lpaths}$.

\medskip
We proceed to show that there exists $\lambda>0$ 
such that $\psi$ is a  parametric contraction.

\smallskip
For $X,X'\in \Lpaths$, $\lambda>0$, and $t'\in[0,T]$, we have
\begin{equation*}
  \begin{split}
    e^{-\lambda p t'}|
    (S\conv F_b(X))_{t'}-(S\conv F_b(X'))_{t'}|_H^p
    &\leq 
    (M')^p
   \left(   \int_0^{t'}e^{-\lambda t'}|b((\cdot,s),X)-b((\cdot,s),X')|_H ds \right) ^p\\
&\leq (M')^p
 \left( \int_0^{t'} e^{-\lambda (t'-s)}g(s) e^{-\lambda s}|X_{s\wedge \cdot}-X'_{s\wedge \cdot}|_\infty ds \right) ^p\\
&\leq C_{\lambda,g,M'}^p\sup_{s\in[0,T]}\left\{e^{-\lambda ps}|X_s-X'_s|_H^p\right\},
  \end{split}
\end{equation*}
where ${\displaystyle C_{\lambda,g,M'}\coloneqq M'\sup_{t'\in[0,T]}\int_0^{t'}e^{-\lambda v}g(t'-v)dv }$.
We then obtain
\begin{equation}
  \label{eq:2016-03-24:15}
  |S\conv F_b(X)-S\conv F_b(X')|_{\Lpaths,\lambda}\leq C_{\lambda,g,M'} |X-X'|_{\Lpaths, \lambda}.
\end{equation}
It is not difficult to see that
$C_{\lambda,g,M'} \rightarrow 0$ as $\lambda \rightarrow \infty$.

Now, if $\Phi\in \Lambda^{p,2,p}_{\mathcal{P}_T,S,\beta}(L(U,H))$, then $e^{-\lambda \cdot}\Phi\in \Lambda^{p,2,p}_{\mathcal{P}_T,e^{-\lambda \cdot}S,\beta}(L(U,H))$ for all $\lambda\geq 0$ and,
for $\mathbb{P}$-a.e.\ $\omega\in \Omega$, 
\begin{equation}
  \label{eq:2016-03-25:00}
  e^{-\lambda t'}(S\sconv \Phi)_{t'}=((e^{-\lambda \cdot}S)\sconv (e^{-\lambda \cdot} \Phi))_{t'}\qquad \forall t'\in[0,T].
\end{equation}
For  $X\in \Lpaths$, we have
$$
\int_{t}^{t'}\mathbb{E} \left[|e^{-\lambda (t'-s)}S_{t'-s}(e^{-\lambda \cdot}F_\sigma(X))_s|_{L_2(U,H)}^2 \right] ds <\infty\qquad \forall t'\in[t,T].
$$
Then,  for $X,X'\in \Lpaths$, $\lambda\geq 0$, and for all $t'\in[t,T]$, formula
\eqref{eq:2016-03-23:01} provides
\begin{equation*}
  ((e^{-\lambda \cdot}S)\sconv (e^{-\lambda \cdot}F_\sigma(X)))_{t'}-
  ((e^{-\lambda \cdot}S)\sconv (e^{-\lambda \cdot}F_\sigma(X')))_{t'}=c_\beta\int_t^{t'}(t'-s)^{\beta-1}\hat Z_sds \qquad \mathbb{P}\mbox{-a.e.,}
\end{equation*}
where $\hat Z$ is an $H$-valued predictable process such that, for a.e.\ $t'\in [t,T]$,
$$
\hat Z_{t'}=\int_t^{t'}(t'-s)^{-\beta}e^{-\lambda (t'-s)}S_{t'-s}(e^{-\lambda \cdot}F_\sigma(X)-e^{-\lambda \cdot}F_\sigma(X'))_sdW_s\qquad \mathbb{P}\mbox{-a.e..}
$$
By collecting the observations above, we can write, for $\lambda
\geq 0$ and for all $t'\in[t,T]$,
\begin{equation*}
    e^{-\lambda p t'} |(S\sconv F_\sigma(X))_{t'}-(S\sconv F_\sigma(X'))_{t'}|_H^p\leq c_\beta^p  \left( \int_0^Tv^{\frac{(\beta-1)p}{p-1}}dv \right) ^{p-1}\int_{t}^T |\hat Z_s|_H^pds,
\end{equation*}
then, by applying 
\cite[Lemma 7.7]{DaPrato1992},
\begin{equation*}
    |S\sconv F_\sigma(X)-S\sconv F_\sigma(X')|^p_{\Lpaths,\lambda}
\leq 
c'_{\beta,T,p} 
|e^{-\lambda \cdot}F_\sigma(X)-e^{-\lambda \cdot}F_\sigma(X')|^p_{p,2,e^{-\lambda \cdot}S,\beta}
\end{equation*}
where $c'_{\beta,T,p}$ is a constant depending only on $\beta,T,p$.
Now, by using 
Assumption \ref{2016-03-24:08}\emph{(\ref{2016-03-24:14})}, we have
\begin{equation*}
|e^{-\lambda \cdot}F_\sigma(X)-e^{-\lambda \cdot}F_\sigma(X')|^p_{p,2,e^{-\lambda \cdot}S,\beta}
\leq
M^p
 \left( \int_0^T
 \left( \int_0^t
   v^{-(\beta+\gamma)2}e^{-\lambda v}
   dv
    \right) ^{p/2}dt \right) |X-X'|^p_{\Lpaths,\lambda}.
\end{equation*}
We finally obtain
\begin{equation}
  \label{eq:2016-03-25:01}
      |S\sconv F_\sigma(X)-S\sconv F_\sigma(X')|_{\Lpaths,\lambda}\leq
c''_{\beta,\gamma,T,p,M,\lambda}|X-X'|_{\Lpaths,\lambda},
\end{equation}
where $c''_{\beta,\gamma,T,p,M,\lambda}$ is a constant depending only on $\beta,\gamma,T,p,M,\lambda$, and is such that 
$$
\lim_{\lambda\rightarrow \infty}
c''_{\beta, \gamma,T,p,M,\lambda}=0.
$$

\medskip
\noindent By \eqref{eq:2016-03-24:15} and
\eqref{eq:2016-03-25:01}, we have, for all $Y,X,Y',X'$,
\begin{equation}
  \label{eq:2016-03-25:02}
  \begin{multlined}[c][.85\displaywidth]
      |\psi(Y,X)
  -
  \psi(Y',X')|_{\Lpaths,\lambda}\leq\\
  \leq M'|Y-Y'|_{\Lpaths,\lambda}+
  C'_{\lambda,g,\gamma,M',\beta,T,p,M}|X-X'|_{\Lpaths,\lambda},
\end{multlined}
\end{equation}
where $C'_{\lambda,g,\gamma,M',\beta,T,p,M}$ is a constant depending only on $\lambda,g,\gamma,M',\beta,T,p,M$, 
such that
\begin{equation}
  \label{eq:2016-03-25:03}
  \lim_{\lambda\rightarrow \infty}C'_{\lambda,g,\gamma,M',\beta,T,p,M}=0.
\end{equation}
\vskip5pt

\begin{theorem}\label{2016-04-13:00}
Let Assumption~\ref{2016-03-24:08} hold and  let $t\in[0,T]$, $p>p^*$.
Then there exists a unique mild solution $X^{t,Y}\in \Lpaths$ to SDE \eqref{2016-03-01:00}. 
Moreover, there exists
a constant $C$,
depending only on $g,\gamma,M,M',T,p$,
such that,
$$
|X^{t,Y}-X^{t,Y'}|_{\Lpaths}\leq 
C
|Y-Y'|_{\Lpaths}
\qquad\forall Y,Y'\in \Lpaths.
$$
\end{theorem}
\begin{proof}
Let us fix any $\beta\in (1/p,1/2-\gamma)$ and let $\psi$ be defined by \eqref{eq:2016-03-24:00}.
It is clear that any fixed point of  $\psi(Y,\cdot)$ is a mild solution to SDE \eqref{2016-03-01:00}.
Then, it is sufficient to apply Lemma~\ref{2016-02-25:02} to $\psi$, taking into account \eqref{eq:2016-03-25:02} and \eqref{eq:2016-03-25:03}, and recalling the equivalence of the norms $|\cdot|_{\Lpaths},|\cdot|_{\Lpaths,\lambda}$.
\end{proof}

\begin{Remark}
Since, for $p^*<p<q$, we have $ \mathcal{L}_{\mathcal{P}_T}^q(\mathbb{S})\subset \mathcal{L}_{\mathcal{P}_T}^p(\mathbb{S})$, then, if $Z\in \mathcal{L}_{\mathcal{P}_T}^q(\mathbb{S})$, the associated mild solution $X^{t,Z}\in
\mathcal{L}_{\mathcal{P}_T}^q(\mathbb{S})$ is also a
mild solution in 
$\mathcal{L}_{\mathcal{P}_T}^p(\mathbb{S})$ and, by uniqueness, it is \emph{the} solution in that space. Hence the solution does not depend on the specific $p>p^*$ chosen.
\end{Remark}

\subsection{G\^ateaux differentiability with respect to the initial datum}\label{2016-04-20:02}

We now study the differentiability of the mild solution $X^{t,Y}$ with respect to the initial datum $Y$.

\begin{assumption}\label{2016-04-05:02}
Let $b,\sigma,g,\gamma$ be as in 
Assumption~\ref{2016-03-24:08}.
 Let $n\in \mathbb{N}$, $n\geq 1$.
 \begin{enumerate}[(i)]
 \item \label{2016-04-21:00}
For all $(\omega,t)\in \Omega_T$ and $u\in U$, $b((\omega,t),\cdot)\in \Gatot{\paths}{H}{n}$, $\sigma((\omega,t),\cdot)u\in \Gatot{\paths}{H}{n}$.
\item There exists $M''$ and
$c\coloneqq  \{c_m\}_{m\in \mathcal{M}}\in \ell^2(\mathcal{M})$ such that
  \begin{equation}
    \label{eq:2016-03-30:01}
  \sup_{j=1,\ldots,n}
    \sup_{\substack{
        \omega\in \Omega\\
        \mathbf{x},\mathbf{y}_1,\ldots,\mathbf{y}_j\in \paths\\
           |\mathbf{y}_1|_\infty=\ldots =|\mathbf{y}_j|_\infty=1
      }
    }
       | \partial ^j_{\mathbf{y}_1\ldots \mathbf{y}_j}b((\omega,s),\mathbf{x})|_H
      \leq M'' g(s),
  \end{equation}
  \begin{equation}
    \label{eq:2016-03-30:02}
    \sup_{j=1,\ldots,n}
    \sup_{
      \substack{
        \omega\in\Omega\\
        \mathbf{x},\mathbf{y}_1,\ldots,\mathbf{y}_j\in \paths\\
        |\mathbf{y}_1|_\infty=\ldots =|\mathbf{y}_j|_\infty=1
      }
      }
   |S_t \partial^j_{\mathbf{y}_1\ldots \mathbf{y}_j} (\sigma((\omega,s),\mathbf{x})e'_m))|_H
     \leq M'' t^{-\gamma}c_m,
   \end{equation}
for all $s\in[0,T]$, $t\in(0,T]$, $ m\in \mathcal{M}$.
\end{enumerate}
\end{assumption}
In accordance with Assumption~\ref{2016-04-05:02}\emph{(\ref{2016-04-21:00})},
by writing $ \partial ^j_{\mathbf{y_1}\ldots \mathbf{y_j}}(\sigma((\omega,s),\mathbf{x})u)$, we mean the G\^ateaux derivative of the map $\mathbf{x} \mapsto \sigma((\omega,s),\mathbf{x}).u$, for fixed $u\in U$.

\begin{lemma}\label{2016-04-05:04}
Suppose that Assumption~\ref{2016-03-24:08} and Assumption~\ref{2016-04-05:02} are satisfied.
Let $p>p^*$, $\beta\in (1/p,1/2-\gamma)$.
Then, for $j=1,\ldots,n$, 
$$
F_b\in \Gat{\Lpaths}{L^{p,1}_{\mathcal{P}_T}(H)}{\mathcal{L}_{\mathcal{P}_T}^{jp}(\paths)}{j},
$$ 
$$
F_\sigma\in \Gat{\Lpaths}{
\overline \Lambda^{p,2,p}_{\mathcal{P}_T,S,\beta}(L(U,H))
}{\mathcal{L}_{\mathcal{P}_T}^{jp}(\paths)}{j}.
$$
and, for
 $X\in \Lpaths$, $Y_1,\ldots, Y_j\in \mathcal{L}^{jp}_{\mathcal{P}_T}(\paths)$, $u\in U$, $\mathbb{P} \otimes m$-a.e.\ $(\omega,t)\in  \Omega_T$,
\begin{equation}\label{2016-04-05:00}
 \begin{dcases}
   \partial ^j_{Y_1\ldots Y_j}F_b(X)(\omega,t)= \partial^j _{Y_1(\omega)\ldots Y_j(\omega)}b((\omega,t),X(\omega))\\ 
   \partial ^j_{Y_1\ldots Y_j}F_\sigma(X)(\omega,t)u= \partial^j _{Y_1(\omega)\ldots Y_j(\omega)}(\sigma((\omega,t),X(\omega))u). 
 \end{dcases}
\end{equation}
Moreover, 
\begin{equation*}
  \sup_{j=1,\ldots,n}
  \sup_{
    \substack{
      X\in \Lpaths\\
      Y_1,\ldots,Y_j\in \mathcal{L}^{jp}_{\mathcal{P}_T}(\paths)\\
      |Y_1|_{\mathcal{L}^{jp}_{\mathcal{P}_T}(\paths)}
      =
      \ldots
      =
      |Y_j|_{\mathcal{L}^{jp}_{\mathcal{P}_T}(\paths)}=1
    }
  }
 \left( 
   | \partial ^j_{Y_1\ldots Y_j}F_b(X)|_{L^{p,1}_{\mathcal{P}_T}(H)}
   +
   | \partial ^j_{Y_1\ldots Y_j}F_\sigma(X)|_{p,2,S,\beta}
 \right) 
\leq M''',
\end{equation*}
where $M'''$ depends only on
$T,p,\beta,\gamma,|g|_{L^1((0,T),\mathbb{R})},M'',|c|_{\ell^2(\mathcal{M})}$.
\end{lemma}
\begin{proof}
  We prove the lemma by induction on $n$.

\smallskip
 \underline{\emph{Case $n=1$.}}
Let $X,Y\in \Lpaths$.
First notice that the function
$$
(\Omega_T,\mathcal{P}_T)\rightarrow H,\ (\omega,t) \mapsto  \partial _{Y(\omega)}b((\omega,t),X(\omega))
$$
is measurable.
Let $\epsilon\in \mathbb{R}\setminus \{0\}$.
  Since $b((\omega,t),\cdot)\in \Gatot{\paths}{H}{1}$ for all $(\omega,t)\in \Omega_T$, we can write
  \begin{equation}\label{2016-03-30:03}
    \begin{split}
      \Delta_{\epsilon Y} F_b(X) (\omega,t)&\coloneqq \epsilon^{-1}
 \left( 
F_b(X+\epsilon Y)(\omega,t)
-F_b(X)(\omega,t)
 \right) \\
&=\epsilon^{-1} \left( 
b((\omega,t),X(\omega)+\epsilon Y(\omega))
-
b((\omega,t),X(\omega))
 \right) \\
&=\int_0^1  \partial _{Y(\omega)}b((\omega,t),X(\omega)+\epsilon \theta Y(\omega))d\theta\qquad \mathbb{P} \otimes m\mbox{-a.e.\ } (\omega,t)\in \Omega_T.
\end{split}
\end{equation}
By
 \eqref{eq:2016-03-30:01}, we also have
\begin{equation}\label{2016-03-30:04}
  \begin{split}
    |\partial _{Y(\omega)}b((\omega,t),X(\omega)+\epsilon Y(\omega))|_H\leq M''g(t) |Y(\omega)|_\infty\qquad \forall (\omega,t)\in\Omega_T, \ \forall \epsilon\in \mathbb{R}.
  \end{split}
\end{equation}
By 
\eqref{2016-03-30:03}
and 
\eqref{2016-03-30:04},
we can 
 apply Lebesgue's dominated convergence theorem and obtain
\begin{equation*}
  \begin{split}
\lim_{\epsilon\rightarrow 0}
\int_0^T 
 \left( 
\mathbb{E}
 \left[     
|\Delta_{\epsilon Y}F_b(X)(\cdot,t)- \partial _Yb((\cdot,t),X)|_H^p
 \right] 
 \right) ^{1/p}
 dt=0.
  \end{split}
\end{equation*}
This proves that $F_b$ has directional derivative at $X$ for the increment $Y$ and that
\begin{equation}
  \label{eq:2016-03-30:05}
   \partial _YF_b(X) (\omega,t)= \partial _{Y(\omega)}b((\omega,t),X(\omega))\qquad \mathbb{P} \otimes m\mbox{-a.e.\ }(\omega,t)\in\Omega_T.
 \end{equation}
 We now show that $ \partial _YF_b(X)$ is continuous in $(X,Y)\in \Lpaths$.
Notice that, by \eqref{eq:2016-03-30:01}, 
the linear map
$\Lpaths\rightarrow L^{p,1}_{\mathcal{P}_T}(H),\ Y \mapsto   \partial _YF_b(X)$,
 is bounded, uniformly in $X$.
Then it is sufficient to verify the continuity of $ \partial _YF_b(X)$ in $X$, for fixed $Y$.
Let $X_k\rightarrow X$ in $\Lpaths$.
By \eqref{eq:2016-03-30:01}, \eqref{eq:2016-03-30:05}, and  Lebesgue's dominated convergence theorem, we have
\begin{equation*}
  \lim_{k\rightarrow \infty} \partial _YF_b(X_k)= \partial _YF_b(X)\ \mbox{in \ }L^{p,1}_{\mathcal{P}_T}(H).
\end{equation*}
This concludes the proof that $F_b\in \Gatot{\Lpaths}{L^{p,1}_{\mathcal{P}_T(H)}}{1}$ and that the differential is uniformly bounded.

Similarly, as regarding $F_\sigma$, we have that, for all $u\in U$, the function
$$
(\Omega_T,\mathcal{P}_T)\rightarrow H,\ (\omega,t) \mapsto  \partial _{Y(\omega)}(\sigma(t,X(\omega))u)
$$
is measurable, and 
  \begin{equation}\label{2016-03-31:00}
    \begin{split}
      \Delta_{\epsilon Y} (F_\sigma(X)u) &(\omega,t)\coloneqq \epsilon^{-1}
 \left( 
(F_\sigma(X+\epsilon Y)u)(\omega,t)
-(F_\sigma(X)u)(\omega,t)
 \right) \\
&=\epsilon^{-1} \left( 
\sigma(
(\omega,t),X(\omega)+\epsilon Y(\omega))u
-
\sigma((\omega,t),X(\omega))u
 \right) \\
&=\int_0^1  \partial _{Y(\omega)}(\sigma((\omega,t),X(\omega)+\epsilon \theta Y(\omega))u)d\theta\qquad \mathbb{P} \otimes m\mbox{-a.e.\ } (\omega,t)\in \Omega_T.
\end{split}
\end{equation}
By \eqref{eq:2016-03-30:02}, for all $0\leq s<t\leq T$, $\omega\in\Omega$, $\epsilon\in \mathbb{R}$, $m\in \mathcal{M}$,
\begin{equation}
  \label{eq:2016-03-31:01}
|S_{t-s} \partial _{Y(\omega)}(\sigma((\omega,s),X(\omega)+\epsilon Y(\omega))e'_m)|_H\leq M''(t-s)^{-\gamma}c_m| Y(\omega)|_\infty.
\end{equation}
By repeatedly applying Lebesgue's dominated convergence theorem, we have that
\begin{equation*}
\hskip-0.1cm
    \int_0^T
    \left(
      \int_0^t
      (t-s)^{-2 \beta }
      \left( 
        \mathbb{E}
        \left[ 
          \left( 
            \sum_{m\in \mathcal{M}}
            \left|S_{t-s}
            \left( 
              \Delta_{\epsilon Y} F_\sigma(X)(\cdot,s).e'_m
              - \partial _Y(\sigma((\cdot,s),X).e'_m)
            \right)
            \right|^2_H
          \right) ^{p/2}
        \right] 
      \right) ^{2/p}
      ds
    \right) ^{p/2}
    dt
\end{equation*}
goes to $0$ as $\epsilon\rightarrow 0$.
This proves that $F_\sigma$ has directional derivative at $X$ for the increment $Y$ and, taking into account the separability of $U$, that
\begin{equation}
  \label{eq:2016-03-30:05b}
   \partial _YF_\sigma(X) (\omega,t)= \partial _{Y(\omega)}(\sigma((\omega,t),X(\omega))\#)\qquad \mathbb{P} \otimes m\mbox{-a.e.\ }(\omega,t)\in\Omega_T.
 \end{equation}
By  \eqref{eq:2016-03-31:01} and arguing similarly as done for  $ \partial _YF_b(X)$, in order to show the continuity of $ \partial _YF_\sigma(X)$ in  $(X,Y)\in \Lpaths$, it is sufficient
to verify the continuity of $ \partial _YF_\sigma(X)$ in $X$, for fixed $Y$.
Let $X_k\rightarrow X$ in $\Lpaths$.
By
\eqref{eq:2016-03-30:02},
 \eqref{eq:2016-03-30:05b}, and Lebesgue's dominated convergence theorem, we have
$$
\lim_{k\rightarrow \infty} \partial _YF_\sigma(X_k)= \partial _YF_\sigma(X)\ \mbox{in \ }\overline \Lambda^{p,2,p}_{\mathcal{P}_T,S,\beta}(L(U,H)).
$$
This shosws that $F_\sigma\in \Gatot{\Lpaths}{\overline \Lambda^{p,2,p}_{\mathcal{P}_T,S,\beta}(L(U,H))}{1}$ and that the differential is uniformly bounded.

\smallskip
 \underline{\emph{Case $n>1$.}}
Let $X\in \Lpaths$ and $Y_1,\ldots, Y_n\in \mathcal{L}^{np}_{\mathcal{P}_T}(\paths)$.
By inductive hypothesis, 
we can assume that $ \partial^{n-1} _{Y_1\ldots Y_{n-1}}F_b(X)\in L^{p,1}_{\mathcal{P}_T}(H)$ exists, 
jointly continuous in $X\in \Lpaths$ and $Y_1,\ldots, Y_{n-1}\in \mathcal{L}^{(n-1)p}_{\mathcal{P}_T}(H)$, 
and that
$$
 \partial ^{n-1}_{Y_1\ldots Y_{n-1}}F_b(X)(\omega,t)= \partial ^{n-1}_{Y_1(\omega)\ldots Y_{n-1}(\omega)}b((\omega,t),X(\omega))\qquad \mathbb{P} \otimes m\mbox{-a.e.\ }(\omega,t)\in \Omega_T.
$$
The argument goes like the case $n=1$.
Let $\epsilon\in \mathbb{R}\setminus \{0\}$.
  Since $b((\omega,t),\cdot)\in \Gatot{\paths}{H}{n}$ for $(\omega,t)\in \Omega_T$, we can write,
for $ \mathbb{P} \otimes m\mbox{-a.e.\ } (\omega,t)\in \Omega_T$,
  \begin{equation*}
    \begin{split}
      \Delta_{\epsilon Y_n} & \partial ^{n-1}_{Y_1\ldots Y_{n-1}} F_b(X) (\omega,t)\coloneqq \epsilon^{-1}
 \left( 
\partial ^{n-1}_{Y_1\ldots Y_{n-1}}F_b(X+\epsilon Y_n)(\omega,t)
-\partial ^{n-1}_{Y_1\ldots Y_{n-1}}F_b(X)(\omega,t)
 \right) \\
&=\epsilon^{-1} \left( 
 \partial ^{n-1}_{Y_1(\omega)\ldots Y_{n-1}(\omega)}b((\omega,t),X(\omega)+\epsilon Y_n(\omega))
-
 \partial ^{n-1}_{Y_1(\omega)\ldots Y_{n-1}(\omega)}b((\omega,t),X(\omega))
 \right) \\
&=\int_0^1  \partial ^n_{Y_1(\omega)\ldots Y_{n-1}(\omega)Y_n(\omega)}b((\omega,t),X(\omega)+\epsilon \theta Y_n(\omega))d\theta.
\end{split}
\end{equation*}
By
 \eqref{eq:2016-03-30:01} we have
\begin{equation*}
  \begin{split}
    |\partial^n _{Y_1(\omega)\ldots Y_n(\omega)}b((\omega,t),X(\omega)+\epsilon Y_n(\omega))|_H\leq M''g(t)\prod_{j=1}^n |Y_j(\omega)|_\infty\qquad \forall (\omega,t)\in\Omega_T, \ \forall \epsilon\in \mathbb{R}.
  \end{split}
\end{equation*}
Since $Y_j\in \mathcal{L}^{np}_{\mathcal{P}_T}(H)$, 
by the generalized H\"older inequality
 $\prod_{j=1}^n |Y_j|_\infty\in L^p((\Omega,\mathcal{F}_T,\mathbb{P}),\mathbb{R})$.
Then we can  apply Lebesgue's dominated convergence theorem twice to obtain
\begin{equation*}
  \begin{split}
\lim_{\epsilon\rightarrow 0}
\int_0^T 
 \left( 
\mathbb{E}
 \left[     
|\Delta_{\epsilon Y_n} \partial ^{n-1}_{Y_1\ldots Y_{n-1}}F_b(X)(\cdot,t)- \partial^n _{Y_1\ldots Y_n}b((\cdot,t),X)|_H^p
 \right] 
 \right) ^{1/p}
 dt=0.
  \end{split}
\end{equation*}
This proves that $ \partial ^{n-1}_{Y_1\ldots Y_{n-1}}F_b$ has directional derivative at $X$ for the increment $Y_n$ and that
\begin{equation}\label{2017-05-22:00}
    \partial^n _{Y_1\ldots Y_{n-1}Y_n}F_b(X) (\omega,t)= \partial^n _{Y_1(\omega)\ldots Y_n(\omega)}b((\omega,t),X(\omega))\qquad \mathbb{P} \otimes m\mbox{-a.e.\ }(\omega,t)\in\Omega_T.
 \end{equation}
The continuity of $\partial^{n} _{Y_1\ldots Y_{n-1}Y_n}F_b(X)$ in $X\in \Lpaths$, $Y_1,\ldots,Y_n\in \mathcal{L}^{np}_{\mathcal{P}_T}(H)$, is proved similarly as for the case $n=1$, again by invoking the generalized H\"older inequality.
%
This concludes the proof that $F_b\in \Gat{\Lpaths}{L^{p,1}_{\mathcal{P}_T(H)}}{\mathcal{L}^{np}_{\mathcal{P}_T}(H)}{n}$.
The uniform boundedness of the differentials is
 obtained by
\eqref{eq:2016-03-30:01}, 
\eqref{2017-05-22:00},
and the generalized H\"older inequality.

Finally, as regarding $F_\sigma$,
let again
$X\in \Lpaths$ and $Y_1,\ldots,Y_n\in\mathcal{L}^{np}_{\mathcal{P}_T}(\paths)$.
By inductive hypothesis, we can assume that
 $ \partial ^{n-1}_{Y_1\ldots Y_{n-1}}F_\sigma(X)\in \overline \Lambda^{p,2,p}_{\mathcal{P}_T,S,\beta}(L(U,H))$ exists, that it is continuous in $X\in \Lpaths$, $Y_1,\ldots,Y_{n-1}\in \mathcal{L}^{(n-1)p}_{\mathcal{P}_T}(\paths)$, and that, for all $u\in U$,
$$
\partial ^{n-1}_{Y_1\ldots Y_{n-1}}F_\sigma(X)(\omega,t)u=
\partial ^{n-1}_{Y_1(\omega)\ldots Y_{n-1}(\omega)}
(\sigma((\omega,t),X(\omega))u)\qquad \mathbb{P} \otimes m\mbox{-a.e.\ }(\omega,t)\in\Omega_T.
$$
For $\epsilon\in \mathbb{R}\setminus\{0\}$, by strongly continuous G\^ateaux differentiability of 
$$
x \mapsto  \partial ^{n-1}_{Y_1(\omega)\ldots Y_{n-1}(\omega)}(\sigma(t,x)u),
$$
 we can write,
 \begin{equation*}
   \begin{split}
     \Delta_{\epsilon Y_n}& \partial ^{n-1}_{Y_1\ldots Y_{n-1}} F_\sigma(X) (\omega,t)u
\coloneqq
\epsilon^{-1} \left(  \partial ^{n-1}_{Y_1\ldots Y_{n-1}}F_\sigma(X+\epsilon Y_n) (\omega,t)u- \partial ^{n-1}_{Y_1\ldots Y_{n-1}} F_\sigma(X) (\omega,t)u \right) \\
&=
\epsilon^{-1}
 \left( 
 \partial ^{n-1}_{Y_1(\omega)\ldots Y_{n-1}(\omega)}
(\sigma((\omega,t),X(\omega)+\epsilon Y_n(\omega))u)
- \partial ^{n-1}_{Y_1(\omega)\ldots Y_{n-1}(\omega)}
(\sigma((\omega,t),X(\omega))u) \right) \\
&=
\int_0^1 \partial ^n_{Y_1(\omega)\ldots Y_n(\omega)}
(\sigma((\omega,t),X(\omega)+\epsilon \theta Y_n(\omega))u)d\theta.
   \end{split}
 \end{equation*}
By \eqref{eq:2016-03-30:02} we have,
for all
$\omega\in\Omega$,  $\epsilon\in \mathbb{R}$, $0\leq s<t\leq T$, $m\in \mathcal{M}$,
$$
|
S_{t-s} \partial ^n_{Y_1(\omega)\ldots Y_n(\omega)}
(
\sigma((\omega,s),X(\omega)+\epsilon Y_n(\omega))e'_m
)
|_H\leq 
M''(t-s)^{-\gamma}c_m \prod _{j=1}^n|Y_j(\omega)|_\infty.
$$
By the generalized H\"older inequality and by Lebesgue's dominated convergence theorem, we conclude
\begin{equation}\label{2017-05-22:01}
  \begin{multlined}[c][.85\displaywidth]
\lim_{\epsilon\rightarrow 0}\int_0^T
 \left( 
\int_0^t
(t-s)^{-2\beta}
 \left( 
   \mathbb{E}
    \left[ 
       \left( 
         \sum_{m\in \mathcal{M}}
         \left| S_{t-s}
          \left( 
            \Delta_{\epsilon Y_n}
             \partial ^{n-1}_{Y_1\ldots Y_{n-1}}
            F_\sigma(X)
            (\omega,s)e'_m\right.\right.\right.\right.\right.\right.\\
\left.\left.\left.\left.\left.\left.            -
             \partial ^n_{Y_1(\omega)\ldots Y_n(\omega)}
             (\sigma((\cdot,s),X)e'_m)
              \right) \right|^2_H
            \right) ^{p/2}
          \right] 
        \right) ^{2/p}
        ds
      \right)^{p/2} dt=0.
\end{multlined}
\end{equation}
Then $ \partial ^{n-1}_{Y_1\ldots Y_{n-1}}F_\sigma$ has directional derivative at $X$ for the increment $Y_n$, given by, for all $u\in U$,
$$
 \partial _{Y_n} \partial ^{n-1}_{Y_1\ldots Y_{n-1}}F_\sigma(X)(\omega,t)u= \partial ^n_{Y_1(\omega)\ldots Y_n(\omega)}
(\sigma((\omega,t),X(\omega))u)\qquad \mathbb{P} \otimes m\mbox{-a.e.\ }(\omega,t)\in \Omega_T.
$$
The continuity of $ \partial _{Y_n} \partial ^{n-1}_{Y_1\ldots Y_{n-1}}F_\sigma(X)$ with respect to $X\in \Lpaths$, $Y_1,\ldots, Y_n\in \mathcal{L}^{np} _{\mathcal{P}_T}(H)$, is proved as for the case $n=1$.
Then $F_\sigma\in \Gat{\Lpaths}{\overline \Lambda^{p,2,p}_{\mathcal{P}_T,S,\beta}(L(U,H))}{\mathcal{L}^{np}_{\mathcal{P}_T}(H)}{n}$.
The uniform boundedness of the differentials is
 obtained by
\eqref{eq:2016-03-30:02}, 
\eqref{2017-05-22:01},
and the generalized H\"older inequality.
\end{proof}

Due to the fact that $X^{t,Y}$ is the fixed point of $\psi(Y,\cdot)$ and due to the structure of $\psi$, the previous lemma permits to easily obtain the following

\begin{theorem}\label{2016-04-05:05}
  Suppose that Assumption~\ref{2016-04-05:02} is satisfied.
   Let $t\in[0,T]$, $ p> p^*$, $p\geq n$.
  Then
  the map
  \begin{equation}
    \label{eq:2016-04-05:03}
\mathcal{L}^{p^n}
_{\mathcal{P}_T}(\paths)
\rightarrow \Lpaths,\ 
Y \mapsto  X^{t,Y}
  \end{equation}
belongs to $\Gatot{\mathcal{L}^{p^n}
_{\mathcal{P}_T}(\paths)
}{\Lpaths}{n}$ and the G\^ateaux differentials 
up to order $n$
are uniformly bounded by a constant depending only on
$T,p,\gamma,g,M,M',M'',|c|_{\ell^2(\mathcal{M})}$.
\end{theorem}
\begin{proof}
Let $\beta\in (1/p,1/2-\gamma)$.
We have  $p^k> p^*$ and $\beta\in (1/p^k,1/2-\gamma)$
for all $k=1,\ldots,n$. 
Then, for  $k=1,\ldots,n$,
 the map
$$
\psi_k\colon 
\mathcal{L}^{p^k}_{\mathcal{P}_T}(\paths)
\times
\mathcal{L}^{p^k}_{\mathcal{P}_T}(\paths)
\rightarrow
\mathcal{L}^{p^k}_{\mathcal{P}_T}(\paths),\ 
(Y,X)
 \mapsto 
\id^S_t(Y)
+S\conv F_b(X)
+S\sconv F_\sigma(X)
$$
is well-defined,
where we have implicitly 
chosen the space
$L^{p^k,1}_{\mathcal{P}_T}(H)$ as codomain of $F_b$ and
 $\overline \Lambda_{\mathcal{P}_T,S,\beta}^{p^k,2,p^k}(L(U,H))$ as codomain of $F_\sigma$.
Since the functions
\begin{gather*}
  \mathcal{L}^{p^k}_{\mathcal{P}_T}(\paths)\rightarrow
\mathcal{L}^{p^k}_{\mathcal{P}_T}(\paths)\\ 
S\conv \#\colon L^{p^k,1}_{\mathcal{P}_T}(H)\rightarrow \mathcal{L}^{p^k}_{\mathcal{P}_T}(\paths)\\
S\sconv \#\colon
\overline \Lambda^{p^k,2,p^k}_{\mathcal{P}_T,S,\beta}(L(U,H))\rightarrow \mathcal{L}^{p^k}_{\mathcal{P}_T}(\paths)
\end{gather*}
are linear and continuous,
with an upper bound for the operator norms 
depending only on $\beta,M',T,p$,
we have,
 by
applying
Lemma \ref{2016-04-05:04},
 for $k,j=1,\ldots,n$,
$$
\psi_k\in \Gat{
\mathcal{L}^{p^k}_{\mathcal{P}_T}(\paths)\times
\mathcal{L}^{p^k}_{\mathcal{P}_T}(\paths)
}{
\mathcal{L}^{p^k}_{\mathcal{P}_T}(\paths)}
{
\mathcal{L}^{p^k}_{\mathcal{P}_T}(\paths)\times
\mathcal{L}^{jp^k}_{\mathcal{P}_T}(\paths)
}
{j},
$$
with differentials bounded by a constant depending only on $g,\gamma,M,M',M'',|c|_{\ell^2(\mathcal{M})},T$,
on $p^k$ (hence on $p$), and
 on $\beta$, which depends on $p,\gamma$.
In particular, 
since $np^k\leq p^{k+1}$,  we have,
for the rescritions $\psi_{k|\mathcal{L}^{p^n}_{\mathcal{P}_T}(\paths)\times
\mathcal{L}^{p^k}_{\mathcal{P}_T}(\paths)}$
of $\psi_k$ to
$\mathcal{L}^{p^n}_{\mathcal{P}_T}(\paths)\times
\mathcal{L}^{p^k}_{\mathcal{P}_T}(\paths)$,
\begin{equation*}
  \begin{dcases}
  \psi_{k|\mathcal{L}^{p^n}_{\mathcal{P}_T}(\paths)\times
\mathcal{L}^{p^k}_{\mathcal{P}_T}(\paths)}\in 
\Gatot{
\mathcal{L}^{p^n}_{\mathcal{P}_T}(\paths)\times
\mathcal{L}^{p^k}_{\mathcal{P}_T}(\paths)
}{
\mathcal{L}^{p^k}_{\mathcal{P}_T}(\paths)}{1}\\
\psi_{k|\mathcal{L}^{p^n}_{\mathcal{P}_T}(\paths)\times
\mathcal{L}^{p^k}_{\mathcal{P}_T}(\paths)}\in \Gat{
\mathcal{L}^{p^n}_{\mathcal{P}_T}(\paths)\times
\mathcal{L}^{p^k}_{\mathcal{P}_T}(\paths)
}{
\mathcal{L}^{p^k}_{\mathcal{P}_T}(\paths)}
{
\mathcal{L}^{p^n}_{\mathcal{P}_T}(\paths)\times
\mathcal{L}^{p^{k+1}}_{\mathcal{P}_T}(\paths)
}
{n}
\end{dcases}
\end{equation*}
for $k=1,\ldots,n$,
with the G\^ateaux differentials that are uniformly bounded
by a constant depending only on 
$g$,$\gamma$, $M$,$M'$,$M''$,$|c|_{\ell^2(\mathcal{M})},T$, on $\beta$ (hence on $p,\gamma$), and on $p^n,p^k,p^{k+1}$ (hence on $p$).

By \eqref{eq:2016-03-25:02} and \eqref{eq:2016-03-25:03} (where $p$ should be replaced by $p^k$),
 there exists $\lambda>0$, depending only on $g,\gamma,M,M',\beta,T$, and on $p^k$ (hence on $p$), such that $\psi_k$ is a parametric $1/2$-contraction with respect to the second variable, uniformly in the first one, when the space
$\mathcal{L}^{p^k}_{\mathcal{P}_T}(\paths)$ is 
endowed with the equivalent norm $|\cdot|_{\mathcal{L}^{p^k}_{\mathcal{P}_T}(\paths)
,\lambda}$.
Then we can assume that the uniform bound of the G\^ateaux differentials of $\psi_k$, for $k=1,\ldots,n$, holds with respect to the equivalent norms
$|\cdot|_{\mathcal{L}^{p^k}_{\mathcal{P}_T}(\paths)
,\lambda}$, and is again depending only on 
 $g$, $\gamma$, $M$,$M'$,$M''$,$|c|_{\ell^2(\mathcal{M})}$, $T$, $p$.

Now consider Assumption~\ref{2016-02-24:00},
after setting: 
\begin{enumerate}[-]
\item $\alpha\coloneqq 1/2$;
\item $U\coloneqq X\coloneqq (\mathcal{L}^{p^n}_{\mathcal{P}_T}(\paths),|\cdot|_{\mathcal{L}^{p^n}_{\mathcal{P}_T}(\paths),\lambda})$;
\item $Y_1\coloneqq (\Lpaths,|\cdot|_{\Lpaths,\lambda})$,
\ldots,
$Y_k\coloneqq (\mathcal{L}^{p^k}_{\mathcal{P}_T},|\cdot|_
{\mathcal{L}^{p^k}_{\mathcal{P}_T}(\paths),\lambda})$, \ldots,
$Y_n\coloneqq (\mathcal{L}^{p^n}_{\mathcal{P}_T},|\cdot|_{\mathcal{L}^{p^n}_{\mathcal{P}_T}(\paths),\lambda})$;
\item $h_1\coloneqq 
\psi_{1|\mathcal{L}^{p^n}_{\mathcal{P}_T}(\paths)\times \mathcal{L}^{p}_{\mathcal{P}_T}(\paths)}$, \ldots, $h_k\coloneqq \psi_{k|\mathcal{L}^{p^n}_{\mathcal{P}_T}(\paths)\times \mathcal{L}^{p^k}_{\mathcal{P}_T}(\paths)}$, \ldots, $h_n\coloneqq \psi_{n|\mathcal{L}^{p^n}_{\mathcal{P}_T}(\paths)\times \mathcal{L}^{p^n}_{\mathcal{P}_T}(\paths)}$.
\end{enumerate}
The discussion above, together with the smooth dependence of
$h_k$ on the first variable,
 shows that Assumption~\ref{2016-02-24:00} is verified.
We can then apply Theorem~\ref{teo:derivabilita.punto.fisso}, which provides 
$$
(\eqref{eq:2016-04-05:03}=)\ 
\mathcal{L}^{p^n}
_{\mathcal{P}_T}(\paths)
\rightarrow \Lpaths,\ 
Y \mapsto  X^{t,Y},
\in
\Gatot{
\mathcal{L}^{p^n}_{\mathcal{P}_T}(\paths)
}{
\Lpaths
}{n}.
$$
Finally, 
by applying 
Corollary \ref{corr:2012-04-20-aa}, we 
obtain
 the uniform boundedness of the G\^ateaux differentials 
up to order $n$
of 
$\eqref{eq:2016-04-05:03}$, with a  bound 
that depends
only on
$T$,$\gamma$,$g$,$M$,$M'$,$M''$, $|c|_{\ell^2(\mathcal{M})}$,$p$.
\end{proof}

\begin{remark}\label{2017-05-01:00}
  As said in the introduction, we obtain
the G\^ateaux differentiability  of $x \mapsto  X^{t,x}$
by studying the parametric contraction providing $X^{t,x}$ as its unique fixed point, similarly as done in \cite{DaPrato2004} for the non-path-dependent case.
A different approach consists in studying directly the variations $\lim_{h\rightarrow 0}\frac{X^{t,x+hv}-X^{t,x}}{h}$, showing that the limit exists
(under suitable smooth assumptions on the coefficients) and 
is continuous with respect to $v$, for fixed $t,x$. This would provide the  existence of the G\^ateaux differential $ \partial X^{t,x}$.
Usually, in this way one shows also that $ \partial X^{t,x}.v$ solves  an SDE.
By using this SDE, one could go further and prove that the second order derivative $ \partial ^2 X^{t,x}. (v,w)$ exists, and that it is continuous in $v,w$, for fixed $t,x$. This would provide the second order G\^ateaux differentiability of $x \mapsto X^{t,x}$.
In this way, it is possible also to study the continuity of the G\^ateaux differentials, by considering the SDEs solved by the directional derivatives, and to obtain Fr\'echet differentiability (under suitable assumptions on the coefficients, e.g.\ uniformly continuous Fr\'echet differentiability).
By doing so, first- and second-order Fr\'echet differentiability are proved in \cite{Knoche2001}.
But
if one wants to use these methods to obtain
 derivatives of a generic order $n\geq 3$, then 
a recursive formula providing the SDE solved by the $(n-1)$th-order derivatives is needed, hence we fall back to a statement like Theorem~\ref{teo:derivabilita.punto.fisso}.
One could also try to prove the Fr\'echet differentiability of $x \mapsto  X^{t,x}$
 by studying directly
the Fr\'echet differentiability of the parametric contractions providing the mild solution $X^{t,x}$.
This is the approach
followed in
\cite[Theorem~3.9]{Gawarecki2011}, for orders $n=1,2$.
Nevertheless, we notice that the proof of \cite[Theorem~3.8]{Gawarecki2011}, on which \cite[Theorem~3.9]{Gawarecki2011} relies, contains some inaccuracy: 
it is not clear
why the term 
$|\eta(s)|_H/|\eta|_{\mathcal{\tilde H}_2}$ is bounded by $1$, uniformly in $(\omega,s)$, when $\eta$ is  only supposed to be a  process such that $|\eta|^2_{\mathcal{\tilde H}_2}\coloneqq \sup_{s\in[0,T]}\mathbb{E}[|\eta(s) |_H^2]<\infty$.
\end{remark}

\medskip
Let $n=2$ and let $h_1$ as in the proof of Theorem \ref{2016-04-05:05}.
By continuity and linearity of $\id^S_t$, $S\conv \#$, $S\sconv\#$, and by recalling 
Lemma \ref{2016-04-05:04},
we have, for $Y,Y_1,Y_2\in \mathcal{L}^{p^2}_{\mathcal{P}_T}(\paths)$ (the space of the first variable of $h_1$), $X,X_1,X_2\in \Lpaths$ (the space of the second variable of $h_1$),
\begin{equation*}
  \begin{dcases}
         \partial _{Y_1}h_1(Y,X)=\id^S_t(Y_1)\\
     \partial _{X_1}h_1(Y,X)=S\conv  \partial _{X_1}F_b(X)+S\sconv  \partial _{X_1}F_\sigma(X)\\
     \partial^2 _{Y_1Y_2}h_1(Y,X)=     \partial^2 _{Y_1X_1}h_1(Y,X)=
0\\
     \partial^2 _{X_1X_2}h_1(Y,X)=S\conv  \partial^2 _{X_1X_2}F_b(X)+S\sconv  \partial^2 _{X_1X_2}F_\sigma(X).
   \end{dcases}
 \end{equation*}
Then, by Theorem \ref{teo:derivabilita.punto.fisso},
we have
\begin{subequations}
  \begin{equation}
  \label{eq:2016-04-05:06}
   \partial _{Y_1} X^{t,Y}
=\id_t^S(Y_1) 
+
S\conv  \partial _{ \partial _{Y_1} X^{t,Y}}F_b(X^{t,Y})+S\sconv  \partial _{ \partial _{Y_1} X^{t,Y}}F_\sigma(X^{t,Y})
\end{equation}
\begin{equation}
  \label{eq:2016-04-05:07}
  \begin{split}
       \partial ^2_{Y_1Y_2}X^{t,Y}
  =&
S\conv  \partial _{\partial ^2_{Y_1Y_2}X^{t,Y}}F_b(X)+S\sconv  \partial _{\partial ^2_{Y_1Y_2}X^{t,Y}}F_\sigma(X)\\
&+
S\conv  \partial^2 _{ \partial _{Y_1}X^{t,Y}\partial _{Y_2}X^{t,Y}}F_b(X)+S\sconv  \partial^2 _{ \partial _{Y_1}X^{t,Y}\partial _{Y_2}X^{t,Y}}F_\sigma(X)
\end{split}
\end{equation}
\end{subequations}
where the equality  \eqref{eq:2016-04-05:06} holds in the space
$\mathcal{L}^{p^2}_{\mathcal{P}_T}(\paths)$
and the equality \eqref{eq:2016-04-05:07} holds in the space
$\mathcal{L}^{p}_{\mathcal{P}_T}(\paths)$.
Formulae
\eqref{eq:2016-04-05:06} and
\eqref{eq:2016-04-05:07}
 generalize to the present setting the well-known 
SDEs
for the first- and  second-order derivatives with respect to the initial datum of mild solutions
of non-path-dependent SDEs
 (\cite[Theorem 9.8 and Theorem 9.9]{DaPrato2014}).

 \begin{remark}
\label{2017-05-01:01}
   Suppose that $\paths =\mathbb{D}$, where $\mathbb{D}$
   is the space of right-continuous left-limited functions $[0,T]\rightarrow H$.
  Notice that $\mathbb{D}$ satisfies all the properties required at p.\ \pageref{2016-04-13:09}.
   Then our setting applies and \eqref{eq:2016-04-05:06}-\eqref{eq:2016-04-05:07} provide equations for the first- and  second-order directional derivatives of $X^{t,Y}$ with respect to vectors belonging to $\mathcal{L}^{p^2}_{\mathcal{P}_T}(\mathbb{D})$.
   In particular,
if $\varphi\colon \mathbb{D}\rightarrow \mathbb{R}$ is a suitably regular functional,
then
the so-called ``vertical derivatives''
 in the sense of
Dupire
of $F(t,\mathbf{x})\coloneqq \mathbb{E}[\varphi(X^{t,\mathbf{x}})]$,
used in the finite dimensional It\=o calculus
developed
by \cite{Cont2010a,Cont2010,Cont2013,Dupire2009} to show that $F$ solves a path-dependent Kolmogorov equation associated to $X$,
can be classically obtained by
the chain rule
starting from
the G\^ateaux derivatives
$ \partial _{Y_1}X^{t,Y}$, $ \partial ^2_{Y_1Y_2}X^{t,Y}$,
where
$y_1,y_1\in H$ and
$Y_1\coloneqq \mathbf{1}_{[t,T]}(\cdot)y_1,Y_2\coloneqq \mathbf{1}_{[t,T]}(\cdot)y_2$.
 \end{remark}

\subsection{Perturbation of path-dependent SDEs}\label{2017-04-28:03}

In this section we study the stability of the mild solution $X^{t,Y}$ and of its G\^ateaux derivatives with respect to perturbations of the data
$t,Y,S,b,\sigma$.

Let us fix sequences 
$\mathbf{t}\coloneqq\{t_j\}_{j\in\mathbb{N}}\subset [0,T]$, 
$\{S_j\}_{j\in \mathbb{N}}\subset L(H)$, $\{b_j\}_{j\in \mathbb{N}}$, $\{\sigma_j\}_{j\in \mathbb{N}}$,
satisfying the following assumption.

\begin{assumption}\label{2016-04-05:08}
Let $b$, $\sigma$, $g$, $\gamma$, $M$, be as in  
Assumption~\ref{2016-03-24:08}.
Assume that
  \begin{enumerate}[(i)]
  \item $\{t_j\}_{j\in \mathbb{N}}$ is a sequence 
converging
to $ \hat t$ in $[0,T]$;
  \item\label{2016-04-05:09} for all $j\in \mathbb{N}$, $b_j\colon (\Omega_T\times \paths,\mathcal{P}_T
 \otimes \mathcal{B}_\paths)\rightarrow (H,\mathcal{B}_H)$ is measurable; 
\item\label{2016-04-05:10} for all $j\in \mathbb{N}$, $\sigma_j\colon (\Omega_T\times \paths,\mathcal{P}_T
 \otimes \mathcal{B}_\paths)\rightarrow
L(U,H)$ is strongly measurable;
\item\label{2016-04-05:11} 
for all $j\in \mathbb{N}$ and all $ ((\omega,t),\mathbf{x})\in \Omega_T \times \paths$,
$ b_j((\omega,t),\mathbf{x})=b_j((\omega,t),\mathbf{x}_{t\wedge \cdot})$
and
$  \sigma_j((\omega,t),\mathbf{x})=\sigma_j((\omega,t),\mathbf{x}_{t\wedge \cdot})$;
\item\label{2016-04-05:12} 
for all $j\in \mathbb{N}$,
  \begin{equation*}
    \begin{dcases}
      |b_j((\omega,t),\mathbf{x})|_H\leq g(t)(1+|\mathbf{x}|_\infty)& \forall ((\omega,t),\mathbf{x})\in\Omega_T\times \paths,\\
      |b_j((\omega,t),\mathbf{x})-
b_j((\omega,t),\mathbf{x}')|_H\leq g(t)|\mathbf{x}-\mathbf{x}'|_\infty& \forall (\omega,t)\in\Omega_T,\ \forall \mathbf{x},\mathbf{x}'\in \paths;
    \end{dcases}
  \end{equation*}
\item\label{2016-04-05:13} 
for all $j\in \mathbb{N}$,
  \begin{equation*}
\hskip-1.1cm    \begin{dcases}
      |(S_j)_t\sigma_j((\omega,s),\mathbf{x})|_{L_2(U,H)}\leq M t^{-\gamma}(1+|\mathbf{x}|_\infty)& \hskip-7pt\forall ((\omega,s),\mathbf{x})\in\Omega_T\times \paths,\ \forall t\in (0,T],\\
      |(S_j)_t\sigma_j((\omega,s),\mathbf{x})-
(S_j)_t\sigma_j((\omega,s),\mathbf{x}')|_{L_2(U,H)}\leq M t^{-\gamma}|\mathbf{x}-\mathbf{x}'|_\infty&\hskip-7pt \forall (\omega,s)\in\Omega_T,\forall\mathbf{x},\mathbf{x}'\in \paths, \forall t\in(0,T];
    \end{dcases}
  \end{equation*}
\item for all $t\in[0,T]$, $\{(S_j)_t\}_{j\in \mathbb{N}}$ converges strongly to $S_t$, that is
$$
\lim_{j\rightarrow \infty}(S_j)_tx=S_tx\qquad \forall x\in H;
$$
\item 
the following convergences hold true:
  \begin{equation*}
    \label{eq:2016-04-07:00}
    \begin{dcases}
    \lim_{j\rightarrow \infty}
|b((\omega,t),\mathbf{x})-b_j((\omega,t),\mathbf{x})|_H=0 & \forall (\omega,t)\in\Omega_T,\ \forall \mathbf{x}\in \paths\\
    \lim_{j\rightarrow \infty}
    |
S_t\sigma((\omega,s),\mathbf{x})-(S_j)_t\sigma_j((\omega,s),\mathbf{x})|_{L_2(U,H)}=0 &\forall (\omega,s)\in\Omega_T,\
\forall t\in(0,T],
\ 
\forall \mathbf{x}\in \paths.
\end{dcases}
\end{equation*}
\end{enumerate}
\end{assumption}
Under 
Assumption~\ref{2016-04-05:08},
for  $p>p^*$ and $\beta\in (1/p,1/2-\gamma)$,
we define $\id^{S_j}_{t_j}$, $F_{b_j}$, $F_{\sigma_j}$, $S_j\convt{t_j}\#$, $S_j\sconvt{t_j}\#$, $\psi_j$, similarly as done for $\id ^S_t$,
$F_b$, $F_\sigma$, $S\conv\#$, $S\sconv \#$, $\psi$, that is
\begin{subequations}
  \begin{equation*}
    \id^{S_j}_{t_j}\colon \Lpaths\rightarrow \Lpaths,\ Y \mapsto
\mathbf{1}_{[0,t_j]}(\cdot)
 Y+\mathbf{1}_{(t_j,T]}(\cdot)(S_j)_{\cdot-t_j}Y_{t_j}
  \end{equation*}
  \begin{equation*}
    F_{b_j}\colon \Lpaths\rightarrow L^{p,1}_{\mathcal{P}_T}(H),\ X \mapsto  b_j((\cdot,\cdot),X)
  \end{equation*}
  \begin{equation*}
    F_{\sigma_j}\colon \Lpaths\rightarrow \overline \Lambda^{p,2,p}_{\mathcal{P}_T,S_j,\beta}(L(U,H)),\ X \mapsto \sigma_j((\cdot,\cdot),X)
  \end{equation*}
  \begin{equation*}
    S_j\convt{t_j} \#\colon L^{p,1}_{\mathcal{P}_T}(H)\rightarrow \Lpaths,\ X \mapsto 
\mathbf{1}_{[t_j,T]}(\cdot)\int_{t_j}^\cdot  (S_j)_{\cdot-s}X_sds
\end{equation*}
\begin{equation*}
  S_j\sconvt{t_j} \#\colon \overline \Lambda^{p,2,p}_{\mathcal{P}_T,S_j,\beta}(L(U,H))\rightarrow \Lpaths,\ \Phi \mapsto (S_j)\sconvt{t_j} \Phi.
\end{equation*}
\begin{equation*}\label{2016-04-19:00}
  \psi^{(j)}\colon \Lpaths\times \Lpaths
  \rightarrow \Lpaths,\ (Y,X) \mapsto \id^{S_j}_{t_j}(Y)+S_j\convt{t_j} F_{b_j}(X)+S_j\sconvt{t_j} F_{\sigma_j}(X).
\end{equation*}
\end{subequations}

In a similar way as done for $\psi$, we can obtain 
\eqref{eq:2016-03-25:02} for each $\psi^{(j)}$, with a constant $C'_{\lambda,g,\gamma,M',\beta,T,p,M}$ independent
of $j$.
In particular, there exists $\lambda_0$ large enough such that, 
for all $\lambda>\lambda_0$ and all $Y,X\in \Lpaths$,
\begin{equation}
\label{2016-04-13:04}
\begin{multlined}[c][.85\displaywidth]
  |\psi^{(j)}(Y,X)
  -
  \psi^{(j)}(Y',X')|_{\Lpaths,\lambda}\leq\\
  \leq M'|Y-Y'|_{\Lpaths,\lambda}+
  \frac{1}{2}|X-X'|_{\Lpaths,\lambda},\qquad \forall j\in \mathbb{N},
\end{multlined}
\end{equation}
 where
$$
\begin{minipage}{0.8\linewidth}
  \begin{center}
      $M'$ is any upper bound for ${\displaystyle \sup_{\substack{t\in[0,T]\\j\in \mathbb{N}}}|(S_j)_t|_{L(H)}}$.
    \end{center}
  \end{minipage}
$$
Let $A_j$ denotes the infinitesimal generator of $S_j$.
By arguing as done in the proof of Theorem \ref{2016-04-13:00},
we have that,
 for each $j\in \mathbb{N}$, there exists a unique mild solution $X^{t,Y}_j$ in $\Lpaths$  to
\begin{equation}\label{2016-04-13:01}
\begin{dcases}
  d (X_j)_s = \left(A_j(X_j)_s + b_j\left((\cdot,s),X_j\right)\right)  d t+ \sigma_j\left((\cdot,s),X_j\right)  d W_s & s\in(t_j, T]\\
(X_j)_s   = Y_s & s\in [0,t_j],
\end{dcases}
\end{equation}
and that,  due to the equivalence of the norms $|\cdot|_{\Lpaths,\lambda}$, the map $\Lpaths\rightarrow \Lpaths, \ Y \mapsto X^{t_j,Y}_j$ is Lipschitz, with Lipschitz constant bounded by some $C''_{g,\gamma,M,M',T,p}$ depending only on $g,\gamma,M,M',T,p$ and independent of $j$.


\medskip

For a given set $B\subset [0,T]$, let us denote
$$
\paths_B\coloneqq   \left\{ \mathbf{x}\in \paths\colon 
\forall t\in B,\ 
\mathbf{x}\ \mbox{is continuous in $t$} \right\}.
$$
Then $\paths_B$ is a closed subspace of $\paths$
 and it satisfies
all the three  conditions required for $\paths$ at
p.\ \pageref{2016-04-13:09}.
Moreover, if $t\in[0,T]$ and $Y\in \mathcal{L}^p_{\mathcal{P}_T}(\paths_B)$, then $X^{t,Y}\in \mathcal{L}^p_{\mathcal{P}_T}(\paths_B)$, because $X^{t,Y}$ is continuous
on $[t,T]$ (recall that $S\conv \#$ and $S\sconv \#$ are $\mathcal{L}^p_{\mathcal{P}_T}(\Cb{H})$-valued)
and
 coincides with $Y$ on $[0,t]$.


\begin{proposition}\label{2016-04-16:03}
  Suppose that Assumption~\ref{2016-03-24:08}
and
Assumption~\ref{2016-04-05:08} are satisfied
  and let $p> p^*$. 
Then
\begin{equation}
  \label{2016-04-13:08}
  \lim_{j\rightarrow \infty}X_j^{ t_j,Y}=X^{\hat t,Y}
\end{equation}
in $\mathcal{L}^p_{\mathcal{P}_T}(\paths_{\{\hat t\}})$,
 uniformly for $Y$ on compact subsets of $\mathcal{L}^p_{\mathcal{P}_T}(\paths_{\{\hat t\}})$.
\end{proposition}
\begin{proof}
Let $\psi^{(j)}$ be defined as above (p.\ \pageref{2016-04-19:00}).
  It is clear that, if $Y\in \mathcal{L}^p_{\mathcal{P}_T}(\paths_{\{\hat t\}})$ and $X\in \Lpaths$, then $\psi(Y,X)\in\mathcal{L}^p_{\mathcal{P}_T}(\paths_{\{\hat t\}})$, because it is continuous 
on $[\hat t,T]$ and
 coincides with $Y$ 
on $[0,\hat t]$.
Similarly, $\psi^{(j)}(Y,X)$ is continuous 
on $[t_j,T]$ and
coincides with $Y$ 
on $[0,t_j]$,
than also
 $\psi^{(j)}(Y,X)\in\mathcal{L}^p_{\mathcal{P}_T}(\paths_{\{\hat t\}})$.
Then, if the claimed convergence 
 occurs, it does
 in $\mathcal{L}^p_{\mathcal{P}_T}(\paths_{\{\hat t\}})$.

In order to prove the convergence, we consider the restrictions
$$
\begin{dcases}
  \hat \psi^{(j)}\coloneqq \psi^{(j)}_{|\mathcal{L}^p_{\mathcal{P}_T}(\paths_{\{\hat t\}})\times \Lpaths} &\forall j\in \mathbb{N}\\
\hat \psi\coloneqq \psi_{|\mathcal{L}^p_{\mathcal{P}_T}(\paths_{\{\hat t\}})\times \Lpaths},
\end{dcases}
$$
which are $\mathcal{L}^p_{\mathcal{P}_T}(\paths_{\{\hat t\}})$-valued, as noticed above.
Clearly \eqref{2016-04-13:04} still holds true with $\hat \psi^{(j)}$, $\hat \psi $ in place of $\psi^{(j)}$, $\psi$, respectively, and then
$$
\mathcal{L}^p_{\mathcal{P}_T}(\paths_{\{\hat t\}})\rightarrow \mathcal{L}^p_{\mathcal{P}_T}(\paths_{\{\hat t\}}),\ Y \mapsto X_j^{t_j,Y}
$$
is Lipschitz in $Y$, uniformly in $j$.
We then need only to prove the convergence 
$$
X_j^{t_j,Y}\rightarrow X^{\hat t,Y}
\mbox{\ in\ } \mathcal{L}^p_{\mathcal{P}_T}(\paths_{\{\hat t\}}),
\forall Y\in 
\mathcal{L}^p_{\mathcal{P}_T}(\paths_{\{\hat t\}}).
$$
Thanks to 
Lemma \ref{2016-02-25:02}\emph{(\ref{2016-02-23:02})},
the latter convergence
  reduces to the pointwise convergence
$$
\hat \psi^{(j)}\rightarrow \hat \psi.
$$
Let $Y\in \mathcal{L}^p_{\mathcal{P}_T}(\paths(\{\hat t\}))$. 
Due to the
continuity of $Y(\omega)$ in $\hat t$ for
$\mathbb{P}$-a.e.\ $\omega\in\Omega$,
 the strong continuity of $S_j$ and $S$,
and 
the strong  convergence $S_j\rightarrow S$,
 we have $\id^{S_j}_{t_j}(Y)\rightarrow \id^S_{\hat t}(Y)$ in $\mathcal{L}^p_{\mathcal{P}_T}(\paths_{\{\hat t\}})$
for all $Y\in \Lpaths$
(this can be seen by \eqref{eq:2016-04-21:01}).

We show that $S_j\sconvt{t_j}F_{\sigma_j}(X)\rightarrow S\sconvt{\hat t}F_\sigma(X)$, for all $X\in \Lpaths$.
Write
\begin{equation*}
      S_j\sconvt{t_j}F_{\sigma_j}- S\sconvt{\hat t}F_\sigma
=
(S_j\sconvt{t_j}F_{\sigma_j}- S\sconvt{ t_j}F_\sigma)  +
(S\sconvt{ t_j}F_\sigma- S\sconvt{\hat t}F_\sigma).
\end{equation*}
By Lebesgue's dominated convergence theorem and 
by Assumption~\ref{2016-04-05:08}, 
we have, for  $\beta\in (1/p,1/2-\gamma)$,
\begin{equation*}
  \lim_{j\rightarrow \infty}
\int_0^T
 \left( 
\int_0^t(t-s)^{-2\beta}
 \left( \mathbb{E} \left[ |(S_j)_{t-s}  \sigma_j((\cdot,s),X))-S_{t-s}\sigma((\cdot,s),X))|_{L_2(U,H)}^p  \right]  \right)^{2/p} ds
 \right) ^{p/2}dt=0
\end{equation*}
Then, by \eqref{eq:2016-04-06:01} (which holds uniformly in $t$), 
$$
S_j\sconvt{t_j}F_{\sigma_j}(X)- S\sconvt{ t_j}F_\sigma(X)\rightarrow 0\ \mbox{in}\ \Lpaths.
$$
By \eqref{eq:2016-04-16:01}, we also have
$$
S\sconvt{ t_j}F_{\sigma}(X)- S\sconvt{\hat t}F_{\sigma}(X)\rightarrow 0\ \mbox{in}\ \Lpaths.
$$
Then, we conclude
$$
      S_j\sconvt{t_j}F_{\sigma_j}- S\sconvt{\hat t}F_\sigma\rightarrow 0\ \mbox{in}\ \Lpaths.
$$
By arguing in a very similar way as done for $S_j\sconvt{t_j}F_{\sigma_j}-S\sconvt{\hat t}F_\sigma$, one can prove that
$$
\forall X\in \Lpaths,\ 
S_j\convt{t_j}F_{b_j}(X)
-
S\convt{\hat t}F_{b}(X)
\rightarrow
0\ \mbox{in}\ \Lpaths.
$$
Then $\hat \psi^{(j)}\rightarrow \hat \psi$ pointwise and the proof is complete.
\end{proof}


The following result provides continuity of the mild solution with respect to perturbations of all the data of the system.

\begin{theorem}\label{2016-04-22:09}
Suppose that Assumption~\ref{2016-03-24:08}
and Assumption~\ref{2016-04-05:08}
are satisfied, 
 let  $p> p^*$,
 $Y\in \mathcal{L}^p_{\mathcal{P}_T}(\paths_{\{\hat t\}})$,
 and let $\{Y_j\}_{j\in \mathbb{N}}\subset\Lpaths$ be a sequence converging to $Y$ in $\Lpaths$.
Then
$$
\lim_{j\rightarrow \infty}X_j^{t_j,Y_j}
=
X^{\hat t,Y}\ \mbox{in}\ 
\mathcal{L}^p_{\mathcal{P}_T}(\paths).
$$
\end{theorem}
\begin{proof}
Write 
\begin{equation}
  \label{eq:2016-04-21:16}
  X^{\hat t,Y}-X^{t_j,Y_j}_j
=
(X^{\hat t,Y}-X_j^{ t_j,Y})
+
(X^{t_j,Y}_j-X^{t_j,Y_j}_j),
\end{equation}
The  term
$X^{\hat t,Y}-X_j^{ t_j,Y}$
 tends to $0$ by 
 Proposition \ref{2016-04-16:03}, 
whereas the term
$X^{t_j,Y}_j-X^{t_j,Y_j}_j$
tends to $0$ by uniform equicontinuity of the family
$$
 \left\{ \Lpaths\rightarrow \Lpaths,\ Y \mapsto X^{t_j,Y}_j \right\} _{j\in \mathbb{N}}.
\eqno\qed
$$
\let\qed\relax
\end{proof}

We end this chapter with a result regarding stability of G\^ateaux differentials of mild solutions.

\begin{assumption}\label{2016-04-11:00}
Let $b,\sigma,g,\gamma,n,c,M''$
be as in
Assumption~\ref{2016-04-05:02},
and let $\{b_j\}_{j\in \mathbb{N}}$, $\{\sigma\}_{j\in \mathbb{N}}$, $\{S_j\}_{j\in \mathbb{N}}$, be as in Assumption~\ref{2016-04-05:08}.
Assume that
 \begin{enumerate}[(i)]
 \item 
for all $j\in \mathbb{N}$, $(\omega,t)\in \Omega_T$, and $u\in U$, $b_j((\omega,t),\cdot)\in \Gatot{\paths}{H}{n}$ and $\sigma_j((\omega,t),\cdot)u\in \Gatot{\paths}{H}{n}$;
\item 
for all $s\in [0,T]$,
  \begin{equation}\label{2016-04-11:01}
  \sup_{\substack{i=1,\ldots,n\\j\in \mathbb{N}}}
    \sup_{\substack{
        \omega\in \Omega\\
        \mathbf{x},\mathbf{y}_1,\ldots,\mathbf{y}_j\in \paths\\
           |\mathbf{y}_1|_\infty=\ldots =|\mathbf{y}_i|_\infty=1
      }
    }
       | \partial ^i_{\mathbf{y}_1\ldots \mathbf{y}_i}b_j((\omega,s),\mathbf{x})|_H
      \leq M'' g(s),
  \end{equation}
and,
for all
$s\in[0,T]$, $t\in(0,T]$,
 and all $ m\in \mathcal{M}$,
  \begin{equation}\label{2016-04-11:02}
    \sup_{\substack{i=1,\ldots,n\\j\in \mathbb{N}}}
    \sup_{
      \substack{
        \omega\in\Omega\\
        \mathbf{x},\mathbf{y}_1,\ldots,\mathbf{y}_i\in \paths\\
        |\mathbf{y}_1|_\infty=\ldots =|\mathbf{y}_i|_\infty=1
      }
      }
   |(S_j)_{t} \partial^i_{\mathbf{y}_1\ldots \mathbf{y}_i} (\sigma_j((\omega,s),\mathbf{x})e'_m))|_H
     \leq M'' t^{-\gamma}c_m;
   \end{equation}
\item\label{2016-04-21:14} 
for all 
 $X\in \paths$,
\begin{equation*}
\hskip-0.7cm \begin{dcases}
    \lim_{j\rightarrow \infty}
    |
    \partial ^i_{\mathbf{y}_1\ldots \mathbf{y}_i}b((\omega,t),\mathbf{x})-\partial ^i_{\mathbf{y}_1\ldots \mathbf{y}_i}b_j((\omega,t),\mathbf{x})|_H=0 & \hskip-5pt\forall (\omega,t)\in\Omega_T\\
    \lim_{j\rightarrow \infty}
|
    S_t \partial ^i_{\mathbf{y}_1\ldots \mathbf{y}_i}(\sigma((\omega,s),\mathbf{x})e'_m)-
(S_j)_t
\partial ^i_{\mathbf{y}_1\ldots \mathbf{y}_i}(\sigma_j((\omega,s),\mathbf{x})e'_m)|_H=0
    &
\hskip-5pt    \begin{dcases}
      \forall \omega\in\Omega,\\
      \forall s\in[0,T],\forall t\in (0,T],\\
       \forall m\in \mathcal{M}.
    \end{dcases}
      \end{dcases}
\end{equation*}
\end{enumerate}
\end{assumption}

\begin{theorem}\label{2016-04-22:10}
Suppose that Assumption~\ref{2016-03-24:08} and
Assumption~\ref{2016-04-05:08} are satisfied, and that, 
for some $n\in \mathbb{N}$, $n\geq 1$,
Assumption~\ref{2016-04-05:02} 
and Assumption~\ref{2016-04-11:00}
are satisfied.
Let $p>p^*$, $p\geq n$.
Then, for $i=1,\ldots,n$,
\begin{equation}
  \label{eq:2016-04-21:03}
   \partial ^i_{Y_1\ldots Y_i}X^{t_j,Y}_j
\rightarrow
 \partial ^i_{Y_1\ldots Y_i}X^{\hat t,Y}
\ \mbox{in}\ \mathcal{L}^p_{\mathcal{P}_T}(\paths_{\{\hat t\}}),
\end{equation}
uniformly for 
$Y,
Y_1,\ldots,Y_i$ in compact subsets of 
$\mathcal{L}^{p^n}_{\mathcal{P}_T}(\paths_{\{\hat t\}})$.
\end{theorem}
\begin{proof}
  By Theorem \ref{2016-04-05:05}, 
  $
  \mathcal{L}^{p^n}_{\mathcal{P}_T}(\paths)\rightarrow \Lpaths,\ Y \mapsto X_j^{t_j,Y}
  $
belongs to 
$\Gatot
{
  \mathcal{L}^{p^n}_{\mathcal{P}_T}(\paths)
}{\Lpaths}{n}$. Then, since $X^{t_j,Y}_j\in \mathcal{L}^{p}_{\mathcal{P}_T}(\paths_{\{\hat t\}})$ if $Y\in \mathcal{L}^{p}_{\mathcal{P}_T}(\paths_{\{\hat t\}})$,  the map
  $
  \mathcal{L}^{p^n}_{\mathcal{P}_T}(\paths_{\{\hat t\}})\rightarrow \mathcal{L}^p_{\mathcal{P}_T}(\paths_{\{\hat t\}}),\ Y \mapsto X_j^{t_j,Y}
  $ belongs to
$\Gatot
{
  \mathcal{L}^{p^n}_{\mathcal{P}_T}(\paths_{\{\hat t\}})
}{
\mathcal{L}^p_{\mathcal{P}_T}(\paths_{\{\hat t\}}
}{n}$.

To prove 
\eqref{eq:2016-04-21:03},
we wish to apply Proposition~\ref{propp:2012-05-23-aa}.
In the proof of Theorem~\ref{2016-04-05:05},
we associated
 the map $\psi$ and the spaces $\mathcal{L}^{p^k}_{\mathcal{P}_T}(\paths)$
to
Assumption~\ref{2016-02-24:00}.
In the same way, here, 
we  associate the restrictions
$$
\psi ^{(1)}_{|\mathcal{L}^{p^n}_{\mathcal{P}_T}(\paths_{\{\hat t\}})\times \mathcal{L}^p_{\mathcal{P}_T}(\paths_{\{\hat t\}})},
\psi ^{(2)}_{|\mathcal{L}^{p^n}_{\mathcal{P}_T}(\paths_{\{\hat t\}})\times \mathcal{L}^p_{\mathcal{P}_T}(\paths_{\{\hat t\}})},
\psi ^{(3)}_{|\mathcal{L}^{p^n}_{\mathcal{P}_T}(\paths_{\{\hat t\}})\times \mathcal{L}^p_{\mathcal{P}_T}(\paths_{\{\hat t\}})},
\ldots,
$$
 respectively
to the functions $h^{(1)}_1,h_1^{(2)},h_1^{(3)},\ldots$
appearing in the assumption of 
Proposition~\ref{propp:2012-05-23-aa},
and,
 to each $h_1^{(m)}$, we associate the functions
$h^{(m)}_k$, for $k=1,\ldots,n$,  
defined by
$
h^{(m)}_k\coloneqq 
\psi_{k|\mathcal{L}^{p^n}_{\mathcal{P}_T}(\paths_{\{\hat t\}})\times \mathcal{L}^{p^k}_{\mathcal{P}_T}(\paths_{\{\hat t\}})}
$
and considered as $\mathcal{L}^{p^k}_{\mathcal{P}_T}(\paths)$-valued functions.

As argued several times above,
we can choose
$\lambda>0$
such that,
for 
$m=1,2,\ldots$ and $k=1,\ldots,n$,
each function 
$h^{(m)}_k$ 
is a parametric $1/2$-contractions
with respect to the norm
$|\cdot|_{\mathcal{L}^{p^k}_{\mathcal{P}_T}(\paths),\lambda}$.
With respect to this equivalent norm,
for each $h_1^{(m)}$, Assumption~\ref{2016-02-24:00} can be verified in exactly the same way as it was verified for the function $h_1$ appearing in the proof of 
Theorem~\ref{2016-04-05:05}.
Then, in order to apply 
Proposition~\ref{propp:2012-05-23-aa}, it remains to 
verify
hypotheses
\emph{(\ref{2016-04-21:04})},\emph{(\ref{2016-04-21:05})},\emph{(\ref{2016-04-21:06})} appearing in the statement of that proposition.
Since the norms $|\cdot|_{\mathcal{L}^{p^k}_{\mathcal{P}_T}(\paths),\lambda}$, $\lambda\geq 0$, are equivalent, the three hypotheses
reduce to the following convergences:
\begin{enumerate}[(i)]
\item\label{2016-04-21:10} for all  $k=1,\ldots,n$,
$X\in \mathcal{L}^{p^k}_{\mathcal{P}_T}(\paths_{\{\hat t\}})$,
  \begin{equation}
    \label{eq:2016-04-21:07}
    \psi ^{(j)}(Y,X)\rightarrow \psi (Y,X)\ \mbox{in}\ 
(\mathcal{L}^{p^k}_{\mathcal{P}_T}(\paths_{\{\hat t\}}),
|\cdot|_{\mathcal{L}^{p^k}_{\mathcal{P}_T}(\paths)}
)
  \end{equation}
uniformly for $Y$ on compact subsets of
$\mathcal{L}^{p^n}_{\mathcal{P}_T}(\paths_{\{\hat t\}})$;
\item\label{2016-04-21:11} for $k=1,\ldots,n$
  \begin{equation}
    \label{eq:2016-04-21:08}
    \begin{dcases}
\lim_{j\rightarrow \infty}       \partial _{Y'}\psi^{(j)}(Y,X)=  \partial _{Y'}\psi(Y,X)&\mbox{in\ } 
(\mathcal{L}^{p^k}_{\mathcal{P}_T}(\paths_{\{\hat t\}}),
|\cdot|_{\mathcal{L}^{p^k}_{\mathcal{P}_T}(\paths)})
\\
\lim_{j\rightarrow \infty}       \partial _{X'}\psi^{(j)}(Y,X)=  \partial _{X'}\psi(Y,X)&\mbox{in\ } (\mathcal{L}^{p^k}_{\mathcal{P}_T}(\paths_{\{\hat t\}}),
|\cdot|_{\mathcal{L}^{p^k}_{\mathcal{P}_T}(\paths)})
    \end{dcases}
  \end{equation}
uniformly for $Y,Y'$ on compact subsets of 
$\mathcal{L}^{p^n}_{\mathcal{P}_T}(\paths_{\{\hat t\}})$
and $X,X'$ on compact subsets of
$\mathcal{L}^{p^k}_{\mathcal{P}_T}(\paths_{\{\hat t\}})$;
\item\label{2016-04-21:12} for all $k=1,\ldots,n-1$, $Y\in
\mathcal{L}^{p^n}_{\mathcal{P}_T}(\paths_{\{\hat t\}})$,
$l,i=0,\ldots,n$, $1\leq l+i\leq n$, 
\begin{equation}
  \label{eq:2016-04-21:09}
  \lim_{j\rightarrow \infty}
   \partial ^{l+i}   _{Y_1\ldots Y_l X_1\ldots X_i}
   \psi^{(j)}(Y,X)
   =
  \partial ^{l+i}   _{Y_1\ldots Y_l X_1\ldots X_i}
   \psi(Y,X)
   \ \mbox{in\ }
(\mathcal{L}^{p^k}_{\mathcal{P}_T}(\paths_{\{\hat t\}}),
|\cdot|_{\mathcal{L}^{p^k}_{\mathcal{P}_T}(\paths)})
\end{equation}
uniformly for $Y,Y_1,\ldots,Y_l$ on compact subsets of 
$\mathcal{L}^{p^n}_{\mathcal{P}_T}(\paths_{\{\hat t\}})$,
 $X$ on compact subsets of
$\mathcal{L}^{p^k}_{\mathcal{P}_T}(\paths_{\{\hat t\}})$,
$X_1,\ldots,X_i$ on compact subsets of
$\mathcal{L}^{p^{k+1}}_{\mathcal{P}_T}(\paths_{\{\hat t\}})$.
\end{enumerate}
Taking into account  the equicontinuity of the family $\{\psi^{(j)}\}_{j\in \mathbb{N}}$ with respect to the second variable,
 \eqref{2016-04-21:10} is contained in the proof 
Proposition~\ref{2016-04-16:03}.
As regarding \eqref{2016-04-21:11} and \eqref{2016-04-21:12},
since the linear term $\id^{S_j}_{t_j}$ is easily treated in $\mathcal{L}^p_{\mathcal{P}_T}(\paths_{\{\hat t\}})$ (as shown in the proof of Proposition~\ref{2016-04-16:03}),
 the only comments to make are about the convergences of the derivatives
$$
\begin{dcases}
   \partial_{Y'}(S_j\convt{t_j}F_{b_j})(X)\\
 \partial_{X'}(S_j\convt{t_j}F_{b_j})(X)\\
 \partial_{Y'}(S_j\sconvt{t_j}F_{\sigma_j})(X)\\
 \partial_{X'}(S_j\sconvt{t_j}F_{\sigma_j})(X)
\end{dcases}
\quad \mbox{and}\quad
\begin{dcases}
\partial ^{l+i}   _{Y_1\ldots Y_l X_1\ldots X_i}
(S_j\convt{t_j}F_{b_j})(X)\\
\partial ^{l+i}   _{Y_1\ldots Y_l X_1\ldots X_i}
(S_j\convt{t_j}F_{\sigma_j})(X).
\end{dcases}
$$
Due to linearity and continuity of the convolution operators, to the independence of the first variable of $F_b$ and $F_\sigma$, and to Lemma~\ref{2016-04-05:04}, the above
derivatives are 
 respectively
equal to
\begin{equation}
  \label{eq:2016-04-21:13}
  \begin{dcases}
  0\\
 S_j\convt{t_j}(\partial_{X'}F_{b_j})(X)\\
 0 \\ 
 S_j\sconvt{t_j}(\partial_{X'}F_{\sigma_j})(X)
\end{dcases}
\quad \mbox{and}\quad
\begin{dcases}
  \begin{dcases}
    S_j\convt{t_j}(
    \partial ^{i} _{X_1\ldots X_i}
    F_{b_j})(X)& \mbox{if\ }l=0\\
    0&\mbox{otherwise}
  \end{dcases}\\
  \begin{dcases}
    S_j\convt{t_j}(
    \partial ^{i} _{X_1\ldots X_i} F_{\sigma_j})(X)&\mbox{if\ }l=0\\
0&\mbox{otherwise.}
  \end{dcases}
\end{dcases}
\end{equation}
Let us consider, for example, the difference
\begin{equation}\label{2016-04-21:15}
S_j\convt{t_j}(
    \partial ^{i} _{X_1\ldots X_i} F_{\sigma_j})(X_j)
-S\convt{\hat t}(
    \partial ^{i} _{X_1\ldots X_i} F_{\sigma})(X)
\end{equation}
for some sequence $\{X_j\}_{j\in \mathbb{N}}$ converging to $X$ in $\mathcal{L}^{p^k}_{\mathcal{P}_T}(\paths)$.
We can decompose the above difference as done in \eqref{eq:2016-04-21:16}, and then use the same arguments, together with
expressions \eqref{2016-04-05:00},
the bounds
\eqref{2016-04-11:01} and \eqref{2016-04-11:01},
the generalized H\"older inequality, 
the pointwise  convergences
in Assumption~\ref{2016-04-11:00}\emph{(\ref{2016-04-21:14})}, and Lebesgue's dominated convergence theorem,
to conclude
$$
S_j\convt{t_j}(
    \partial ^{i} _{X_1\ldots X_i} F_{\sigma_j})(X_j)
-S\convt{\hat t}(
    \partial ^{i} _{X_1\ldots X_i} F_{\sigma})(X)\rightarrow 0
$$
in $\mathcal{L}^{p^k}_{\mathcal{P}_T}(\paths_{\{\hat t\}})$, for all $X_1,\ldots,X_i\in \mathcal{L}^{p^{k+1}}_{\mathcal{P}_T}(\paths_{\{\hat t\}})$.
By recalling the continuity of $X \mapsto \partial ^{i} _{X_1\ldots X_i} F_{\sigma}(X)$ (Lemma~\ref{2016-04-05:04}), this shows
 the convergence
\begin{equation}
  \label{eq:2016-04-21:17}
  S_j\convt{t_j}(
    \partial ^{i} _{X_1\ldots X_i} F_{\sigma_j})(X)
-S\convt{\hat t}(
    \partial ^{i} _{X_1\ldots X_i} F_{\sigma})(X)\rightarrow 0,
  \end{equation}
  uniformly for $X$ on compact sets of $\mathcal{L}^{p^k}_{\mathcal{P}_T}(\paths_{\{\hat t\}})$, for fixed 
$X_1,\ldots,X_i\in \mathcal{L}^{p^{k+1}}_{\mathcal{P}_T}(\paths_{\{\hat t\}})$.
But, since by Lemma~\ref{2016-04-05:04} the derivatives 
\eqref{eq:2016-04-21:13}
are jointly continuous in $X,X',X_1,\ldots,X_i$,
and uniformly bounded,
the convergence
\eqref{eq:2016-04-21:17} occurs  uniformly for 
$X$ on compact sets of $\mathcal{L}^{p^k}_{\mathcal{P}_T}(\paths_{\{\hat t\}})$ and
$X_1,\ldots,X_i$ on compact sets of $ \mathcal{L}^{p^{k+1}}_{\mathcal{P}_T}(\paths_{\{\hat t\}})$.
The arguments for the other derivatives are similar.
This shows that we can apply Proposition~\ref{propp:2012-05-23-aa}, which provides \eqref{eq:2016-04-21:03}.
\end{proof}

\addcontentsline{toc}{chapter}{References}
\bibliographystyle{plain}
\bibliography{rosestolato_pdsde_2018-06-20_arxiv.bbl}

\begin{thebibliography}{10}

\bibitem{Cerrai2001}
S.~Cerrai.
\newblock {\em Second-order PDE's in finite and infinite dimension}.
\newblock Springer, 2001.

\bibitem{Clark2013}
D.~E. Clark and J.~Houssineau.
\newblock Fa\`a di {B}runo's formula for chain differentials.
\newblock 2013.
\newblock Preprint, arXiv:1310.2833.

\bibitem{Cont2010a}
R.~Cont and D.-A. Fourni\'e.
\newblock Change of variable formulas for non-anticipative functionals on path
  space.
\newblock {\em Journal of Functional Analysis}, 259:1043--1072, 2010.

\bibitem{Cont2010}
R.~Cont and D.-A. Fourni\'e.
\newblock A functional extension of the {I}t\^o formula.
\newblock {\em C.\ R.\ Math.\ Acad.\ Sci.\ Paris, Ser.\ I}, 348:57--61, 2010.

\bibitem{Cont2013}
R.~Cont and D.-A. Fourni\'e.
\newblock Functional {I}t\^o calculus and stochastic integral representation of
  martingales.
\newblock {\em The Annals of Probability}, 41:109--133, 2013.

\bibitem{Cosso2014a}
A.~Cosso, C.~Di~Girolami, and F.~Russo.
\newblock Calculus via regularizations in {B}anach spaces and {K}olmogorov-type
  path-dependent equations.
\newblock 2014.
\newblock Preprint, arXiv:1411.8000.

\bibitem{Cosso2014b}
A.~Cosso and F.~Russo.
\newblock A regularization approach to functional {I}t\^o calculus and
  strong-viscosity solutions to path-dependent {PDEs}.
\newblock 2014.
\newblock Preprint, arXiv:1401.5034.

\bibitem{DaPrato1992}
G.~Da~Prato and J.~Zabczyck.
\newblock {\em Stochastic Equations in Infinite Dimensions}.
\newblock Cambridge University Press, 1992.

\bibitem{DaPrato2004}
G.~Da~Prato and J.~Zabczyck.
\newblock {\em Second Order Partial Differential Equations in Hilbert Spaces}.
\newblock Cambridge University Press, 2002.

\bibitem{DaPrato2014}
G.~Da~Prato and J.~Zabczyck.
\newblock {\em Stochastic Equations in Infinite Dimensions}.
\newblock Cambridge University Press, $2^{\textrm{nd}}$ edition, 2014.

\bibitem{Dupire2009}
B.~Dupire.
\newblock Functional {I}t\^o calculus.
\newblock {\em Bloomberg Portfolio Research Paper}, 2009.

\bibitem{Flett1980}
T.M. Flett.
\newblock {\em Differential Analysis}.
\newblock Cambridge University Press, 1980.

\bibitem{Gawarecki2011}
L.~Gawarecki and V.~Mandrekar.
\newblock {\em Stochastic Differential Equations in Infinite Dimensions}.
\newblock Springer, 2011.

\bibitem{Granas2003}
A.~Granas and J.~Dugundji.
\newblock {\em Fixed Point Theory}.
\newblock Springer, 2003.

\bibitem{Knoche2001}
C.~Knoche and K.~Frieler.
\newblock Solutions of stochastic differential equations in infinite
  dimensional {H}ilbert spaces and their dependence on initial data, 2001.
\newblock Diplomarbeit, Fakult\"at f\"ur Mathematik, Universit\"at Bielefeld.

\bibitem{Mohammed1984}
S.-E.A. Mohammed.
\newblock {\em Stochastic Functional Differential Equations}.
\newblock Pitman, 1984.

\end{thebibliography}

\end{document}